\documentclass{amsart}

\usepackage{amsmath}
\usepackage{amssymb, amscd, color, graphicx}
\usepackage{pst-all}
\usepackage[noend]{algpseudocode}
\usepackage{caption}

\newcommand\STATE[1]{\State{#1}}
\newcommand\COMMENT[1]{\Comment{#1}}
\newcommand\REQUIRE[1]{\Require{#1}}
\newcommand\FOR[1]{\For{#1}}
\newcommand\ENDFOR{\EndFor}
\newcommand\WHILE[1]{\While{#1}}
\newcommand\ENDWHILE{\EndWhile}
\newcommand\REPEAT{\Repeat}
\newcommand\UNTIL[1]{\Until{#1}}
\newcommand\IF[1]{\If{#1}}
\newcommand\ELSEIF[1]{\ElsIf{#1}}
\newcommand\ELSE{\Else}
\newcommand\ENDIF{\EndIf}
\newcommand\CONTEXT[1]{\renewcommand\algorithmicrequire{\textbf{Context:}}\REQUIRE{#1}}
\newcommand\INPUT[1]{\renewcommand\algorithmicrequire{\textbf{Input:}}\REQUIRE{#1}}
\newcommand\OUTPUT[1]{\renewcommand\algorithmicrequire{\textbf{Output:}}\REQUIRE{#1}}
\newcommand\RETURN[1]{\STATE{\textbf{return} #1}}
\newcommand\PROCEDURE[2]{\Procedure{#1}{#2}}
\newcommand\ENDPROCEDURE{\EndProcedure}
\newcommand\FUNCTION[2]{\Function{#1}{#2}}
\newcommand\ENDFUNCTION{\EndFunction}
\newcommand\CALL[2]{\Call{#1}{#2}}

\newcommand\BOX[1]{\begin{minipage}{125mm}{#1}\end{minipage}}

\numberwithin{equation}{section}

\theoremstyle{plain}
\newtheorem{prop}[equation]{Proposition}
\newtheorem{coro}[equation]{Corollary}
\newtheorem{lemm}[equation]{Lemma}

\theoremstyle{definition}
\newtheorem{defi}[equation]{Definition}
\newtheorem{exam}[equation]{Example}
\newtheorem{nota}[equation]{Notation}
\newtheorem{rema}[equation]{Remark}

\newtheoremstyle{algorithm}
{10pt}
{10pt}
{\mdseries}
{}
{\bfseries}
{.}
{5pt}
{\thmname{#1}\thmnumber{ #2}\ (\bfseries{#3})}
\theoremstyle{algorithm}
\newtheorem{myalgo}[equation]{Algorithm}
\newenvironment{algo}[1]{\begin{myalgo}[#1]\hspace*{0pt}\vspace{2pt}\par}{\end{myalgo}}

\allowdisplaybreaks

\psset{unit=1mm}
\psset{fillcolor=white}
\psset{dotsep=1.5pt}
\psset{dash=1.5pt 1.5pt}
\psset{linewidth=0.8pt}
\psset{arrowsize=3.5pt}
\psset{doublesep=1pt}
\psset{coilwidth=1.2mm}
\psset{coilheight=2}
\psset{coilaspect=20}
\psset{coilarm=2}
\newpsstyle{double}{linewidth=0.5pt, doubleline=true}
\newpsstyle{etc}{linestyle=dotted}
\newpsstyle{exist}{linestyle=dashed}
\newpsstyle{thick}{linewidth=1.5pt}
\newpsstyle{thin}{linewidth=0.5pt}
\newpsstyle{thinexist}{linewidth=0.5pt, linestyle=dashed}

\newcounter{ITEM}
\newcommand\ITEM[1]{\setcounter{ITEM}{#1}\leavevmode\hbox{\rm(\roman{ITEM})}}

\renewcommand\aa{a}
\newcommand\AAA{\mathcal{A}}

\newcommand\AAAh{\widehat\AAA}

\newcommand\bb{b}
\newcommand\BBB{\mathcal{B}}
\newcommand\BBBh{\widehat\BBB}
\newcommand\BP[1]{B^{\scriptscriptstyle+}_{#1}}

\newcommand\Cat{\mathcal{C}\hspace{-0.3ex}a\hspace{-0.1ex}t}
\newcommand\cc{c}
\newcommand\CC{C}
\newcommand\CCC{\mathcal{C}}
\newcommand\CCCi{\CCC^{\!\scriptscriptstyle\times}}

\newcommand\cl[1]{[#1]}

\newcommand\DELTA[1]{\Delta^{\![#1]}}
\newcommand\DELTAA[1]{\Delta^{\!(#1)}}
\newcommand\der{\partial}
\newcommand\derL{\widetilde\partial}
\renewcommand\div{\mathrel{\prec}}

\newcommand\dive{\mathrel{\preccurlyeq\nobreak}}

\newcommand\diveS{\mathrel{\preccurlyeq_{\SSS}}}
\newcommand\divesmall{\scriptstyle\dive}

\newcommand\divS{\mathrel{\prec_{\SSS}}}

\newcommand\ee{e}
\newcommand\EE{E}
\newcommand\Env[1]{\mathcal{E}\!nv(#1)}
\newcommand\eqir{\mathrel{=^{\!\scriptscriptstyle\times}}}
\newcommand\equivp{\mathrel{\equiv^{\scriptscriptstyle+}}}
\newcommand\etc{\pdots}
\newcommand\ew{\varepsilon}

\newcommand\ff{f}

\renewcommand\gcd{\mathbin{\wedge}}
\renewcommand\ge{\mathrel{\geqslant}}
\renewcommand\gg{g}

\newcommand\GGG{\mathcal{G}}

\newcommand\hh{h}

\newcommand\HS[1]{\hspace{#1ex}}

\newcommand\id[1]{1_{#1}}
\newcommand\Id[1]{\boldsymbol{1}_{\!#1}}
\newcommand\ii{i}
\newcommand\inv{^{-1}}
\newcommand\INV[1]{\overline{\VR(2,0)#1}}
\newcommand\isdef{{\downarrow}}

\newcommand\jj{j}

\newcommand\JJJ{\mathcal{J}}

\newcommand\LC{\widetilde\RC}
\newcommand\lcm{\mathbin{\vee}}
\renewcommand\le{\mathrel{\leqslant}}
\newcommand\LG[1]{\Vert#1\Vert}

\newcommand\MM{M}

\newcommand\nn{n}

\newcommand\NNNN{\mathbb{N}}

\newcommand\Obj{\mathcal{O}\HS{-0.15}b\HS{-0.25}j}
\newcommand\OP{\,{\scriptstyle\bullet}\,}

\newcommand\pc{\HS{0.05}\vert\HS{0.05}}
\newcommand\pdots{\HS{0.2}{\cdot}{\cdot}{\cdot}\HS{0.2}}
\newcommand\pp{p}

\newcommand\PRES[2]{\langle#1\,\vert\, #2\rangle}
\newcommand\PRESp[2]{\langle#1\,\vert\, #2\rangle^{\scriptscriptstyle\!+}\!}

\newcommand\qq{q}

\newcommand\RC{\theta}
\newcommand\RCd{\RC'}
\newcommand\RCh{\widehat\RC}
\newcommand\resp{\mbox{\it resp}.\ }
\newcommand\rev{\curvearrowright}
\newcommand{\revLinv}{\mathrel{\raisebox{5pt}{\rotatebox{180}{$\curvearrowright$}}}}
\newcommand\ROP{\mathcal{R}_{\bullet}}
\newcommand\rr{r}
\newcommand\RRR{\mathcal{R}}

\newcommand\Seq[2]{#1^{[#2]}}

\newcommand\seqq[2]{#1\pc#2}
\newcommand\seqqq[3]{#1\pc#2\pc#3}
\newcommand\seqqqq[4]{#1\pc#2\pc#3\pc#4}
\newcommand\seqqqqq[5]{#1\pc#2\pc#3\pc#4\pc#5}
\newcommand\seqqqqqq[6]{#1\pc#2\pc#3\pc#4\pc#5\pc#6}
\newcommand\seqqqqqqq[7]{#1\pc#2\pc#3\pc#4\pc#5\pc#6\pc#7}
\newcommand\seqqqqqqqq[8]{#1\pc#2\pc#3\pc#4\pc#5\pc#6\pc#7\pc#8}
\newcommand\Square{\square}
\newcommand\src[1]{\mathrm{src}(#1)}
\renewcommand{\SS}{S}
\newcommand\SSS{\mathcal{S}}
\newcommand\SSSg{\underline\SSS}

\newcommand\SSSs{\SSS^{\scriptstyle\sharp}}

\newcommand\SW{W}

\newcommand\trg[1]{\mathrm{trg}(#1)}
\newcommand\tta{\mathtt{a}}
\newcommand\ttb{\mathtt{b}}

\newcommand\TV[1]{[\![#1]\!]}

\newcommand\ud{\hbox{-}}
\newcommand\under{\,\backslash\,}
\newcommand\uu{u}

\def\VR(#1,#2){\vrule width0pt height#1mm depth#2mm}
\newcommand\vv{v}

\newcommand\wdots{, ...\HS{0.2},}
\newcommand\wit{\lambda}
\newcommand\ww{w}

\newcommand\xx{x}

\newcommand\yy{y}

\newcommand\zz{z}
\newcommand\ZZZZ{\mathbb{Z}}

\title{Algorithms for Garside calculus}

\author{Patrick DEHORNOY}
\address{Laboratoire de Math\'ematiques Nicolas Oresme,
CNRS UMR 6139, Universit\'e de Caen, 14032 Caen, France}
\email{patrick.dehornoy@unicaen.fr}
\urladdr{www.math.unicaen.fr/\!\hbox{$\sim$}dehornoy}
\thanks{Both authors acknowledge support under Australian Research Council's Discovery Projects funding scheme (project number DP1094072). Volker Gebhardt acknowledges support under the Spanish Project MTM2010-19355.}

\author{Volker GEBHARDT}
\address{School of Computing, Engineering, and Mathematics, University of Western Sydney, Locked bag 1797, Penrith, NSW 2751, Australia}
\email{v.gebhardt@uws.edu.au}



\begin{document}

\begin{abstract}
Garside calculus is the common mechanism that underlies a certain type of normal form for the elements of a monoid, a group, or a category. Originating from Garside's approach to Artin's braid groups, it has been extended to more and more general contexts, the latest one being that of categories and what are called Garside families. One of the benefits of this theory is to lead to algorithms solving effectively the naturally occurring problems, typically the Word Problem. The aim of this paper is to present and solve these algorithmic questions in the new extended framework.
\end{abstract}

\begin{center}
\tiny Version of \today
\end{center}

\maketitle

In 1969, F.A.\,Garside~\cite{Gar} solved the Word and Conjugacy Problems in Artin's braid group~$B_\nn$ \cite{Art} by describing the latter as a group of fractions and analyzing the involved monoid in terms of its divisibility relation. This approach was continued and extended in several steps, first to Artin-Tits groups of spherical type~\cite{BrS, Dlg, Adj, Thu, Eps, ElM, Cha}, then to a larger family of groups now known as Garside groups~\cite{Dfx, Dgk}. More recently, it was realized that going to a categorical context allows for capturing further examples~\cite{DiM}, and a coherent theory has recently emerged with a central unifying notion called Garside families~\cite{Dif, Garside}: the central notion is a certain way of decomposing the elements of the reference category or its groupoid of fractions and a Garside family is what makes the construction possible.

What we do in this paper is to present and analyse the main algorithms arising in this new, extended context of Garside families, with two main directions, namely recognizing that a candidate family is a Garside family and using a Garside family to compute in the category, typically finding distinguished decompositions and solving the Word Problem along the lines of~\cite{Dhp}. This results in a corpus of about twenty algorithms that are proved to be correct, analysed, and given examples. We do not address the Conjugacy Problem here, as extending the methods of~\cite{Geb, GeG} will require further developments that we keep for a subsequent work.

The paper consists of five sections. Section~\ref{S:Context} is a review of Garside families and the derived notions involved in the approach, together with some basic results that appear in other sources. Next, we address the question of effectively recognizing Garside families and we describe and analyse algorithms doing it: in Section~\ref{S:Pres}, we consider the case when the ambient category is specified using a presentation (of a certain type), whereas, in Section~\ref{S:Germ}, we consider the alternative approach when the category is specified using what is called a germ. Finally, the last two sections are devoted to those computations that can be developed once a Garside family is given. In Section~\ref{S:Pos}, we consider computations taking place in the reference category or monoid (``positive case''), whereas, in Section~\ref{S:Sym}, we address similar questions in the groupoid or group of fractions of the reference category (``signed case'').

\section{The general context}
\label{S:Context}

In this introductory section, we present the background of categories and Garside families, together with some general existence and uniqueness results that will be used and, often, refined in the sequel of the paper. Proofs appear in other sources and they will be omitted in general.

\subsection{Categories}
\label{SS:Cat}

The general context is that of categories, which should be seen here just as monoids with a partial product, that is, one that is not necessarily defined everywhere. A \emph{precategory} (or multigraph) is a family $\AAA$ plus two maps, ``source'' and ``target'', of~$\AAA$ to another family~$\Obj(\AAA)$ (the objects of~$\AAA$), and a \emph{category} is a precategory equipped with a partial multiplication such that $\ff\gg$ exists if and only if the target of~$\ff$, denoted~$\trg\ff$, coincides with the source of~$\gg$, denoted~$\src\gg$. The multiplication is associative whenever defined and, in addition, has a neutral element~$\id\xx$ for each object~$\xx$, that is, $\id\xx \gg = \gg = \gg\id\yy$ holds for every~$\gg$ with source~$\xx$ and target~$\yy$. If $\CCC$ is a category, the family of all neutral elements is denoted by $\Id\CCC$ and, for $\AAA$ included in~$\CCC$ and $\xx, \yy$ in~$\Obj(\CCC)$, the family of elements of $\AAA$ with source~$\xx$ and target~$\yy$ is denoted by $\AAA(\xx, \yy)$. A monoid is the special case of a category when there is only one object, so that the product is always defined. It is convenient to represent the elements of a category using arrows, so that an element~$\gg$ with source~$\xx$ and target~$\yy$ is represented by 
\begin{picture}(13,4)
\put(0,0){$\xx$}
\put(11,0){$\yy$}
\pcline{->}(3,0.5)(10,0.5)
\taput{$\gg$}
\end{picture}.

\subsubsection*{Free categories and paths}

If $\AAA$ is a precategory, the free category generated by~$\AAA$ is the family~$\AAA^*$ of all $\AAA$-paths, that is, all finite sequences~$(\aa_1 \wdots \aa_\pp)$ of elements of~$\AAA$ such that the target of~$\aa_{\ii-1}$ is the source of~$\aa_\ii$ for every~$\ii$, together with, for each object~$\xx$, an empty path~$\ew_\xx$, and equipped with concatenation of paths. When~$\ww$ is an $\AAA$-path, we denote by~$\LG\ww$ the length of~$\ww$, and, for $1 \le \ii \le \LG\ww$, we denote by~$\ww[\ii]$ the $\ii$-th entry in~$\ww$. For $\BBB$ included in~$\AAA$, we denote the family of all $\BBB$-paths of length~$\pp$ by $\Seq{\BBB}\pp$. If $\ww_1, \ww_2$ are two paths, we denote by~$\ww_1 \pc \ww_2$ the concatenation of~$\ww_1$ and~$\ww_2$ when it exists, that is, when the target of~$\ww_1$ (defined to be the target of the last entry in~$\ww_1$) coincides with the source of~$\ww_2$ (defined to be the source of the first entry in~$\ww_2$). We identify a length one path with its unique entry. Then a length~$\pp$ path $(\aa_1 \wdots \aa_\pp)$ is the concatenation of the length one paths made of its successive entries, so that it can be denoted by $\seqqq{\aa_1}\etc{\aa_\pp}$. When $\AAA$ is a set, that is, a precategory with one object only, the condition about source and targets vanishes, and it is usual to say \emph{$\AAA$-word} for~$\AAA$-path.

\subsubsection*{Presentations}

Every category that is generated by a family~$\AAA$ is a quotient of the free category~$\AAA^*$, and, for $\RRR$ a family of pairs of $\AAA$-paths, it is said to admit the presentation $(\AAA ; \RRR)$ if it is isomorphic to~$\AAA^*\!{/}{\equivp_\RRR}$ where $\equivp_\RRR$ is the congruence on~$\AAA^*$ generated by~$\RRR$. The elements of~$\RRR$ are called \emph{relations}, and, in this context, it is customary to write $\uu = \vv$ instead of $(\uu, \vv)$ for a relation and, in concrete examples, to omit the concatenation sign, thus writing $\aa_1 \pdots \aa_\pp = \bb_1 \pdots \bb_\qq $ rather than $\seqqq{\aa_1}\etc{\aa_\pp} = \seqqq{\bb_1}\etc{\bb_\qq}$. If $(\AAA ; \RRR)$ is a presentation, we write $\PRESp\AAA\RRR$ for the category presented by~$(\AAA ; \RRR)$---which is determined only up to isomorphism---and, for $\ww$ an $\AAA$-path, we write~$\cl\ww$ for the $\equivp_\RRR$-class of~$\ww$, that is, for the element of $\PRESp\AAA\RRR$ represented by~$\ww$---which is also the evaluation of the path~$\ww$ in the category~$\PRESp\AAA\RRR$. 

\subsubsection*{Cancellativity}

All categories we shall consider here will have to satisfy some cancellability condition, at least on one side.

\begin{defi}
A category~$\CCC$ is called \emph{left-cancellative} (\resp \emph{right-cancellative}) if $\ff\gg = \ff\gg'$ (\resp $\gg\ff = \gg'\ff$) implies $\gg = \gg'$ for all~$\ff, \gg, \gg'$ in~$\CCC$.
\end{defi}

A category is called \emph{cancellative} if it is both left- and right-cancellative. In a left- (or right-) cancellative category, an element has a left-inverse if and only if it has a right-inverse, and so there is a unique, non-ambiguous notion of \emph{invertible} element. For~$\CCC$ a left-cancellative category, we denote by~$\CCCi$ the subgroupoid of~$\CCC$ consisting of all invertible elements. An invertible element will be called \emph{nontrivial} if it is not an identity-element~$\id\xx$.

\subsubsection*{Divisibility, lcm, lcm-selector}

Associated with every (left-cancellative) category---hence, in particular, every monoid---comes a natural left-divisibility relation.

\begin{defi}
Assume that $\CCC$ is a left-cancellative category. For $\ff, \gg$ in~$\CCC$, we say that $\ff$ is a \emph{left-divisor} of~$\gg$, or, equivalently, that $\gg$ is a \emph{right-multiple} of~$\ff$, denoted $\ff \dive \gg$, if there exists~$\gg'$ in~$\CCC$ so that $\ff \gg' = \gg$ holds.
\end{defi}

The hypothesis that the ambient category is left-cancellative is needed to guarantee that left-divisibility is a partial preordering; the associated equivalence relation is right-multiplication by an invertible element: the conjunction of $\ff \dive \gg$ and $\gg \dive \ff$ is equivalent to the existence of an invertible element~$\ee$ satisfying $\ff\ee = \gg$, which will be denoted by~$\ff \eqir \gg$ hereafter.

Simple derived notions stem from the left-divisibility relation, corresponding to greatest common lower bound and least common upper bound. We say that $\hh$ is a \emph{greatest common left-divisor}, or \emph{left-gcd} of~$\ff$ and~$\gg$ if $\hh$ left-divides~$\ff$ and~$\gg$ and every left-divisor of~$\ff$ and~$\gg$ left-divides~$\hh$. Symmetrically, we say $\hh$ is a \emph{least common right-multiple}, or \emph{right-lcm} of~$\ff$ and~$\gg$ if $\hh$ is a right-multiple of~$\ff$ and~$\gg$ and every right-multiple of~$\ff$ and~$\gg$ is a right-multiple of~$\hh$. Right-lcms and left-gcds, if they exist, are unique up to right-multiplication by an invertible element, hence unique if the ambient category has no nontrivial invertible element. In the latter case, we write $\ff \lcm \gg$ and $\ff \gcd \gg$ for the right-lcm and the left-gcd of~$\ff$ and~$\gg$ (if they exist), and $\ff \under\gg$ for the unique element that satisfies $\ff \lcm \gg = \ff(\ff \under \gg)$. 

\begin{defi}
We say that a left-cancellative category~$\CCC$ \emph{admits right-lcms} (\resp \emph{admits local right-lcms}) if any two elements of~$\CCC$ with the same source (\resp if any two elements of~$\CCC$ that admit a common right-multiple) admit a right-lcm.
\end{defi}

Finally, we shall sometimes need to choose right-lcms explicitly. The following terminology is then natural.

\begin{defi}
\label{D:Selector}
Assume that $\CCC$ is a left-cancellative category, and $\AAA$ is a generating subfamily of~$\CCC$. A \emph{right-lcm selector} on~$\AAA$ is a partial map~$\RC : \AAA \times \AAA \to \AAA^*$ such that, for all~$\aa, \bb$ in~$\AAA$, the elements~$\RC(\aa, \bb)$ and~$\RC(\bb, \aa)$ are defined if and only if~$\aa$ and~$\bb$ admit a right-lcm and, in this case, there exists a right-lcm~$\hh$ of~$\aa$ and~$\bb$ such that both $\aa \pc \RC(\aa, \bb)$ and $\bb \pc \RC(\bb, \aa)$ represent~$\hh$.
\end{defi}

Note that, if $\CCC$ is a left-cancellative category that admits no nontrivial invertible element, then the map $(\ff, \gg) \mapsto \ff \under\gg$ is a right-lcm selector on~$\CCC$.

The notions of a right-divisor and left-multiple are defined symmetrically, and so are the derived notions of a right-gcd, and a left-lcm selector.

\subsubsection*{Noetherianity conditions}

\begin{defi}
A left-cancellative category~$\CCC$ is called \emph{right-Noetherian} if every bounded $\div$-increasing sequence in~$\CCC$ is finite.
\end{defi}

By extension, we say that a (positive) presentation is right-Noetherian if the associated category is right-Noetherian. Standard results (see for instance \cite{Lev}) give the following criterion for establishing Noetherianity conditions.

\begin{lemm}
\label{L:Noeth}
A presentation $(\AAA ; \RRR)$ is right-Noe\-ther\-ian if and only if there exists a map $\wit$ of~$\AAA^*$ to the ordinals that is $\equivp_\RRR$-invariant and satisfies $\wit(\aa) > 0$ for every~$\aa$ in~$\AAA$ and $\wit(\aa \pc \ww) > \wit(\ww)$ for all~$\ww$ in~$\AAA^*$ and~$\aa$ in~$\AAA$.
\end{lemm}

Note that every presentation~$(\AAA ; \RRR)$ such that $\RRR$ consists of relations of the form~$\uu = \vv$ with $\uu, \vv$ of the same length is right-Noetherian, as one can then define $\wit(\ww)$ to be the length of~$\ww$.

\subsubsection*{Ore category, groupoid of fractions}

Finally, we recall that a groupoid is a category in which every element is invertible, a group corresponding to the special case when there is only one object, that is, the product is always defined. 

We say that a groupoid~$\GGG$ is a \emph{groupoid of left-fractions} for a subcategory~$\CCC$ if every element of~$\GGG$ admits an expression of the form~$\ff\inv \gg$ with $\ff, \gg$ in~$\CCC$. The following result of Ore is classical:

\begin{prop}\cite{ClP}
\label{P:Ore}
Say that a category is \emph{left-Ore} (\resp \emph{right-Ore}) if it is cancellative and any two elements that have the same target (\resp source) admit a common left-multiple (\resp right-multiple). Then a category embeds in a groupoid of left-fractions (\resp of right fractions) if and only if it is left-Ore (\resp right-Ore).
\end{prop}

A category that is both left- and right-Ore will be called an \emph{Ore category}.

\subsection{Normal decompositions and Garside families}
\label{SS:Gar}

The central idea in our approach consists in introducing distinguished decompositions of a certain type for the elements of the considered category. These decompositions involve a reference subfamily of the ambient category and correspond to the principle of recursively selecting maximal left-divisors lying in the reference family.

\begin{defi}
\label{D:Greedy}
Assume that $\CCC$ is a left-cancellative category. For $\SSS$ included in~$\CCC$, a $\CCC$-path $\seqqq{\gg_1}\etc{\gg_\pp}$ is called \emph{$\SSS$-greedy} (\resp \emph{$\SSS$-normal}) if, for every~$\ii < \pp$, we have
\begin{equation}
\label{E:Greedy}
\forall\aa{\in}\SSS \ \forall\ff{\in}\CCC \ (\aa \dive \ff\gg_\ii \gg_{\ii+1} \Rightarrow \aa \dive \ff\gg_\ii)
\end{equation}
(\resp this and, in addition, every entry~$\aa_\ii$ lies in~$\SSSs$, defined to be $\SSS\CCCi \cup \CCCi$).
\end{defi}

When using diagrams in which the elements of the category are represented by arrows, we shall indicate that a path $\seqq{\gg_1}{\gg_2}$ is $\SSS$-greedy by appending a small arc as in 
\begin{picture}(31,4)(0,0)
\pcline{->}(1,0)(14,0)
\taput{$\gg_1$}
\pcline{->}(16,0)(29,0)
\taput{$\gg_2$}
\psarc[style=thin](14.5,0){3}{0}{180}
\end{picture}.

\begin{exam}\label{X:FreeAbelianMonoid}
Consider the free Abelian monoid $M$ defined by the presentation $(\tta,\ttb;\tta\ttb=\ttb\tta)$ and let $\SSS=\{\tta,\ttb,\tta\ttb\}$, whence $\SSSs=\{1,\tta,\ttb,\tta\ttb\}$. The $\SSS$-normal paths in $M$ are precisely the paths $\seqqq{\gg_1}{\etc}{\gg_\qq}$ that, for some $\pp,\pp'$ in~$\{0 \wdots \qq\}$, satisfy \ITEM1 $\gg_\ii=\tta\ttb$ for $\ii=1 \wdots \pp$;
\ITEM2 either $\gg_\ii=\tta$ for $\ii=\pp+1 \wdots \pp'$, or $\gg_\ii=\ttb$ for $\ii=\pp+1 \wdots \pp'$; and
\ITEM3 $\gg_\ii=1$ for $\ii=\pp'+1 \wdots \qq$.
\end{exam}

If $\gg$ is an element of a category~$\CCC$, a path $\ww$ satisfying $\cl\ww = \gg$ is called a \emph{decomposition} of~$\gg$. What we shall be interested in in the sequel are the (possible) $\SSS$-normal decompositions of the elements of the considered category. 
One of the interests of such decompositions is that they are essentially unique.

\begin{prop}
\label{P:NormalUnique}\cite[Proposition~2.11]{Dif}
Assume that $\CCC$ is a left-cancellative category and $\SSS$ is included in~$\CCC$. Then any two $\SSS$-normal decompositions of an element of~$\CCC$ (if any) are $\CCCi$-deformations of one another, where $\seqqq{\aa_1}\etc{\aa_\pp}$ is said to be a \emph{$\CCCi$-deformation} of $\seqqq{\bb_1}\etc{\bb_\qq}$ if there exist invertible elements~$\ee_0 \wdots \ee_{\rr}$, with $\rr = \max(\pp, \qq)$, such that $\ee_0, \ee_\rr$ are identity-elements and $\bb_\ii \ee_\ii = \ee_{\ii-1} \aa_\ii$ holds for $1 \le \ii \le \rr$, where, for $\pp \not=\qq$, the shorter path is expanded by identity-elements (see Figure~\ref{F:NFDeformation}). 
\end{prop}

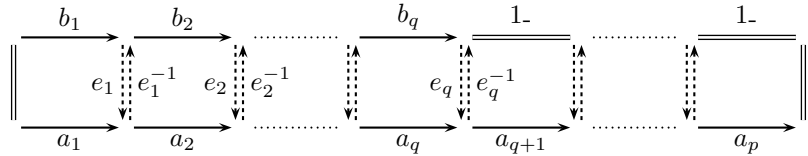
\begin{figure}[htb]
\begin{picture}(105,14)(0,-1)
\pcline{->}(1,0)(14,0)
\tbput{$\aa_1$}
\pcline{->}(1,12)(14,12)
\taput{$\bb_1$}
\pcline{->}(16,0)(29,0)
\tbput{$\aa_2$}
\pcline{->}(16,12)(29,12)
\taput{$\bb_2$}
\pcline[style=etc](32,0)(43,0)
\pcline[style=etc](32,12)(43,12)
\pcline{->}(46,0)(59,0)
\tbput{$\aa_\qq$}
\pcline{->}(46,12)(59,12)
\taput{$\bb_\qq$}
\pcline{->}(61,0)(74,0)
\tbput{$\aa_{\qq+1}$}
\pcline[style=double](61,12)(74,12)
\taput{$\id\ud$}
\pcline[style=etc](77,0)(88,0)
\pcline[style=etc](77,12)(88,12)
\pcline{->}(91,0)(104,0)
\tbput{$\aa_\pp$}
\pcline[style=double](91,12)(104,12)
\taput{$\id\ud$}
\psline[style=double,style=thin](0,1)(0,11)
\pcline[style=exist]{<-}(14.5,1)(14.5,11)
\tlput{$\ee_1$}
\pcline[style=exist]{->}(15.5,1)(15.5,11)
\trput{$\ee_1\inv$}
\pcline[style=exist]{<-}(29.5,1)(29.5,11)
\tlput{$\ee_2$}
\pcline[style=exist]{->}(30.5,1)(30.5,11)
\trput{$\ee_2\inv$}
\pcline[style=exist]{<-}(44.5,1)(44.5,11)
\pcline[style=exist]{->}(45.5,1)(45.5,11)
\pcline[style=exist]{<-}(59.5,1)(59.5,11)
\tlput{$\ee_\qq$}
\pcline[style=exist]{->}(60.5,1)(60.5,11)
\trput{$\ee_\qq\inv$}
\pcline[style=exist]{<-}(74.5,1)(74.5,11)
\pcline[style=exist]{->}(75.5,1)(75.5,11)
\pcline[style=exist]{<-}(89.5,1)(89.5,11)
\pcline[style=exist]{->}(90.5,1)(90.5,11)
\psline[style=double](105,1)(105,11)
\end{picture}
\caption[]{\sf\smaller Deformation by invertible elements: invertible elements connect the corresponding entries; if one path is shorter than the other (here we are in the case $\qq < \pp$), it is extended by identity-elements.}
\label{F:NFDeformation}
\end{figure}

Note that, if $\CCC$ contains no nontrivial invertible element, that is, the only invertible elements are the identity-elements, then Proposition~\ref{P:NormalUnique} provides a genuine uniqueness result provided one discards the $\SSS$-normal paths that finish with an identity-element.

As for the existence of $\SSS$-normal decompositions, it naturally depends on the family~$\SSS$. Here is where Garside families appear:

\begin{defi}
A subfamily~$\SSS$ of a left-cancellative category~$\CCC$ is called a \emph{Garside family} if every element of~$\CCC$ admits an $\SSS$-normal decomposition.
\end{defi}

Every left-cancellative category is a Garside family in itself, so every left-cancell\-at\-ive category contains a Garside family. Practically recognizing whether a given family is a Garside family will be one of the main tasks of Sections~\ref{S:Pres} and~\ref{S:Germ} below. We shall appeal to the following simple characterization which is valid whenever the ambient category satisfies special assumptions:

\begin{lemm}\cite[Corollary~IV.2.18]{Garside}
\label{L:RecGar}
Assume that $\CCC$ is a left-cancellative category that is right-Noetherian and admits unique local right-lcms. Then a subfamily~$\SSS$ of~$\CCC$ is a Garside family in~$\CCC$ if and only if $\SSS$ generates~$\CCC$ and, for all $\aa, \bb$ in~$\SSS$ admitting a common right-multiple, $\aa \lcm \bb$ and $\aa \under\bb$ lie in~$\SSS \cup \Id\CCC$.
\end{lemm}

Also, we shall use the following closure result, which in some sense extends Lemma~\ref{L:RecGar}, but need not characterize Garside families in general.

\begin{lemm}
\label{L:GarClosed}
\cite[Proposition~3.9]{Dif} If $\SSS$ is a Garside family in a left-cancellative category~$\CCC$ and $\cc$ is a common right-multiple of two elements~$\aa, \bb$ of~$\SSSs$, there exists a common right-multiple~$\cc'$ of~$\aa$ and~$\bb$ such that $\cc$ is a right-multiple of~$\cc'$ and $\cc'$, together with~$\aa'$ and~$\bb'$ defined by $\aa \bb' = \bb \aa' = \cc'$, lie in~$\SSSs$.
\end{lemm}

An application of Lemma~\ref{L:GarClosed} is that every Garside family gives rise to a simple presentation of the ambient category. For our current purpose it will be sufficient to state the result in the particular case when no nontrivial invertible element exists, so that a right-lcm is unique when it exists.

\begin{prop}
\label{P:Gar2Pres}
\cite[Proposition~IV.3.6]{Garside} Assume that $\SSS$ is a Garside family in a left-cancell\-at\-ive category~$\CCC$ that contains no nontrivial invertible element and admits local right-lcms. Let $\RRR$ consist of all relations $\aa (\aa \under\bb) = \bb(\bb \under \aa)$ for $\aa, \bb$ in~$\SSS$ admitting a common right-multiple. Then $(\SSS ; \RRR)$ is a presentation of~$\CCC$.
\end{prop}

Finally, we shall use the following simple diagrammatic rule about $\SSS$-greedy paths.

\begin{lemm}
\label{L:Domino1}
\rightskip40mm
\cite[Lemma 3.3]{Dif} (first domino rule) Assume that $\CCC$ is a left-cancellative category, $\SSS$ is included in~$\CCC$, and we have a commutative diagram with edges in~$\CCC$ as on the right. If $\seqq{\gg_1}{\gg_2}$ and $\seqq{\gg'_1}{\ff_1}$ are $\SSS$-greedy, then $\seqq{\gg'_1}{\gg'_2}$ is $\SSS$-greedy as well.
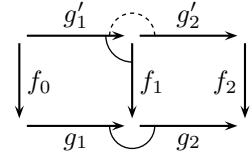
\hfill\begin{picture}(0,0)(-8,-1)
\psarc[style=thin](15,0){3}{180}{360}
\psarc[style=thin](15,12){3.5}{180}{270}
\psarc[style=thinexist](15,12){3}{0}{180}
\pcline{->}(1,0)(14,0)
\tbput{$\gg_1$}
\pcline{->}(16,0)(29,0)
\tbput{$\gg_2$}
\pcline{->}(1,12)(14,12)
\taput{$\gg'_1$}
\pcline{->}(16,12)(29,12)
\taput{$\gg'_2$}
\pcline{->}(0,11)(0,1)
\trput{$\ff_0$}
\pcline{->}(15,11)(15,1)
\trput{$\ff_1$}
\pcline{->}(30,11)(30,1)
\tlput{$\ff_2$}
\end{picture}
\end{lemm}

A second, symmetric domino rule will be mentioned in Lemma~\ref{L:Domino2} below.

\subsection{Symmetric normal decompositions and strong Garside families}
\label{SS:Strong}

If $\CCC$ is a left-Ore category, then, by Proposition~\ref{P:Ore}, it embeds in the groupoid of left-fractions that we shall denote by~$\Env\CCC$ (like ``enveloping groupoid''). If $\SSS$ is a Garside family in~$\CCC$, every element of~$\CCC$ admits an $\SSS$-normal decomposition, and it is natural to try to extend the result from~$\CCC$ to~$\Env\CCC$, that is, to find distinguished decompositions for the elements of~$\Env\CCC$ in terms of elements of~$\SSSs$ and their inverses.

To do it, we extend the notion of an $\AAA$-path into that of a signed $\AAA$-path. Formally, if $\AAA$ is any precategory, we introduce a family $\INV\AAA$ that is disjoint from and in bijection to $\AAA$ as $\INV\AAA = \{ \INV\aa \mid \aa\in\AAA \}$, where the source of $\INV\aa$ is the target of $\aa$ and vice versa. A \emph{signed $\AAA$-path} is defined to be a $(\AAA\cup\INV\AAA)$-path. We extend the ``bar'' map to all signed paths by defining $\INV{(\INV\aa)}=\aa$ for $\aa$ in~$\AAA$ and $\INV{\ww_1 \pc \ww_2} = \INV{\ww_2} \pc \INV{\ww_1}$. For $\GGG$ a groupoid and $\AAA$ included in~$\GGG$, we extend the notation~$\cl\ww$ to signed $\AAA$-paths by declaring that $\cl{\,\INV\gg\,}$ is $\gg\inv$ for~$\gg$ in~$\AAA$, that is, we use the letters of~$\INV\AAA$ to represent the inverses of the elements of~$\AAA$. If $\ww$ is a signed path and $\cl\ww = \gg$ holds, we again say that $\ww$ is a decomposition of~$\gg$.

When $\CCC$ is a left-Ore category and $\SSS$ is a Garside family of~$\CCC$, we now look for distinguished decompositions for the elements of the groupoid~$\Env\CCC$. As every element of~$\Env\CCC$ is a left-fraction, it is natural to consider decompositions that are \emph{negative--positive} $\SSSs$-paths, this meaning that every negative entry precedes every positive entry, where the entries in~$\SSSs$ are called positive and those in~$\INV{\SSSs}$ are called negative.

\begin{defi}
\label{D:GreedySym}
Assume that $\CCC$ is a left-Ore category.

\ITEM1 Two elements~$\ff, \gg$ of~$\CCC$ are called \emph{left-disjoint} if, for all~$\ff', \gg'$ in~$\CCC$ satisfying $\ff\inv \gg = \ff'{}\inv \gg'$ in~$\Env\CCC$, there exists~$\hh$ in~$\CCC$ satisfying $\ff' = \hh \ff$ and $\gg' = \hh \gg$.

\ITEM2 For $\SSS$ included in~$\CCC$, a negative--positive path $\seqqqqqq{\INV{\gg_\qq}}\etc{\INV{\gg_1}}{\ff_1}\etc{\ff_\pp}$ is called \emph{symmetric $\SSS$-greedy} (\resp \emph{symmetric $\SSS$-normal}) if $\seqqq{\gg_1}\etc{\gg_\qq}$ and $\seqqq{\ff_1}\etc{\ff_\pp}$ are $\SSS$-greedy (\resp $\SSS$-normal) and, in addition, $\gg_1$ and $\ff_1$ are left-disjoint.
\end{defi}

When using diagrams, we shall indicate that two elements~$\ff, \gg$ are left-disjoint by appending a small arc as in 
\begin{picture}(30,4)(0,0)
\pcline{<-}(1,0)(14,0)
\taput{$\gg$}
\pcline{->}(16,0)(29,0)
\taput{$\ff$}
\psarc[style=thin](15,0){3}{0}{180}
\end{picture}. So a generic symmetric $\SSS$-greedy path pictorially corresponds to a diagram 
$$\begin{picture}(90,4)(0,0)
\pcline{<-}(1,0)(14,0)
\taput{$\gg_\qq$}
\pcline[style=etc](17,0)(28,0)
\pcline{<-}(31,0)(44,0)
\taput{$\gg_1$}
\pcline{->}(46,0)(59,0)
\taput{$\ff_1$}
\pcline[style=etc](62,0)(73,0)
\pcline{->}(76,0)(89,0)
\taput{$\ff_\pp$}
\psarc[style=thin](15.5,0){3}{0}{180}
\psarc[style=thin](30.5,0){3}{0}{180}
\psarc[style=thin](45,0){3}{0}{180}
\psarc[style=thin](59.5,0){3}{0}{180}
\psarc[style=thin](74.5,0){3}{0}{180}
\end{picture},$$
and it is symmetric $\SSS$-normal if, in addition, all edges correspond to elements of~$\SSSs$. Note that a positive path is symmetric $\SSS$-normal if and only if it $\SSS$-normal: indeed, a positive path is a negative--positive path whose negative part is empty, and, for every~$\gg$ in~$\SSSs(\xx, \ud)$, the elements~$\ew_\xx$ and~$\gg$ are (trivially) left-disjoint.

\begin{exam}\label{X:FreeAbelianGroup}
Consider the free Abelian monoid~$\mathcal A$ and the set~$\SSS$ as in Example~\ref{X:FreeAbelianMonoid}.
The unordered pairs of left-disjoint elements of $\SSSs$ are precisely $\{1,1\}$, $\{1,\tta\}$, $\{1,\ttb\}$, $\{1,\tta\ttb\}$, and $\{\tta,\ttb\}$, where $1=\id{\xx}$ for the unique element $\xx$ of $\Obj(\mathcal{A})$.
\end{exam}

Like $\SSS$-normal decompositions in the positive case, symmetric $\SSS$-normal decompositions turn out to be (nearly) unique when they exist.

\begin{prop}\cite[Proposition~III.2.16]{Garside}
\label{P:SymUnique}
Assume that $\CCC$ is a left-Ore category and $\SSS$ is included in~$\CCC$. Then any two symmetric $\SSS$-normal decompositions of an element of~$\Env\CCC$ (if any) are $\CCCi$-deformations of one another, this meaning that there exists a commutative diagram as in Figure~\ref{F:SymUnique}.
\end{prop}

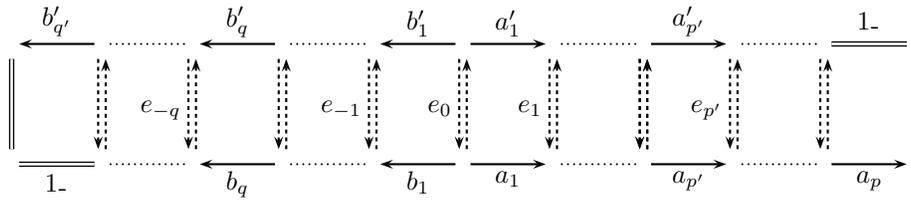
\begin{figure}[htb]
\begin{picture}(120,25)(-2,-3)
\pcline[style=double](0,1)(10,1)
\tbput{$\id\ud$}
\pcline{<-}(0,17)(10,17)
\taput{$\bb'_{\qq'}$}
\pcline[style=etc](12,1)(22,1)
\pcline[style=etc](12,17)(22,17)
\pcline{<-}(24,1)(34,1)
\tbput{$\bb_\qq$}
\pcline{<-}(24,17)(34,17)
\taput{$\bb'_\qq$}
\pcline[style=etc](36,1)(46,1)
\pcline[style=etc](36,17)(46,17)
\pcline{<-}(48,1)(58,1)
\tbput{$\bb_1$}
\pcline{<-}(48,17)(58,17)
\taput{$\bb'_1$}
\pcline{->}(60,1)(70,1)
\tbput{$\aa_1$}
\pcline{->}(60,17)(70,17)
\taput{$\aa'_1$}
\pcline[style=etc](72,1)(82,1)
\pcline[style=etc](72,17)(82,17)
\pcline{->}(84,1)(94,1)
\tbput{$\aa_{\pp'}$}
\pcline{->}(84,17)(94,17)
\taput{$\aa'_{\pp'}$}
\pcline[style=etc](96,1)(106,1)
\pcline[style=etc](96,17)(106,17)
\pcline{->}(108,1)(118,1)
\tbput{$\aa_\pp$}
\pcline[style=double](108,17)(118,17)
\taput{$\id\ud$}
\psline[style=double](-1,3)(-1,15)
\pcline[style=exist]{<-}(10.5,3)(10.5,15)
\pcline[style=exist]{->}(11.5,3)(11.5,15)
\pcline[style=exist]{<-}(22.5,3)(22.5,15)
\tlput{$\ee_{-\qq}$}
\pcline[style=exist]{->}(23.5,3)(23.5,15)
\pcline[style=exist]{<-}(34.5,3)(34.5,15)
\pcline[style=exist]{->}(35.5,3)(35.5,15)
\pcline[style=exist]{<-}(46.5,3)(46.5,15)
\tlput{$\ee_{-1}$}
\pcline[style=exist]{->}(47.5,3)(47.5,15)
\pcline[style=exist]{<-}(58.5,3)(58.5,15)
\tlput{$\ee_0$}
\pcline[style=exist]{->}(59.5,3)(59.5,15)
\pcline[style=exist]{<-}(70.5,3)(70.5,15)
\tlput{$\ee_1$}
\pcline[style=exist]{->}(71.5,3)(71.5,15)
\pcline[style=exist]{<-}(82.5,3)(82.5,15)
\pcline[style=exist]{->}(83.5,3)(83.5,15)
\pcline[style=exist]{<-}(82.5,3)(82.5,15)
\pcline[style=exist]{->}(83.5,3)(83.5,15)
\pcline[style=exist]{<-}(94.5,3)(94.5,15)
\tlput{$\ee_{\pp'}$}
\pcline[style=exist]{->}(95.5,3)(95.5,15)
\pcline[style=exist]{<-}(106.5,3)(106.5,15)
\pcline[style=exist]{->}(107.5,3)(107.5,15)
\psline[style=double](119,3)(119,15)
\end{picture}
\caption[]{\sf\smaller Deformation of a signed path by invertible elements: invertible elements connect the corresponding entries ; if some path is shorter than the other (here we have $\qq < \qq'$ and $\pp' < \pp$), it is extended by identity-elements.}
\label{F:SymUnique}
\end{figure}

As for existence, symmetric normal decompositions are connected with left-lcms:

\begin{lemm}
\label{L:Disjoint}
\cite[Lemma~III.2.19]{Garside} Assume that $\SSS$ is a Garside family in a left-Ore category~$\CCC$, that $\seqqq{\aa_1}\etc{\aa_\pp}$ and $\seqqq{\bb_1}\etc{\bb_\qq}$ are $\SSS$-normal paths, and that $\seqqqq{\bb_1}{\etc}{\bb_\qq}{\gg} = \seqqqq{\aa_1}{\etc}{\aa_\pp}{\ff}$ holds. Then $\seqqqqqq{\INV{\bb_\qq}}\etc{\INV{\bb_1}}{\aa_1}\etc{\aa_\pp}$ is symmetric $\SSS$-greedy if and only if $\bb_1 \pdots \bb_\qq$ and $\aa_1 \pdots \aa_\pp$ are left-disjoint, if and only if $\bb_1 \pdots \bb_\qq \gg$ is a left-lcm of~$\ff$ and~$\gg$.
\end{lemm}

Lemma~\ref{L:Disjoint} implies that an element~$\ff \gg\inv$ of~$\Env\CCC$ admits a symmetric $\SSS$-normal decomposition if and only if $\ff$ and~$\gg$ admit a left-lcm in~$\CCC$, whence:

\begin{prop}
\label{P:SymExist}
If $\SSS$ is a Garside family in a left-Ore category~$\CCC$ that admits left-lcms, every element of~$\Env\CCC$ that can be expressed as a right-fraction admits a symmetric $\SSS$-normal decomposition. 
\end{prop}

In particular, if $\CCC$ is an Ore category, then every element of~$\Env\CCC$ can be expressed as a right-fraction, and we obtain:

\begin{coro}
If $\SSS$ is a Garside family in an Ore category~$\CCC$ that admits left-lcms, every element of~$\Env\CCC$ admits a symmetric $\SSS$-normal decomposition.
\end{coro}

When algorithmic questions are involved, it will be convenient to consider special Garside families that we introduce now.

\begin{defi}
\label{D:Strong}
\rightskip40mm
Assume that $\CCC$ is a left-Ore category. A Garside family~$\SSS$ of~$\CCC$ is called \emph{strong} if, for all~$\aa, \bb$ in~$\SSSs$ with the same target, there exist~$\aa', \bb'$ in~$\SSSs$ that are left-disjoint and satisfy $\aa' \bb = \bb' \aa$. 
\hfill
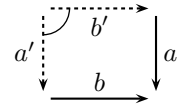
\begin{picture}(0,0)(-15,1)
\pcline{->}(1,0)(14,0)
\taput{$\bb$}
\pcline[style=exist]{->}(1,12)(14,12)
\tbput{$\bb'$}
\pcline[style=exist]{->}(0,11)(0,1)
\tlput{$\aa'$}
\pcline{->}(15,11)(15,1)
\trput{$\aa$}
\psarc[style=thin](0,12){3.5}{270}{360}
\end{picture}
\end{defi}

The interest of introducing the notion of a strong Garside family is to allow for a refined version of Proposition~\ref{P:SymExist}:

\begin{prop}\cite[Proposition~III.2.31]{Garside}
If $\SSS$ is a strong Garside family in a left-Ore category~$\CCC$ that admits left-lcms, every element of~$\Env\CCC$ that admits a positive--negative $\SSSs$-decomposition of length~$\ell$ admits a symmetric $\SSS$-normal decomposition of length at most~$\ell$. 
\end{prop}

Finally, as in the positive case, we shall appeal to diagrammatic rules involving normal paths.

\begin{lemm}
\cite[Propositions III.2.39 and III.2.42]{Garside}
\label{L:Domino3}

\ITEM1 (third domino rule) Assume that $\CCC$ is a left-cancell\-ative category, $\SSS$ is included in~$\CCC$, and we have a commutative diagram with edges in~$\CCC$ as on the right. If $\seqq{\gg_1}{\gg_2}$ is $\SSS$-greedy, and $\ff_1$, $\gg'_2$ are left-disjoint, then $\seqq{\gg'_1}{\gg'_2}$ is $\SSS$-greedy as well.
\hfill\rightskip40mm
\begin{picture}(0,0)(-8,-4.5)
\psarc[style=thin](14.5,0){3}{180}{360}
\psarc[style=thin](15,12){3.5}{270}{360}
\psarc[style=thinexist](14.5,12){3}{0}{180}
\pcline{->}(1,0)(14,0)
\tbput{$\gg_1$}
\pcline{->}(16,0)(29,0)
\tbput{$\gg_2$}
\pcline{->}(1,12)(14,12)
\taput{$\gg'_1$}
\pcline{->}(16,12)(29,12)
\taput{$\gg'_2$}
\pcline{->}(0,11)(0,1)
\trput{$\ff_0$}
\pcline{->}(15,11)(15,1)
\tlput{$\ff_1$}
\pcline{->}(30,11)(30,1)
\tlput{$\ff_2$}
\end{picture}

\ITEM2 (fourth domino rule)
Assume that $\CCC$ is a left-Ore category, $\SSS$ is included in~$\CCC$, and we have a commutative diagram with edges in~$\CCC$ as on the right. If $\gg_1, \gg_2$ are left-disjoint, and $\ff_1, \gg'_2$ are left-disjoint, then $\gg'_1, \gg'_2$ are left-disjoint as well.
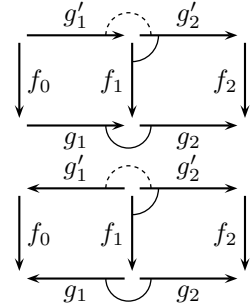
\hfill\begin{picture}(0,0)(-8,-1)
\psarc[style=thin](14.5,0){3}{180}{360}
\psarc[style=thin](15,12){3.5}{270}{360}
\psarc[style=thinexist](14.5,12){3}{0}{180}
\pcline{<-}(1,0)(14,0)
\tbput{$\gg_1$}
\pcline{->}(16,0)(29,0)
\tbput{$\gg_2$}
\pcline{<-}(1,12)(14,12)
\taput{$\gg'_1$}
\pcline{->}(16,12)(29,12)
\taput{$\gg'_2$}
\pcline{->}(0,11)(0,1)
\trput{$\ff_0$}
\pcline{->}(15,11)(15,1)
\tlput{$\ff_1$}
\pcline{->}(30,11)(30,1)
\tlput{$\ff_2$}
\end{picture}
\end{lemm}

\subsection{Bounded Garside families and $\Delta$-normal decompositions}
\label{SS:Bounded}

In many cases, interesting Garside families in a monoid consist of the left-divisors of some maximal element~$\Delta$, in which case it is natural to call them bounded by~$\Delta$. In a category context, the maximal element has to depend on the source, and it is replaced with a map from the objects to the elements.

\begin{defi}
\label{D:Bounded}
A Garside family~$\SSS$ in a cancellative category~$\CCC$ is called \emph{bounded} if there exists a map~$\Delta$ from~$\Obj(\CCC)$ to~$\CCC$ satisfying the following conditions:

\ITEM1 $\aa \in \SSS(\xx, \ud)$ implies $\aa \dive\Delta(x)$,

\ITEM2 $\forall\yy \in \Obj(\CCC)\ \exists!\zz \in \Obj(\CCC)\ \forall \aa \in \SSSs(\ud, \yy) \ \exists \aa' \in \SSSs(\zz, \ud) \ (\aa' \aa = \Delta(\zz))$.
\end{defi}

Note that, in Definition~\ref{D:Bounded}, \ITEM1 and~\ITEM2 are symmetric, as \ITEM1 can be stated as
$$\forall\xx \in \Obj(\CCC)\ \exists!\yy \in \Obj(\CCC)\ \forall \aa \in \SSSs(\xx, \ud) \ \exists \aa' \in \SSSs(\ud, \yy) \ (\aa \aa' = \Delta(\xx)),$$
with $\yy$ the target of~$\Delta(\xx)$. In the above context, for every element~$\aa$ of~$\CCC$, the unique element~$\aa'$ satisfying $\aa \aa' = \Delta(\xx)$ is denoted by~$\der(\aa)$, whereas the unique element~$\aa'$ satisfying $\aa' \aa = \Delta(\zz)$ for some~$\zz$ is denoted by~$\derL(\aa)$.

\begin{rema}
A Garside family satisfying Definition~\ref{D:Bounded}\ITEM1 only is called \emph{right-bounded} by~$\Delta$. The latter notion is natural and useful in the positive case, but it is not sufficient in the signed case and we shall not consider it here.
\end{rema}

\begin{prop}\cite[Proposition~VI.3.11]{Garside}
\ITEM1 Every cancellative category that admits a bounded Garside family is an Ore category.

\ITEM2 Every bounded Garside family is strong.
\end{prop}

The most important technical property implied by the existence of the bounding map~$\Delta$ is the existence of a derived automorphism of the ambient category.

\begin{lemm}
\label{L:Auto}
\cite[Proposition~VI.1.11]{Garside} Assume that $\SSS$ is a Garside family bounded by~$\Delta$ in a cancellative category~$\CCC$. Put $\phi(\xx) = \trg{\Delta(\xx)}$ for~$\xx$ in~$\Obj(\CCC)$ and $\phi(\gg) = \der^2(\gg)$ for~$\gg$ in~$\CCC$. Then $\phi$ is an automorphism of~$\CCC$ that makes the diagram aside commutative for every~$\gg$ in~$\CCC(\xx, \yy)$. 
\hfill\rightskip40mm
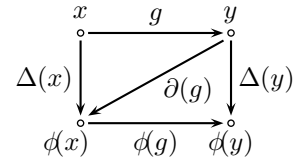
\begin{picture}(0,0)(-12,-5.5)
\pscircle[style=thin](0,12){0.5}
\put(-1,14){$\xx$}
\pcline{->}(1,12)(19,12)
\taput{$\gg$}
\pscircle[style=thin](20,12){0.5}
\put(19,14){$\yy$}
\pscircle[style=thin](0,0){0.5}
\put(-5,-3.5){$\phi\!(\xx)$}
\pcline{->}(1,0)(19,0)
\tbput{$\phi\!(\gg)$}
\pscircle[style=thin](20,0){0.5}
\put(17,-3.5){$\phi\!(\yy)$}
\pcline{->}(0,11)(0,1)
\tlput{$\Delta(\xx)$}
\pcline{->}(20,11)(20,1)
\trput{$\Delta(\yy)$}
\pcline{->}(19,11)(1,1)
\put(11,3.5){$\der(\gg)$}
\end{picture}
\end{lemm}

In terms of normal decompositions, the bounding map implies a simple connection between greediness and left-disjointness that need not be valid in general.

\begin{lemm}
\label{L:DeltaDisjoint}
\cite[Proposition~VI.1.46]{Garside} Assume that $\SSS$ is a Garside family bounded by a map~$\Delta$ in a cancellative category~$\CCC$. Then, for all~$\aa_1, \aa_2$ in~$\SSSs$, the following are equivalent:

\ITEM1 $\seqq{\aa_1}{\aa_2}$ is $\SSS$-normal;

\ITEM2 $\der(\aa_1)$ and $\aa_2$ are left-disjoint;

\ITEM3 $\der(\aa_1)$ and $\aa_2$ admit no nontrivial common left-divisor.
\end{lemm}

A consequence is the following counterpart of the first domino rule.

\begin{lemm}
\label{L:Domino2}
\cite[Lemma~VI.1.32]{Garside} (second domino rule) Assume that $\CCC$ is a cancellative category and $\SSS$ is a bounded Garside family in~$\CCC$. Then, when we have a commutative diagram as on the right with edges in~$\SSSs$ in which $\seqq{\aa_1}{\aa_2}$ and $\seqq{\bb_1}{\aa'_2}$ are $\SSS$-greedy, then $\seqq{\aa'_1}{\aa'_2}$ is $\SSS$-greedy as well.
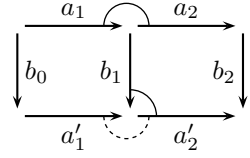
\hfill\rightskip40mm
\begin{picture}(0,0)(-8,-5)
\psarc[style=thinexist](14.5,0){3}{180}{360}
\psarc[style=thin](15,0){3.5}{0}{90}
\psarc[style=thin](14.5,12){3}{0}{180}
\pcline{->}(1,0)(14,0)
\tbput{$\aa'_1$}
\pcline{->}(16,0)(29,0)
\tbput{$\aa'_2$}
\pcline{->}(1,12)(14,12)
\taput{$\aa_1$}
\pcline{->}(16,12)(29,12)
\taput{$\aa_2$}
\pcline{->}(0,11)(0,1)
\trput{$\bb_0$}
\pcline{->}(15,11)(15,1)
\tlput{$\bb_1$}
\pcline{->}(30,11)(30,1)
\tlput{$\bb_2$}
\end{picture}
\end{lemm}

Note that the second domino rule is not an exact counterpart of the first one in that, here, all involved elements are supposed to lie in the reference family~$\SSSs$.

Also, a second type of distinguished decomposition naturally arises for the elements of the associated groupoid of the considered category: called \emph{$\Delta$-normal}, these decompositions differ from symmetric normal decompositions in that the denominator is demanded to involve a power of the Garside map~$\Delta$. In the case of a monoid~$\MM$, a bound~$\Delta$ for a Garside family is an element of~$\MM$, and it makes sense to take powers of~$\Delta$. In the case of a general category~$\CCC$, the bound is a map of~$\Obj(\CCC)$ to~$\CCC$, and the notion of a power has to be adapted.

\begin{nota}
\label{N:PowerOfDelta}
Assume that $\CCC$ is a cancellative category and $\Delta$ is a Garside map in~$\CCC$. For $\nn$ in~$\ZZZZ$ and $\xx$ in~$\Obj(\CCC)$, we put 
\begin{equation}
\DELTA\nn(\xx):=
\begin{cases}
\Delta(\xx) \pc \Delta(\phi(\xx)) \pc \pdots \pc \Delta(\phi^{\nn-1}(\xx))
&\mbox{for $\nn > 0$},\\
\ew_\xx
&\mbox{for $\nn = 0$},\\
\INV{\Delta(\phi\inv(\xx))} \pc \INV{\Delta(\phi^{-2}(\xx))} \pc \pdots \pc \INV{\Delta(\phi^{-\vert\nn\vert}(\xx))}
&\mbox{for $\nn < 0$},
\end{cases}
\end{equation}
and we write $\DELTAA\nn(\xx)$ for the element of~$\Env\CCC$ represented by~$\DELTA\nn(\xx)$.
\end{nota}

Note that, in every case, the source of~$\DELTA\nn(\xx)$ and~$\DELTAA\nn(\xx)$ is~$\xx$ and its target is $\phi^\nn(\xx)$, and that $\DELTAA{-\nn}(\xx)$ is always the inverse of~$\DELTAA\nn(\phi^{-\nn}(\xx))$.

\begin{defi}
\label{D:DeltaNormal}
Assume that $\SSS$ is a Garside family of~$\CCC$ bounded by a map~$\Delta$ in an Ore category~$\CCC$.

\ITEM1 An element~$\gg$ of~$\CC(\xx, \ud)$ is called \emph{$\Delta$-like} if $\gg \eqir \Delta(\xx)$ holds.

\ITEM2 A signed $\SSS$-path is called \emph{$\Delta$-normal} if it has the form $\seqqqq{\DELTA\nn(\ud)}{\aa_1}\etc{\aa_\pp}$ with $\nn$ in~$\ZZZZ$ and $\seqqq{\aa_1}\etc{\aa_\pp}$ an $\SSS$-normal path such that $\aa_1$ is not $\Delta$-like.
\end{defi}

Thus a $\Delta$-normal path is either a positive $\SSS$-path beginning with elements in the image of~$\Delta$, or an empty path, or a negative--positive $\SSS$-path whose negative part consists of elements in the image of~$\Delta$. 

\begin{prop}
\label{P:DeltaUnique}
\cite[Proposition~VI.3.19]{Garside} Assume that $\SSS$ is a Garside family of~$\CCC$ bounded by a map~$\Delta$ in an Ore category~$\CCC$. Then every element of~$\Env\CCC$ admits a $\Delta$-normal decomposition, in which the exponent of~$\Delta$ is uniquely determined and the other entries are unique up to $\CCCi$\!-deformation.
\end{prop}

In the positive case, it is easily seen that $\Delta$-like entries must lie at the beginning of an $\SSS$-normal path, so a $\Delta$-normal path is simply an $\SSS$-normal path in which the initial $\Delta$-like entries are not only $\Delta$-like but even lie in the image of~$\Delta$. In the signed case, the difference with a symmetric $\SSS$-normal path is more important: in a $\Delta$-normal path, the numerator and the denominator need not be left-disjoint, the requirement being now that $\Delta$ does not left-divide the first positive entry. Actually, Propositions~\ref{P:SymExist} and~\ref{P:DeltaUnique} are not exactly comparable, as they require different assumptions, namely the existence of left-lcms in the former case, and that of a bounded Garside family in the latter.

\section{Recognizing Garside families, case of a presentation}
\label{S:Pres}

We now begin to investigate the effective methods relevant for Garside structures. In this section as well as the next one, we address the question of recognizing that a given family is a Garside family in the ambient category, as well as checking that the category is eligible for the Garside approach, that is, it is left-cancellative (\resp left-Ore). The question depends in turn on the way the category and the candidate Garside family are specified. In this section, we consider the case when the category is specified using a presentation. 

We first define reversing (Subsection~\ref{SS:Rev}), then address establishing that the ambient category is cancellative (Subsection~\ref{SS:Cancel}), recognizing Garside families (Subsection~\ref{SS:RecGar}), and finally proving further properties like being an Ore category or a strong Garside family (Subsection~\ref{SS:Further}).

\subsection{The reversing transformation}
\label{SS:Rev}

Before entering the main development, we introduce a technical tool that will be used several times in the sequel, namely a path transformation called reversing~\cite{Dff, Dgp, Dia}.

\begin{defi}
\label{D:RightComp}
Assume that $\AAA$ is a precategory. A \emph{right-complement} on~$\AAA$ is a partial map~$\RC$ of~$\AAA^2$ to~$\AAA^*$ such that $\RC(\aa, \aa)$ is defined and equal to~$\ew_\yy$ for every~$\aa$ in~$\AAA(\xx, \yy)$ and that, if $\RC(\aa, \bb)$ is defined, then $\aa$ and $\bb$ have the same source, $\RC(\bb, \aa)$ is defined, and both $\seqq{\aa}{\RC(\aa, \bb)}$ and $\seqq{\bb}{\RC(\bb, \aa)}$ are defined and have the same target. A right-complement is called \emph{short} if, whenever $\RC(\aa, \bb)$ is defined, $\RC(\aa, \bb)$ has length at most~$1$, that is, it belongs to~$\AAA$ or is empty. 
\end{defi}

Note that, if $\Obj(\AAA)$ consists of a single element, the conditions for~$\RC$ to be a right-complement are simply that $\RC(\aa, \aa)$ is empty for every~$\aa$ and that $\RC(\aa, \bb)$ is defined if and only if~$\RC(\bb, \aa)$ is.

\subsubsection*{Definition of right-reversing}

We are now ready to introduce right-reversing. We begin with the special case of a short right-complement and a negative--positive input path, which will be frequently used in the sequel.

\begin{nota}
\label{N:ExtEmpty}
If $\AAA$ is a precategory, we write $\ew_\AAA$ for $\{\ew_\xx \mid \xx \in \Obj(\AAA)\}$ and $\AAAh$ for $\AAA \cup \ew_\AAA$. For $\RC$ a right-complement on~$\AAA$, we write $\RCh$ for the extension of~$\RC$ to~$\AAAh$ obtained by adding $\RCh(\ew_\xx, \ew_\xx) = \ew_\xx$ for every object~$\xx$ plus $\RCh(\aa, \ew_\xx) = \ew_\yy$ and $\RCh(\ew_\xx, \aa) = \aa$ for every~$\aa$ in~$\AAA(\xx, \yy)$.
\end{nota}

Note that $\RCh$ is a right-complement on~$\AAAh$. We introduce right-reversing by means of an algorithm working on (certain) signed $\AAA$-paths.

\begin{algo}{Right-reversing, short case, negative--positive input}
\label{A:RightRevShort}
\begin{algorithmic}[1]
\CONTEXT{A precategory~$\AAA$, a short right-complement~$\RC$ on~$\AAA$}
\INPUT{A negative--positive $\AAA$-path $\seqqqqqq{\INV{\bb_\qq}}\etc{\INV{\bb_1}}{\aa_1}\etc{\aa_\pp}$}
\OUTPUT{A positive--negative $\AAA$-path, or $\mathtt{fail}$}
\STATE{$\bb_{\ii,0} := \bb_\ii$ for $\ii = 1 \wdots \qq$}
\STATE{$\aa_{0,\jj} := \aa_\jj$ for $\jj = 1 \wdots \pp$}
\FOR{$\ii$ increasing from $1$ to $\qq$}
\FOR{$\jj$ increasing from $1$ to $\pp$}
\IF{$\RCh(\aa_{\ii-1,\jj}, \bb_{\ii, \jj-1})$ is defined}
\STATE{$\bb_{\ii, \jj} := \RCh(\aa_{\ii-1,\jj}, \bb_{\ii, \jj-1})$}
\STATE{$\aa_{\ii, \jj} := \RCh(\bb_{\ii, \jj-1}, \aa_{\ii-1,\jj})$}
\ELSE
\RETURN{$\mathtt{fail}$}
\ENDIF
\ENDFOR
\ENDFOR
\RETURN{$\seqqqqqq{\aa_{\qq,1}}\etc{\aa_{\qq,\pp}}{\INV{\bb_{\qq,\pp}}}\etc{\INV{\bb_{1,\pp}}}$}
\end{algorithmic}
\end{algo}

\begin{lemm}
\label{L:TerminatingShort}
If $\RC$ is a short right-complement on a finite precategory~$\AAA$, Algorithm~\ref{A:RightRevShort} running on a pair of $\AAA$-paths of length at most~$\ell$ terminates in $O(\ell^2)$~steps.
\end{lemm}
\begin{proof}
The claim is obvious from the pseudocode in Algorithm~\ref{A:RightRevShort}.
\end{proof}

\rightskip40mm
Running Algorithm~\ref{A:RightRevShort} amounts to recursively constructing a rectangular grid whose edges are labeled by elements of~$\AAA$ or empty paths, and in which each elementary square has the type
shown aside, see Figure~\ref{F:RightRevShort} for an example. 
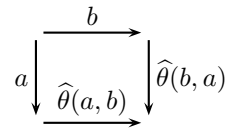
\hfill\begin{picture}(0,0)(-11,-2)
\pcline{->}(1,12)(14,12)
\taput{$\bb$}
\pcline{->}(0,11)(0,1)
\tlput{$\aa$}
\pcline{->}(1,0)(14,0)
\taput{$\RCh(\aa, \bb)$}
\pcline{->}(15,11)(15,1)
\trput{$\RCh(\bb, \aa)$}
\end{picture}

\rightskip0mm
\begin{figure}[htb]
\begin{picture}(60,40)(0,0)
\pcline[style=thick]{->}(1,36)(14,36)
\taput{$\ttb$}
\pcline[style=thick]{->}(16,36)(29,36)
\taput{$\tta$}
\pcline[style=thick]{->}(31,36)(44,36)
\taput{$\ttb$}
\pcline[style=thick]{->}(46,36)(59,36)
\taput{$\ttb$}

\pcline[style=thick]{->}(0,35)(0,25)
\tlput{$\tta$}
\pcline{->}(15,35)(15,25)
\trput{$\tta$}
\pcline[style=double](30,35)(30,25)
\pcline[style=double](45,35)(45,25)
\pcline[style=double](60,35)(60,25)

\pcline{->}(1,24)(14,24)
\taput{$\ttb$}
\pcline[style=double](16,24)(29,24)
\pcline{->}(31,24)(44,24)
\taput{$\ttb$}
\pcline{->}(46,24)(59,24)
\taput{$\ttb$}

\pcline[style=thick]{->}(0,23)(0,13)
\tlput{$\ttb$}
\pcline[style=double](15,23)(15,13)
\pcline[style=double](30,23)(30,13)
\pcline[style=double](45,23)(45,13)
\pcline[style=double](60,23)(60,13)

\pcline[style=double](1,12)(14,12)
\pcline[style=double](16,12)(29,12)
\pcline{->}(31,12)(44,12)
\taput{$\ttb$}
\pcline{->}(46,12)(59,12)
\taput{$\ttb$}

\pcline[style=thick]{->}(0,11)(0,1)
\tlput{$\ttb$}
\pcline{->}(15,11)(15,1)
\trput{$\ttb$}
\pcline{->}(30,11)(30,1)
\trput{$\ttb$}
\pcline[style=double](45,11)(45,1)
\pcline[style=double](60,11)(60,1)

\pcline[style=double](1,0)(14,0)
\pcline[style=double](16,0)(29,0)
\pcline[style=double](31,0)(44,0)
\pcline{->}(46,0)(59,0)
\taput{$\ttb$}
\end{picture}
\caption{\sf\small The grid associated with a right-reversing, short case. Here we consider the right-complement~$\RC$ on~$\{\tta, \ttb\}$ defined by $\RC(\tta, \tta) =\nobreak \RC(\ttb, \ttb) =\nobreak \ew$, $\RC(\tta, \ttb) =\nobreak \ttb$, and $\RC(\ttb, \tta) = \tta$ (with one object only), and apply Algorithm~\ref{A:RightRevShort} to the negative--positive word $\INV\ttb \pc \INV\ttb \pc \INV\tta \pc \ttb \pc \tta \pc \ttb \pc \ttb$: the initial word is written on the left (negative part, here $\tta \pc \ttb \pc \ttb$) and the top (positive part, here $\ttb \pc \tta \pc \ttb \pc \ttb$), and then the grid is constructed by using the right-complement~$\RCh$ to recursively complete the squares; we use a double line for $\ew$-labeled arrows. In the current case, the output of the algorithm is the length one word~$\ttb$.}
\label{F:RightRevShort}
\end{figure}

We now turn to the case of an arbitrary right-complement, that is, we no longer assume that $\RC(\aa, \bb)$ necessarily has length zero or one. Then we can extend the definition of right-reversing, the only differences being that the constructed grid may involve rectangles whose edges contain more than one arrow when $\RC(\aa, \bb)$ has length at least~$2$.
For the description, it is convenient to start from an arbitrary signed path, and not necessarily from a negative--positive one.

\begin{algo}{Right-reversing, general case}
\label{A:RightRev}
\begin{algorithmic}[1]
\CONTEXT{A precategory~$\AAA$, a right-complement $\RC$ on~$\AAA$}
\INPUT{A signed $\AAA$-path $\ww$}
\OUTPUT{A positive--negative $\AAA$-path, or $\mathtt{fail}$, or no output}
\WHILE{$\exists \ii < \LG\ww \ (\ww[\ii] \in \INV\AAA$ and $\ww[\ii+1] \in \AAA)$}
\STATE{$\jj := \min\{\ii \mid \ww[\ii] \in \INV\AAA \mbox{ and }\ww[\ii+1] \in \AAA\}$}
\STATE{$\aa := \INV{\ww[\jj]}$}
\STATE{$\bb := \ww[\jj+1]$}
\IF{$\aa = \bb$}
\STATE{remove $\INV\aa\pc\bb$ in~$\ww$}
\ELSE
\IF{$\RCh(\aa, \bb)$ is defined}
\STATE{replace $\INV\aa \pc \bb$ in~$\ww$ with 
$\RCh(\aa, \bb) \pc \INV{\RCh(\bb, \aa)}$}
\ELSE
\RETURN{$\mathtt{fail}$}
\ENDIF
\ENDIF
\ENDWHILE
\RETURN{$\ww$}
\end{algorithmic}
\end{algo}

\begin{defi}
\label{D:RightRev}
If Algorithm~\ref{A:RightRev} terminates successfully, we say that the initial signed $\AAA$-path~$\ww$ is \emph{right-$\RC$-reversible} to the final path~$\ww'$, and we write $\ww \rev_\RC \ww'$. If $\uu, \vv$ are positive $\AAA$-paths, and there exist positive paths~$\uu', \vv'$ satisfying $\INV\uu \pc \vv \rev_\RC \vv' \pc \INV{\uu'}$, we define $\RC^*(\uu, \vv)$ to be~$\vv'$ and $\RC^*(\vv, \uu)$ to be~$\uu'$.
\end{defi}

\rightskip40mm
As in the short case, running Algorithm~\ref{A:RightRev} amounts to recursively constructing a grid whose edges are labeled by elements of~$\AAA$ or empty paths, and in which each elementary square has the type
shown aside.  The difference to Algorithm~\ref{A:RightRevShort} is twofold:
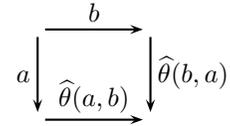
\hfill\begin{picture}(0,0)(-10,-2)
\pcline{->}(1,12)(14,12)
\taput{$\bb$}
\pcline{->}(0,11)(0,1)
\tlput{$\aa$}
\pcline{->}(1,0)(14,0)
\taput{$\RCh(\aa, \bb)$}
\pcline{->}(15,11)(15,1)
\trput{$\RCh(\bb, \aa)$}
\end{picture}

\rightskip0mm

Firstly, if the input is an arbitrary signed path, we do not necessary start with a vertical--horizontal path, but possibly with a staircase in which vertical and horizontal edges alternate; see Figure~\ref{F:RightRevSigned} for an example.

Secondly, if the right-complement $\RC$ is not short, the edges of the grid may have different sizes; see Figure~\ref{F:RightRev} for an example.

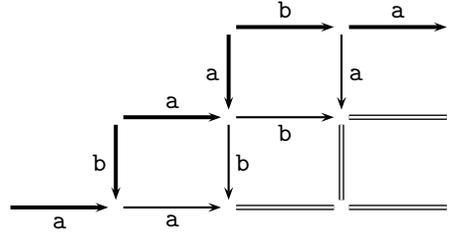
\begin{figure}[htb]
\begin{picture}(60,29)(0,10)
\pcline[style=thick]{->}(31,36)(44,36)
\taput{$\ttb$}
\pcline[style=thick]{->}(46,36)(59,36)
\taput{$\tta$}

\pcline[style=thick]{->}(30,35)(30,25)
\tlput{$\tta$}
\pcline{->}(45,35)(45,25)
\trput{$\tta$}
\pcline[style=double](60,35)(60,25)

\pcline[style=thick]{->}(16,24)(29,24)
\taput{$\tta$}
\pcline{->}(31,24)(44,24)
\tbput{$\ttb$}
\pcline[style=double](46,24)(59,24)

\pcline[style=thick]{->}(15,23)(15,13)
\tlput{$\ttb$}
\pcline{->}(30,23)(30,13)
\trput{$\ttb$}
\pcline[style=double](45,23)(45,13)
\pcline[style=double](60,23)(60,13)

\pcline[style=thick]{->}(1,12)(14,12)
\tbput{$\tta$}
\pcline{->}(16,12)(29,12)
\tbput{$\tta$}
\pcline[style=double](31,12)(44,12)
\pcline[style=double](46,12)(59,12)
\end{picture}
\caption{\sf\small The grid associated with the right-reversing of a signed path. We consider the right-complement from Figure~\ref{F:RightRevShort}, and apply Algorithm~\ref{A:RightRev} to the word $\tta \pc \INV\ttb \pc \tta\pc\INV\tta\pc\ttb \pc \tta$: the initial word is written on the top-left boundary (negative edges vertically and potivive edges horizontally), and completing the grid to the bottom-right yields the output word~$\tta\pc\tta$.}
\label{F:RightRevSigned}
\end{figure}

\begin{figure}[htb]
\begin{picture}(60,40)(0,0)
\pcline[style=thick]{->}(1,36)(14,36)
\taput{$\ttb$}
\pcline[style=thick]{->}(16,36)(29,36)
\taput{$\tta$}
\pcline[style=thick]{->}(31,36)(44,36)
\taput{$\ttb$}
\pcline[style=thick]{->}(46,36)(59,36)
\taput{$\ttb$}

\pcline[style=thick]{->}(0,35)(0,25)
\tlput{$\tta$}
\pcline{->}(15,35)(15,31)
\trput{$\tta$}
\pcline{->}(15,29)(15,25)
\trput{$\ttb$}
\pcline[style=double](30,35)(30,31)
\pcline{->}(30,29)(30,25)
\trput{$\ttb$}
\pcline[style=double](45,35)(45,31)
\pcline[style=double](45,29)(45,25)
\pcline[style=double](60,35)(60,31)
\pcline[style=double](60,29)(60,25)

\pcline[style=double](16,30)(29,30)
\pcline{->}(31,30)(44,30)
\taput{$\ttb$}
\pcline{->}(46,30)(59,30)
\taput{$\ttb$}

\pcline{->}(1,24)(6.5,24)
\taput{$\ttb$}
\pcline{->}(8.5,24)(14,24)
\taput{$\tta$}
\pcline[style=double](16,24)(29,24)
\pcline[style=double](31,24)(44,24)
\pcline{->}(46,24)(59,24)
\taput{$\ttb$}

\pcline[style=thick]{->}(0,23)(0,13)
\tlput{$\ttb$}
\pcline[style=double](7.5,23)(7.5,13)
\pcline[style=double](15,23)(15,13)
\pcline[style=double](30,23)(30,13)
\pcline[style=double](45,23)(45,13)
\pcline[style=double](60,23)(60,13)

\pcline[style=double](1,12)(6.5,12)
\pcline{->}(8.5,12)(14,12)
\taput{$\tta$}
\pcline[style=double](16,12)(29,12)
\pcline[style=double](31,12)(44,12)
\pcline{->}(46,12)(59,12)
\taput{$\ttb$}

\pcline[style=thick]{->}(0,11)(0,1)
\tlput{$\ttb$}
\pcline{->}(7.5,11)(7.5,1)
\trput{$\ttb$}
\pcline{->}(15,11)(15,7)
\trput{$\ttb$}
\pcline{->}(15,5)(15,1)
\trput{$\tta$}
\pcline{->}(30,11)(30,7)
\trput{$\ttb$}
\pcline{->}(30,5)(30,1)
\trput{$\tta$}
\pcline{->}(45,11)(45,7)
\trput{$\ttb$}
\pcline{->}(45,5)(45,1)
\trput{$\tta$}
\pcline[style=double](60,11)(60,7)
\pcline{->}(60,5)(60,1)
\trput{$\tta$}

\pcline[style=double](16,6)(29,6)
\pcline[style=double](31,6)(44,6)
\pcline[style=double](46,6)(59,6)
\pcline[style=double](1,0)(6.5,0)
\pcline{->}(8,0)(11,0)
\taput{$\tta$}
\pcline{->}(12,0)(15,0)
\taput{$\ttb\ $}
\pcline[style=double](16,0)(29,0)
\pcline[style=double](31,0)(44,0)
\pcline[style=double](46,0)(59,0)
\end{picture}
\caption{\sf\small The grid associated with a right-reversing for a right-complement that is not short. Here we consider the right-complement~$\RC$ on~$\{\tta, \ttb\}$ defined by $\RC(\tta, \tta) = \RC(\ttb, \ttb) = \ew$, $\RC(\tta, \ttb) =\nobreak \ttb \pc \tta$, and $\RC(\ttb, \tta) = \tta \pc \ttb$ (with one object only), and apply Algorithm~\ref{A:RightRev} to the negative--positive word $\INV\ttb \pc \INV\ttb \pc \INV\tta \pc \ttb \pc \tta \pc \ttb \pc \ttb$ (the same as in Figure~\ref{F:RightRevShort}): the difference is that, now, edges of variable size occur, so that, a priori, the process need not terminate. In the current case, it terminates, and the output word is~$\tta \pc \ttb \pc \INV\tta$.}
\label{F:RightRev}
\end{figure}

\begin{rema}
\label{R:NonDet}
More general versions of right-reversing are possible: at the expense of renouncing to determinism, we can consider multiform right-complements assigning with every pair of letters~$(\aa, \bb)$ a family of pairs of paths $\{(\uu_1, \vv_1) \wdots (\uu_\nn, \vv_\nn)\}$ and decide that $\INV\aa \pc \bb$ may reverse to any of the paths $\vv_1 \pc \INV{\uu_1} \wdots \vv_\nn \pc \INV{\uu_\nn}$, see~\cite{Dgp}. In such a context, several reversing grids may be associated with an initial path. Although most theoretical results can be adapted, these extended versions are less suitable for algorithms, and we shall not consider them here.
\end{rema}

\subsubsection*{Termination of reversing}

It should be clear that, whereas Algorithm~\ref{A:RightRevShort} always terminates (successfully or not, that is, with an output path or with the output ``$\mathtt{fail}$'') in finitely many steps, Algorithm~\ref{A:RightRev} may not terminate.

\begin{exam}
\label{X:BaumslagSolitar}
Consider the right-complement~$\RC$ defined on~$\{\tta, \ttb\}$ by $\RC(\tta, \ttb) = \ttb$ and $\RC(\ttb, \tta) = \ttb \pc \tta$. Let $\ww = \INV\tta \pc \ttb \pc \tta$. Then $\ww$ reverses in two steps to $\ttb \pc \ww \pc \INV\ttb$, hence in $2\nn$~steps to $\ttb^\nn \pc \ww \pc \INV\ttb^\nn$ for every~$\nn$, never leading to a positive--negative word.
\end{exam}

For our current approach, it will be useful to have a simple termination criterion.

\pagebreak

\begin{lemm}
\label{L:Terminating}
Assume that $\RC$ is a right-complement on a precategory~$\AAA$. 

\ITEM1 Right-$\RC$-reversing terminates successfully for all valid inputs if and only if there exists a family~$\BBB$ of $\AAA$-paths that includes~$\AAA$ and is such that, for all~$\uu, \vv$ in~$\BBB$ with the same source, there exist~$\uu', \vv'$ in~$\BBB$ satisfying $\INV\uu \pc \vv \rev_\RC \vv' \pc \INV{\uu'}$.

\ITEM2 If a family~$\BBB$ with the properties as in \ITEM1 exists and is finite, then, for every signed $\AAA$-path of length~$\ell$, the right-$\RC$-reversing of~$\ww$ terminates (successfully) in $O(\ell^2)$~steps and all involved paths have length in~$O(\ell)$.
\end{lemm}

\begin{proof}
\ITEM1 If right-$\RC$-reversing is terminating, then $\AAA^*$ has the expected property. 

Conversely, assume that $\BBB$ satisfies the property of the lemma. Let~$\RCd$ be the restriction of~$\RC^*$ to~$\BBB$. Then $\RCd$ is a short right-complement on~$\BBB$, so right-$\RCd$-reversing is terminating. Now, assume that $\INV\uu \pc \vv$ is a signed $\AAA$-path. By assumption, $\INV\uu \pc \vv$ is also a signed $\BBB$-path, and its right-$\RCd$-reversing terminates, so there exists a witnessing $\RCd$-grid. Now a $\RCd$-grid of size $\pp \times \qq$ is the juxtaposition of $\pp\qq$ $\RC$-grids, whose existence shows that the right-$\RC$-reversing of~$\INV\uu \pc \vv$ also terminates.

\ITEM2 Right-$\RCd$-reversing terminates in $O(\ell^2)$ steps and all involved $\BBB$-paths have length in $O(\ell^2)$. As the family $\BBB$ is finite, there exists a constant that bounds the length of any element of $\BBB$ considered as an $\AAA$-path, and the claim follows.
\end{proof}

\begin{coro}
\label{C:TerminatingShort}
If $\RC$ is a short right-complement on a finite precategory~$\AAA$, Algorithm~\ref{A:RightRev} running on a signed $\AAA$-path of length~$\ell$ terminates in $O(\ell^2)$~steps.
\end{coro}

\begin{proof}
If $\RC$ is short, the condition of Lemma~\ref{L:Terminating} is satisfied with $\BBB$ equal to the family~$\AAAh$ in Notation~\ref{N:ExtEmpty}, that is, the union of~$\AAA$ and the empty paths.
\end{proof}

So, starting with an arbitrary right-complement~$\RC$, we can possibly show that right-$\RC$-reversing terminates successfully for all valid inputs by applying the following closure method:

\begin{algo}{Termination of right-reversing}
\label{A:Termination}
\begin{algorithmic}[1]
\CONTEXT{A precategory~$\AAA$}
\INPUT{A right-complement~$\RC$ on~$\AAA$}
\OUTPUT{A (minimal) subfamily of~$\AAA^*$ that includes~$\AAA$ and is closed under~$\RC^*$}
\STATE{$\AAA_0 := \AAA$}
\REPEAT
\STATE{$\AAA_{\ii+1} := \AAA_\ii \cup \{\RC^*(\aa, \bb) \mid \aa, \bb \in \AAA_\ii\}$}
\UNTIL{$\AAA_{\ii+1} = \AAA_\ii$}
\RETURN{$\AAA_\ii$}
\end{algorithmic}
\end{algo}

\begin{exam}
\label{X:BraidMonoidB3}
Consider the right-complement of Figure~\ref{F:RightRev} again. Starting with $\AAA_0 = \{\tta, \ttb\}$ and applying Algorithm~\ref{A:Termination}, we find $\AAA_1 = \AAA_0 \cup \{\ew, \tta\pc\ttb, \ttb\pc\tta\}$, and $\AAA_2 =\nobreak \AAA_1$: here the process terminates in one step, leading to $\{\ew, \tta, \ttb, \tta\pc\ttb, \ttb\pc\tta\}$. The existence of this 5 element family that is closed under right-reversing implies that right-reversing is terminating with a quadratic complexity upper bound.
\end{exam}

When the closure under~$\RC^*$ is infinite, the situation is more complicated and there is no general result. Examples are known when right-reversing is always terminating but the time (\resp space) complexity is more than quadratic (\resp linear): for instance, the right-reversing associated with the right-complement~$\RC$ defined on~$\{\tta, \ttb\}$ by $\RC(\tta, \ttb) = \ew$ and $\RC(\ttb, \tta) = \tta (\ttb \tta \ttb)^3 \tta \ttb$ is always terminating, but the time (\resp space) complexity is cubic (\resp quadratic) \cite[Example 10.3]{Die}, whereas \cite[Proposition~VIII.1.11]{Dgd} displays an example (with an infinite family of generators) where right-reversing is terminating but the only known bound for time complexity is a tower of exponentials of exponential height.

\subsection{Establishing left-cancellativity}
\label{SS:Cancel}

We now address our main problem, name\-ly investigating a presented category~$\PRESp\AAA\RRR$ and, in particular, trying to recognize whether it is left-cancellative. Here we consider the problem for presentations of a certain syntactical type. This restriction allows for using right-reversing, and it is natural in our context as one can show that every Garside family gives rise to a presentation that is eligible for this approach (at least in the extended version alluded to in Remark~\ref{R:NonDet}).

\begin{defi}
\label{D:RightCompPres}
A presentation $(\AAA ; \RRR)$ is called \emph{positive} if all relations of~$\RRR$ are of the form $\uu = \vv$ with $\uu, \vv$ nonempty; it is called \emph{right-complemented}, associated with the right-complement~$\RC$, if $\RRR$ consists of all relations $\seqq\aa{\RC(\aa, \bb)} = \seqq\bb{\RC(\bb, \aa)}$ with $(\aa, \bb)$ in the domain of~$\RC$.
\end{defi}

Note that, by definition, a right-complemented presentation is positive and that, if $(\AAA ; \RRR)$ is a positive presentation, then the category~$\PRESp\AAA\RRR$ contains no nontrivial invertible element since an empty path and a nonempty path cannot be $\RRR$-equivalent. Saying that a presentation~$(\AAA ; \RRR)$ is right-complemented just means that it is positive and that, for all~$\aa, \bb$ in~$\AAA$, the family~$\RRR$ contains at most one relation of the form $\aa ... = \bb ...$\,. The involved right-complement is short if the paths ``...'' have length $0$ or~$1$, that is, all relations in~$\RRR$ are of the form $\uu = \vv$ with $\uu$ and $\vv$ of length~$1$ or~$2$. For instance, the presentation $(\tta, \ttb;\tta\ttb = \ttb \tta)$ is associated with the short right-complement of Figure~\ref{F:RightRevShort}, whereas the presentation $(\tta, \ttb;\tta\ttb\tta = \ttb\tta\ttb)$ is associated with the right-complement of Figure~\ref{F:RightRev} and Example~\ref{X:BraidMonoidB3}. By contrast, the presentation $(\tta, \ttb;\tta\ttb = \ttb\tta, \tta^2 = \ttb^2)$ is not right-complemented since it contains two relations of the form $\tta ... = \ttb ...$\,.

When a presentation is right-complemented, it is eligible for the associated right-reversing transformation, leading in good cases to a practical method for recognizing left-cancellativity. The first observation is that right-reversing gives a way to construct common right-multiples in the associated category:

\begin{lemm}
\label{L:Equiv}
Assume that $(\AAA ; \RRR)$ is a presentation associated with a right-comple\-ment~$\RC$. Then, for all $\AAA$-paths~$\uu, \vv, \uu', \vv'$ satisfying $\INV\uu \pc \vv \rev_\RC \vv' \pc \INV{\uu'}$, the paths $\uu \pc \vv'$ and $\vv \pc \uu'$ are $\RRR$-equivalent.
\end{lemm}

\begin{proof}
By definition, each elementary square in the rectangular grid associated with Algorithm~\ref{A:RightRev} corresponds to a relation of~$\RRR$.
\end{proof}

Lemma~\ref{L:Equiv} says in particular that, if $\uu, \vv$ are $\AAA$-paths and $\INV\uu \pc \vv$ is right-$\RC$-reversible to an empty path, then $\uu$ and~$\vv$ are $\RRR$-equivalent, that is, they represent the same element in the category~$\PRESp\AAA\RRR$. In our context, right-reversing will be useful only when the previous implication is an equivalence.

\begin{defi}
If $(\AAA ; \RRR)$ is a presentation associated with a right-complement~$\RC$, we say that right-reversing is \emph{complete} for $(\AAA ; \RRR)$ if $\INV\uu \pc \vv \rev_\RC \ew$ holds whenever $\uu$ and~$\vv$ represent the same element in~$\PRESp\AAA\RRR$.
\end{defi}

In other words, right-reversing is complete if it always detects equivalence. The interest of introducing completeness here is the following easy result:

\begin{lemm}
Assume that $(\AAA ; \RRR)$ is a presentation associated with a right-comple\-ment~$\RC$ and right-reversing is complete for $(\AAA ; \RRR)$. Then the category~$\PRESp\AAA\RRR$ is left-cancellative.
\end{lemm}

\begin{proof}
It is sufficient to prove that, if $\aa$ belongs to~$\AAA$ and $\uu, \vv$ are $\AAA$-paths such that $\aa \pc \uu$ and~$\aa \pc \vv$ are $\RRR$-equivalent, then $\uu$ and~$\vv$ are $\RRR$-equivalent. Now, as right-reversing is complete, the hypothesis implies that $\INV\uu \pc \INV\aa \pc \aa \pc \vv$ right-$\RC$-reverses to an empty path. Now the first step in the reversing process necessarily consists in deleting~$\INV\aa \pc \aa$. We deduce that $\INV\uu \pc \vv$ must right-$\RC$-reverse to an empty path, which, by Lemma~\ref{L:Equiv}, implies that $\uu$ and~$\vv$ are $\RRR$-equivalent.
\end{proof}

We are thus led to looking for completeness criteria. 

\begin{defi}
Assume that $\RC$ is a right-complement on a precategory~$\AAA$. For $\aa, \bb, \cc$ in~$\AAA$, we say that $\RC$ satisfies the \emph{cube condition} at~$(\aa, \bb, \cc)$ if either neither of $\RC^*(\RC^*(\aa, \bb), \RC^*(\aa, \cc))$ and $\RC^*(\RC^*(\bb, \aa), \RC^*(\bb, \cc))$ is defined, or both are defined and 
$$\RC^*(\RC^*(\RC^*(\aa, \bb), \RC^*(\aa, \cc)), \RC^*(\RC^*(\bb, \aa), \RC^*(\bb, \cc)))$$
is empty.
\end{defi}

\begin{prop}
\label{P:Complete}
\cite[Proposition~II.4.11]{Garside} or \cite[Proposition~6.1]{Dgk}
Assume that $(\AAA ; \RRR)$ is a presentation associated with a right-com\-ple\-ment~$\RC$ and, moreover, the right-complement~$\RC$ is short, or the presentation $(\AAA ; \RRR)$ is right-Noetherian. Then right-reversing is complete for $(\AAA ; \RRR)$ if and only if, for all pairwise distinct $\aa, \bb, \cc$ in~$\AAA$ with the same source, the cube condition for~$(\aa, \bb, \cc)$ is satisfied.
\end{prop}

\begin{coro}
\label{C:Cancel}
If the equivalent conditions of Proposition~\ref{P:Complete} are satisfied, the category~$\PRESp\AAA\RRR$ is left-cancellative.
\end{coro}

\begin{exam}
\label{X:Complete}
Both presentations $(\tta, \ttb ; \tta\ttb = \ttb\tta)$ and $(\tta, \ttb ; \tta\ttb\tta = \ttb\tta\ttb)$ are eligible for Proposition~\ref{P:Complete} and Corollary~\ref{C:Cancel}. Indeed, the former is associated with a short right-complement, whereas the latter is associated with a right-complement that is not short, but defining $\wit(\ww)$ to be the length of~$\ww$ and noting that the (unique) relation of the presentation consists of two words with the same length, Lemma~\ref{L:Noeth} yields that the presentation is right-Noetherian. Then, in order to apply Proposition~\ref{P:Complete} and Corollary~\ref{C:Cancel}, we need to check that the cube condition is satisfied for all triples of pairwise distinct elements of~$\{\tta, \ttb\}$, which is vacuously true. We conclude that both monoids $\PRESp{\tta, \ttb}{\tta\ttb = \ttb\tta}$ (the free Abelian monoid~$\NNNN^2$) and $\PRESp{\tta, \ttb}{\tta\ttb\tta = \ttb\tta\ttb}$ (the $3$-strand braid monoid~$\BP3$) are left-cancellative.
\end{exam}

\begin{rema}
Proposition~\ref{P:Complete} does not exhaust all known types of presentations for which right-reversing is complete. For instance, the presentation $(\tta, \ttb ; \tta = \ttb \tta \ttb)$ of the Klein bottle monoid is not eligible for Proposition~\ref{P:Complete}, since the associated right-complement~$\RC$ is neither short nor right-Noetherian since no map~$\wit$ may satisfy $\wit(\tta) = \wit(\ttb \tta \ttb) > \wit(\tta \ttb)$. However right-reversing is complete and terminating for this presentation, as well as for a number of similar presentations~\cite{Die}.
\end{rema}

\subsection{Recognizing Garside}
\label{SS:RecGar}

Assuming that $(\AAA ; \RRR)$ is a right-complemented presentation and the category~$\PRESp\AAA\RRR$ has been shown to be left-cancellative using the method explained in Subsection~\ref{SS:Cancel}, our next task is to recognize that some subfamily is possibly a Garside family. Once again the task will turn out to be easy whenever right-reversing is complete for the considered presentation.

The main observation is that, in the above context, the category~$\PRESp\AAA\RRR$ admits local right-lcms, that is, any two elements that admit a common right-multiple admit a right-lcm.

\begin{lemm}
\label{L:Lcm}
Assume that $(\AAA ; \RRR)$ is a presentation associated with a right-compl\-ement~$\RC$ and right-reversing is complete for it. Then, for all $\AAA$-paths~$\uu, \vv$ with the same source, the elements~$\cl\uu, \cl\vv$ admit a common right-multiple in~$\PRESp\AAA\RRR$ if and only if the right-reversing of~$\INV\uu \pc \vv$ terminates, in which case we have $\cl\uu \lcm \cl\vv = \cl{\uu \pc \RC^*(\uu, \vv)}$ and $\cl\uu \under \cl\vv = \cl{\RC^*(\uu, \vv)}$.
\end{lemm}

\begin{proof}
\rightskip40mm
Assume that $\cl\uu$ and $\cl\vv$ admit a common right-multiple in the category~$\PRESp\AAA\RRR$. This means that there exist $\AAA$-paths~$\uu', \vv'$ such that $\uu\pc\vv'$ and~$\vv \pc \uu'$ are $\RRR$-equivalent. As right-reversing is complete, this implies that $\INV{\vv'}\pc \INV\uu \pc \vv \pc \uu'$ right-reverses to an empty path. Decompose the associated reversing grid as shown aside. The assumption that the reversing of $\INV{\vv'}\pc \INV\uu \pc \vv \pc \uu'$ successfully terminates implies in particular that the reversing of~$\INV\uu \pc \vv$ successfully terminates, that is, the paths $\RC^*(\uu, \vv)$ and $\RC^*(\vv, \uu)$ are defined. 
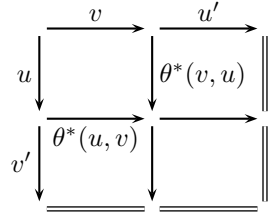
\hfill\begin{picture}(0,0)(-8,-5)
\pcline{->}(1,24)(14,24)
\taput{$\vv$}
\pcline{->}(16,24)(29,24)
\taput{$\uu'$}
\pcline{->}(1,12)(14,12)
\tbput{$\RC^*(\uu, \vv)$}
\pcline{->}(16,12)(29,12)
\pcline[style=double](1,0)(14,0)
\pcline[style=double](16,0)(29,0)
\pcline{->}(0,23)(0,13)
\tlput{$\uu$}
\pcline{->}(0,11)(0,1)
\tlput{$\vv'$}
\pcline{->}(15,23)(15,13)
\trput{$\RC^*(\vv, \uu)$}
\pcline{->}(15,11)(15,1)
\pcline[style=double](30,23)(30,13)
\pcline[style=double](30,11)(30,1)
\end{picture}

\rightskip0mm
The diagram then shows that $\cl{\uu\pc\vv'}$ is a right-multiple of $\cl{\uu \pc \RC^*(\uu, \vv)}$ in~$\PRESp\AAA\RRR$, so the latter, which depends only on~$\cl\uu$ and~$\cl\vv$, is a right-lcm of these elements.
\end{proof}

Applying the criterion of Lemma~\ref{L:RecGar}, we immediately deduce a method for recognizing Garside families in the right-Noetherian case. Actually, we obtain more: instead of just a YES/NO answer for a candidate-subfamily, we obtain the existence and a characterization of the smallest Garside family that includes the given family. Hereafter, for~$\BBB \subseteq \AAA^*$, we write $\cl\BBB$ for $\{\cl\ww \mid \ww \in \BBB\}$. 

\pagebreak

\begin{algo}{Smallest Garside family}
\label{A:SmallestGar}
\begin{algorithmic}[1]
\CONTEXT{A right-Noetherian right-complemented presentation $(\AAA ; \RRR)$ for which right-reversing is complete}
\INPUT{A finite subfamily~$\BBB$ of $\AAA^*$ that includes~$\AAA$}
\OUTPUT{A subfamily of $\AAA^*$ that represents the smallest Garside family of $\PRESp\AAA\RRR$ that includes~$\BBB \cup \Id\AAA$ if the algorithm terminates successfully}
\STATE{enumerate $\BBB $ as $\{\ww_1 \wdots \ww_\nn\}$ and set $\BBBh:=[\ww_1 \wdots \ww_\nn]$}
\STATE{$\ii := 2$}
\WHILE{$\ii\le|\BBBh|$}
\FOR{$\jj := 1$ to $\ii-1$}
\IF{$\RC^*(\ww_\ii, \ww_\jj)$ is defined}\COMMENT{right-reversing may fail to terminate}
\STATE{\CALL{Include}{$\BBBh$, $\RC^*(\ww_\ii, \ww_\jj)$}}\label{A:SmallestGar:compl-1}
\STATE{\CALL{Include}{$\BBBh$, $\RC^*(\ww_\jj, \ww_\ii)$}}\label{A:SmallestGar:compl-2}
\STATE{\CALL{Include}{$\BBBh$, $\ww_\ii \pc \RC^*(\ww_\ii, \ww_\jj)$}}\label{A:SmallestGar:lcm-1}
\STATE{\CALL{Include}{$\BBBh$, $\ww_\jj \pc \RC^*(\ww_\jj, \ww_\ii)$}}\label{A:SmallestGar:lcm-2}
\ELSE
\RETURN{$\mathtt{fail}$}
\ENDIF
\ENDFOR
\STATE{$\ii := \ii+1$}
\ENDWHILE
\RETURN{$\BBBh$}
\medskip

\PROCEDURE{Include}{$\BBBh$, $\ww$}
\COMMENT{append $\ww$ to $\BBBh$, unless $\BBBh$ contains a path equivalent to $\ww$.}
\IF{
$\not\!\exists \ii\in\{1 \wdots |\BBBh|\} \text{ with } \RC^*(\ww, \ww_\ii) = \RC^*(\ww_\ii, \ww) = \ew_-$}\label{A:SmallestGar:equiv}
\STATE{append $w$ to $\BBBh$}
\ENDIF
\ENDPROCEDURE
\end{algorithmic}
\end{algo}

\begin{prop}
\label{P:SmallestGar}
Assume that $(\AAA ; \RRR)$ is a finite right-Noetherian right-compl\-emented presentation for which right-reversing is complete and $\BBB$ is a finite subfamily of~$\AAA^*$ that includes~$\AAA$. Then there exists a smallest finite Garside family of~$\CCC$ including~$\cl\BBB \cup \Id\CCC$ if and only if Algorithm~\ref{A:SmallestGar} successfully terminates, in which case the returned family of paths represents that Garside family.
\end{prop}

\begin{proof}
By Lemma~\ref{L:Lcm}, any two elements of~$\PRESp\AAA\RRR$ that admit a common right-multiple admit a right-lcm, so $\PRESp\AAA\RRR$ is eligible for Lemma~\ref{L:RecGar}. Hence a subfamily~$\SSS$ that includes~$\AAA \cup \Id\AAA$ is a Garside family if and only if, for all~$\aa, \bb$ in~$\SSS$ with a common right-multiple, $\aa \lcm \bb$ and $\aa \under \bb$ belong to~$\SSS$. So, by Lemma~\ref{L:Lcm} again, a family~$\BBBh$ of $\AAA$-paths represents a Garside family if and only if, for all~$\uu, \vv$ in~$\BBBh$, the paths~$\uu \pc \RC^*(\uu, \vv)$ and $\RC^*(\uu, \vv)$ are $\RRR$-equivalent to at least one element of~$\BBBh$. It follows that there exists a smallest Garside family that includes~$\cl\BBB$, namely the smallest family of words that includes~$\BBB$ and is such that, for all~$\uu, \vv$ in~$\BBBh$, the paths $\uu \pc \RC^*(\uu, \vv)$ and $\RC^*(\uu, \vv)$ are $\RRR$-equivalent to elements of~$\BBBh$. 

That smallest subfamily~$\BBBh$ is precisely what Algorithm~\ref{A:SmallestGar} computes. Indeed, what the latter does is to consider systematically all pairs $(\ww_\ii, \ww_\jj)$ and, for each of them, test whether $\RC^*(\ww_\ii, \ww_\jj)$ (lines~\ref{A:SmallestGar:compl-1} and~\ref{A:SmallestGar:compl-2}) as well as $\ww_\ii \pc \RC^*(\ww_\ii, \ww_\jj)$ (lines~\ref{A:SmallestGar:lcm-1} and~\ref{A:SmallestGar:lcm-2}) are $\RRR$-equivalent to some existing path~$\ww_\ii$ of the list and, if not, append the missing paths to the list. Note that, as right-reversing is complete, two paths $\ww, \ww'$ are equivalent if and only if $\RC^*(\ww, \ww')$ and $\RC^*(\ww', \ww)$ exist and are empty (line~\ref{A:SmallestGar:equiv}).
\end{proof}

\begin{exam}
\label{X:BraidMinGar}
Consider again the presentation $(\tta, \ttb ; \tta\ttb\tta = \ttb\tta\ttb)$. Running Algorithm~\ref{A:SmallestGar} on the family $\{\tta, \ttb\}$ yields $\{\ew, \tta, \ttb, \tta\pc\ttb, \ttb\pc\tta, \tta\pc\ttb\pc\tta\}$, a family of words representing the well known smallest Garside family $\{1, \tta, \ttb, \tta\ttb, \ttb\tta, \Delta\}$ of the braid mon\-oid~$B_3^+$ that includes~$1$; here and everywhere in the sequel, we use $\Delta$ for~$\tta\ttb\tta$.
\end{exam}

\begin{rema}\label{R:SmallestGarNotTerminating}
If right-reversing is not always terminating, that is, if some elements of the considered category have no common right-multiple although they share the same source, Algorithm~\ref{A:SmallestGar} may never terminate. Even in the case when right-reversing is always terminating, it might happen that Algorithm~\ref{A:SmallestGar} does not terminate: for instance, in the case of $(\tta, \ttb ; \tta \ttb^2 = \ttb \tta)$, starting from $\{\tta, \ttb\}$, even Algorithm~\ref{A:SmallestGar} does not terminate in finite time: indeed, we have then $\RC^*(\tta, \ttb^\nn) = \ttb^{2\nn}$, and the family $\widehat{\{\tta, \ttb\}}$ is the infinite family $\{\ew, \tta\} \cup \{\ttb^{2^\nn} \mid \nn \ge 0\}$. However, it can be shown \cite[Lemme~1.9]{Dgk} that, when the closure under~$\under$ is finite, then the closure under~$\under$ and~$\lcm$ is necessarily finite. 
\end{rema}

\subsection{Further questions}
\label{SS:Further}

In the previous subsections, we showed how to establish left-cancell\-ativity and to recognize Garside families starting from a presentation. We now briefly address further relevant questions, namely recognizing Ore categories, establishing the existence of lcms, and recognizing strong and bounded Garside families. 

\subsubsection*{Establishing right-cancellativity} 

Right-reversing is not suitable here, and no practical method is known for establishing right-cancellativity directly. However an obvious solution is to consider the opposite category and the opposite presentation, that is, to switch left and right everywhere, and apply the previous methods. Equivalently, we can work with the initial presentation and apply the symmetric counterpart of right-reversing, naturally called \emph{left-reversing}: whereas right-reversing consists in iteratively replacing subpaths of the form~$\INV\aa \pc \bb$ with $\RC(\aa, \bb) \pc \INV{\RC(\bb, \aa)}$, left-reversing consists in replacing subpaths of the form~$\bb \pc \INV\aa$ with $\INV{\LC(\aa, \bb)} \pc {\LC(\bb, \aa)}$ when $\LC$ is a \emph{left-complement} on the considered precategory~$\AAA$, namely a partial map of~$\AAA^2$ to~$\AAA^*$ such that $\LC(\aa, \aa)$ is empty for every~$\aa$ in~$\AAA$ and that, if $\LC(\aa, \bb)$ is defined, then $\aa$ and $\bb$ have the same target, $\LC(\bb, \aa)$ is defined, and both $\seqq{\LC(\aa, \bb)}\bb$ and $\seqq{\LC(\bb, \aa)}\aa$ are defined and have the same source.  In terms of diagrams, this corresponds to constructing a grid starting from the bottom and the right, instead of from the top and the left. Then the counterpart of Proposition~\ref{C:Cancel} gives a criterion for establishing that a category~$\PRESp\AAA\RRR$ is right-cancellative.

\subsubsection*{Establishing the existence of common multiples}

Here two different methods can be used. If $(\AAA ; \RRR)$ is a presentation associated with a right-complement~$\RC$, then the existence of common right-multiples in~$\PRESp\AAA\RRR$ is directly connected with the termination of right-reversing since, as proved in Lemma~\ref{L:Lcm}, two elements~$\cl\uu$ and~$\cl\vv$ admit a common right-multiple (and even a right-lcm) in~$\PRESp\AAA\RRR$ if and only if the right-$\RC$-reversing of~$\INV\uu \pc \vv$ successfully terminates in finite time. We deduce the following sufficient condition:

\begin{prop}
If $(\AAA ; \RRR)$ is a presentation associated with a right-complement~$\RC$ and Algorithm~\ref{A:Termination} running on~$(\AAA ; \RRR)$ succeeds, any two elements of~$\PRESp\AAA\RRR$ with the same source admit a common right-multiple.
\end{prop}

Another approach can be used once a Garside family~$\SSS$ is known. It is based on the following result, which reduces the existence of common multiples for arbitrary elements to the existence of common multiples inside the Garside family.

\begin{prop}
\label{P:CommRMult}
Assume that $\SSS$ is a Garside family in a left-cancellative category~$\CCC$. Then any two elements of~$\CCC$ with the same source admit a common right-multiple if and only if any two elements of~$\SSS$ with the same source admit one.
\end{prop}

\begin{proof}
Obviously the condition is necessary. On the other hand, assume that any two elements of~$\SSS$ with the same source admit a common right-multiple. By Lemma~\ref{L:GarClosed}, for all~$\aa, \bb$ in~$\SSSs$ with the same source, there exist~$\aa', \bb'$ in~$\SSSs$ satisfying $\aa \bb' = \bb \aa'$. Now, consider arbitrary elements~$\ff, \gg$ of~$\CCC$ with the same source. As $\SSSs$ generates~$\CCC$, there exist $\aa_1 \wdots \aa_\pp$ and $\bb_1 \wdots \bb_\qq$ in~$\SSSs$ satisfying $\ff = \aa_1 \pdots \aa_\pp$ and $\gg = \bb_1 \pdots \bb_\qq$. Using the result above, one inductively constructs a $\pp \times \qq$ rectangular grid based on $\aa_1 \wdots \aa_\pp$ and $\bb_1 \wdots \bb_\qq$ with edges in~$\SSSs$, and the diagonal of the grid (as well as any path from the top--left corner to the bottom--right corner) represents a common right-multiple of~$\ff$ and~$\gg$.
\end{proof}

Note that the construction of a grid in the proof of Proposition~\ref{P:CommRMult} is directly reminiscent of a right-reversing process---more exactly, of its non-deterministic extension alluded to in Remark~\ref{R:NonDet} as there is no uniqueness of the elements called~$\aa', \bb'$ in general.

It should be clear that Proposition~\ref{P:CommRMult} directly leads to an effective method for deciding the existence of common right-multiples in the case of a finite Garside family, see Example~\ref{X:BraidFinal} below.

A symmetric argument is possible for common left-multiples. However, as the definition of a Garside family is not invariant under exchanging left and right, the result takes a different form. In particular, it only gives a sufficient condition that need not be necessary in general.

\begin{prop}
\label{P:CommLMult}
Assume $\SSS$ is a Garside family in a left-cancellative category~$\CCC$ and, for all~$\aa, \bb$ in~$\SSSs$ with the same target, there exist~$\aa', \bb'$ in~$\SSSs$ satisfying $\aa' \bb = \bb' \aa$. Then any two elements of~$\CCC$ with the same target admit a common left-multiple.
\end{prop}

The proof is symmetric to that of Proposition~\ref{P:CommRMult}, the inductive step consisting now in constructing a rectangular grid starting from the bottom and the right. Again Proposition~\ref{P:CommLMult} leads to an effective method for deciding the existence of common left-multiples in the case of a finite Garside family, see Example~\ref{X:BraidFinal}.

\subsubsection*{Establishing the existence of lcms}

As for the existence of right-lcms (their computation will be addressed in Subsection~\ref{SS:Lcm}), we shall just mention an algorithmically important consequence of Lemma~\ref{L:GarClosed}: 

\begin{prop}
\label{P:GarLcm}
Assume that $\SSS$ is a Garside family in a left-cancellative category~$\CCC$. Then $\CCC$ admits right-lcms (\resp local right-lcms) if and only if any two elements~$\aa, \bb$ of~$\SSSs$ with the same source (\resp that admit a common right-multiple in~$\SSSs$) admit a right-lcm~$\cc$ inside~$\SSSs$, this meaning that $\cc$ belongs to~$\SSSs$ and every common right-multiple of~$\aa$ and~$\bb$ that lies in~$\SSSs$ is a right-multiple of~$\cc$.
\end{prop}

\begin{proof}
By Lemma~\ref{L:GarClosed}, a common right-multiple of two elements of~$\SSSs$ must be a right-multiple of some common right-multiple that lies in~$\SSSs$. So it is enough to consider right-multiples lying in~$\SSSs$, and we deduce that two elements of~$\SSSs$ admit a right-lcm in~$\CCC$ if and only if they admit one inside~$\SSSs$. Then (the proof of) Proposition~\ref{P:CommRMult} enables one to go from~$\SSSs$ to products of elements of~$\SSSs$, that is, to arbitrary elements of~$\CCC$.
\end{proof}

We can thus establish the possible existence of right-lcms by exclusively inspecting right-multiples inside the considered Garside family.

\subsubsection*{Recognizing strong and bounded Garside families}

So far, we have obtained methods for possibly establishing that a category is left- or right-Ore. Deciding whether a given Garside family~$\SSS$ is strong is then easy when $\SSS$ is finite and an effective method is available for deciding, for all~$\ff, \gg, \hh$, whether~$\hh$ is a left-lcm of~$\ff$ and~$\gg$. Note that, by the counterpart of Lemma~\ref{L:Lcm}, such a method exists whenever the presentation is associated with a left-complement for which left-reversing is complete. Finally, deciding whether a finite Garside family is bounded is easy, as it only requires testing divisibility relations.

\begin{exam}
\label{X:BraidFinal}
We continue with the presentation $(\tta, \ttb;\tta\ttb\tta = \ttb\tta\ttb)$ of the mon\-oid~$\BP3$. In Example~\ref{X:BraidMinGar}, we identified the Garside family $\SSS = \{1, \tta, \ttb, \tta\ttb, \ttb\tta, \Delta\}$ (where we recall $\Delta$ is $\tta\ttb\tta$). It is clear that $\Delta$ is a right-multiple of every element of~$\SSS$. By Proposition~\ref{P:CommRMult}, we conclude that any two elements of~$\BP3$ admit a common right-multiple. Next, we easily check the existence of right-lcms inside~$\SSS$ and, by Proposition~\ref{P:GarLcm}, conclude that $\BP3$ admits right-lcms (alternatively one can invoke Lemma~\ref{L:Lcm} here as the monoid contains no nontrivial invertible element). 

Similarly, $\Delta$ is also a common left-multiple of each element of~$\SSSs$, which is $\SSS$, so, by Proposition~\ref{P:CommLMult}, we deduce that any two elements of~$\BP3$ admit a common left-multiple---whence a left-lcm by the symmetric counterpart of Lemma~\ref{L:Lcm}.

Finally, as every element of~$\SSS$ is both a left- and a right-divisor of~$\Delta$, we deduce that the Garside family~$\SSS$ is bounded by~$\Delta$. So, in particular, it is strong.
\end{exam}

\section{Recognizing Garside families, case of a germ}
\label{S:Germ}

We now consider the same questions as in Section~\ref{S:Pres}, namely establishing left-cancellativity and recognizing Garside families, when the ambient category is specified by giving either the complete multiplication table (Subsection~\ref{SS:Ext_mult}), or a germ, defined to be a fragment of the multiplication table that contains enough information to determine the latter unambiguously (Subsections~\ref{SS:Germ} and~\ref{S:RecGerm}). In this case, one obtains a compact method that enables one to treat both questions (left-cancellativity and Garside family) simultaneously.

\subsection{Using the complete multiplication table}
\label{SS:Ext_mult}

We first quickly consider the case of a finite category. Such a category can be specified by an exhaustive enumeration of its elements, its source and map functions, and its multiplication table, which all are finite data. All questions are then easy.

First, for a finite category with an explicit multiplication table, left-cancellativity can be decided by an exhaustive inspection. Next, the left-divisibility relation is finite, so checking the condition~\eqref{E:Greedy} is easy and, for every subfamily~$\SSS$ of~$\CCC$, one can construct a list of all $\SSS$-greedy paths of length two. \emph{A priori}, it is not clear that this is sufficient to recognize a Garside family, since the definition of the latter mentions no upper bound on the length of the considered decompositions. However, such a bound exists, which limits both the elements to be considered and the length of the candidate-decompositions.

\begin{lemm}\cite[Proposition~3.1]{Dif}
\label{L:LengthTwo}
A subfamily~$\SSS$ of a left-cancellative category~$\CCC$ is a Garside family if and only if $\SSSs$ generates~$\CCC$ and every element of~$(\SSSs)^2$ admits an $\SSS$-normal decomposition of length two.
\end{lemm}

We deduce

\begin{prop}
Assume that $\CCC$ is a finite category~$\CCC$ with $\nn$~elements, and that the source and target maps on $\CCC$ and the multiplication table of $\CCC$ are given. Then the following hold:

\ITEM1 It is decidable in time $O(n^2)$ whether $\CCC$ is left-cancellative.

\ITEM2 The left-divisibility relation can be computed in time $O(n^2)$.

\ITEM3 If $\CCC$ is left-cancellative and $\SSS$ is a subfamily of~$\CCC$, then it is decidable in time $O(n^6)$ whether $\SSS$ is a Garside family in $\CCC$.
\end{prop}

\begin{proof}
We can test whether $\CCC$ is left-cancellative and compute the left-divisibility relation as follows: For each $\ff$ in~$\CCC$, iterate over all $\gg$ in~$\CCC$, keeping track of the products $\ff\gg$ that occur; $\CCC$ is left-cancellative if and only if no such product occurs more than once, and the occurring products are precisely the right-multiples of $\ff$. Hence \ITEM1 and \ITEM2 hold.

Comparing $\ff\gg$ to $1_x$, where $x$ is the source of $\ff$ for all $\ff,\gg$ in~$\CCC$ is sufficient to determine $\CCCi$, and then $\SSSs$ can be obtained by computing all products $\aa\gg$ for $\aa$ in~$\SSS$ and $\gg$ in~$\CCCi$. Thus computing $\SSSs$ takes time $O(n^2)$.
To verify that $\SSSs$ generates~$\CCC$, one computes, for $i=2,3,...$, the set $(\SSSs)^i$ by considering all products $\gg\aa$ with $\gg$ in~$(\SSSs)^{i-1}$ and $\aa$ in~$\SSSs$ until this sequence stabilises, which happens after at most $n$ steps. Thus verifying that $\SSSs$ generates $\CCC$ takes time $O(n^3)$.
Finally, by Lemma~\ref{L:LengthTwo}, one can decide whether a subfamily~$\SSS$ of~$\CCC$ is a Garside family by checking for each element of~$(\SSSs)^2$ all possible $\SSS$-paths of length two against condition~\eqref{E:Greedy}; as each test of condition~\eqref{E:Greedy} involves $O(n^2)$ operations, this takes time $O(n^6)$, completing the proof of \ITEM3.
\end{proof}

Note that, if $\CCC$ is infinite, $\SSS$ being finite does not make recognizing $\SSS$-greediness decidable in general, as \eqref{E:Greedy} contains a universal quantification over an arbitrary element~$\ff$ of~$\CCC$ and, so, in that case, Lemma~\ref{L:LengthTwo} is of no use.

\begin{rema}
The above analysis partly extends to the case when the category~$\CCC$ is infinite but all operations are computable. For instance, if multiplication, viewed as a partial function of~$\CCC^2$ to~$\CCC$, is computable, then left-cancellativity is a $\Pi_1^1$ condition: by enumerating all triples $(\ff, \gg, \gg')$ with $\gg \not= \gg'$ and checking whether $\ff\gg$ and $\ff\gg'$ are equal, one finds a counter-example in finite time if one exists, but one never obtains an answer when the category is left-cancellative. Similarly, left-divisibility is a $\Sigma_1^1$ condition and, therefore, $\SSS$-greediness, which entails an additional existential quantification over an arbitrary element~$\ff$ of~$\CCC$, is a $\Pi^1_2$ condition.
\end{rema}

\subsection{Germs}
\label{SS:Germ}

For an infinite category, it is impossible to exhaustively enumerate the multiplication table. However, in good cases, it may happen that some finite fragment of the latter determines the category and provides methods for establishing properties of the latter. This is the germ approach that we now introduce.

\begin{defi}
\label{D:Germ}
A \emph{germ} is a triple~$(\SSS, \Id\SSS, \OP)$ where $\SSS$ is a precategory, $\Id\SSS$ is a subfamily of~$\SSS$ consisting of an element~$\id\xx$ with source and target~$\xx$ for each object~$\xx$, and $\OP$ is a partial map of~$\Seq\SSS2$ into~$\SSS$ that satisfies
\begin{gather}
\label{E:GEGerm1}
\BOX{\VR(3,1.5) if $\aa \OP \bb$ is defined, its source is the source of~$\aa$ and its target is the target of~$\bb$,}\\
\label{E:GEGerm2}
\BOX{$\id\xx \OP \aa = \aa = \aa \OP \id\yy$ hold for each~$\aa$ in~$\SSS(\xx, \yy)$,}\\
\label{E:GEGerm3}
\BOX{if $\aa \OP \bb$ and $\bb \OP \cc$ are defined, then $(\aa \OP \bb) \OP \cc$ is defined if and only if $\aa \OP (\bb \OP \cc)$ is, in which case they are equal.}
\end{gather}
The germ is called \emph{left-associative} if, for all~$\cc, \aa, \bb$ in~$\SSS$, it satisfies
\begin{equation}
\label{E:GEGermAssociative}
\BOX{if $(\aa \OP \bb) \OP \cc$ is defined, then $\bb \OP \cc$ is defined,}
\end{equation}
and it is called \emph{left-cancellative} if, for all~$\aa, \bb, \bb'$ in~$\SSS$, it satisfies
\begin{equation}
\label{E:GEGermCancellative}
\BOX{if $\aa \OP \bb$ and $\aa \OP \bb'$ are defined and equal, then $\bb = \bb'$ holds.}
\end{equation}
\end{defi}

We shall usually write~$\SSSg$ for a germ whose domain is~$\SSS$. Whenever $\SSS$ is a subfamily of a category~$\CCC$, one obtains a germ by considering the restriction of the product of~$\CCC$ to~$\SSS$, that is, the partial binary operation~$\OP$ on~$\SSS$ defined by $\aa \OP \bb = \aa\bb$ whenever $\aa\bb$ is defined in~$\CCC$ and it belongs to~$\SSS$. In the other direction, starting with a germ, we can always construct a category.

\begin{defi}
If $\SSSg$ is a germ, we denote by~$\Cat(\SSSg)$ the category~$\PRESp\SSS\ROP$, where $\ROP$ is the family of all relations~$\seqq{\aa}{\bb} = \aa \OP \bb$ with $\aa, \bb$ in~$\SSS$ and $\aa \OP \bb$ defined.
\end{defi}

When we start with a category~$\CCC$ and a subfamily~$\SSS$ of $\CCC$, it may or may not be the case that the induced germ on~$\SSS$ contains enough information to reconstruct the initial category~$\CCC$.

\begin{exam}
\label{X:Germ}
\rightskip40mm
Let us consider the braid monoid~$\BP3$ again. If we take $\AAA = \{1, \tta, \ttb\}$, the table of the induced germ is shown aside, and the derived monoid is the free monoid based on~$\tta, \ttb$, hence one that is not isomorphic to~$\BP3$.
\hfill\begin{picture}(0,0)(-12,-7)
\begin{tabular}{c|c@{\hspace{1.5ex}}c@{\hspace{1.5ex}}c}
$\OP$   
&$1$     
&$\tta$     
&$\ttb$\\
\hline
$1$ 
&$1$     
&$\tta$     
&$\ttb$\\
$\tta$   
&$\tta$\\
$\ttb$   
&$\ttb$
\end{tabular}
\end{picture}

\rightskip0mm
\noindent By contrast, when we start with $\SSS = \{1, \tta, \ttb, \tta\ttb, \ttb\tta, \Delta\}$, the table of the induced germ is less hollow, and the derived monoid is indeed isomorphic to~$\BP3$.
$$\begin{tabular}{c|c@{\hspace{1.5ex}}c@{\hspace{1.5ex}}c@{\hspace{1.5ex}}c@{\hspace{1.5ex}}c@{\hspace{1.5ex}}c}
$\OP$   
& $1$     
& $\tta$     
& $\ttb$     
& $\tta\ttb$   
& $\ttb\tta$ 
& $\Delta$ \\
\hline
$1$   
& $1$     
& $\tta$     
& $\ttb$     
& $\tta\ttb$   
& $\ttb\tta$ 
& $\Delta$ \\
$\tta$   
& $\tta$     
&        
& $\tta\ttb$   
&        
& $\Delta$ \\
$\ttb$   
& $\ttb$     
& $\ttb\tta$   
&        
& $\Delta$ \\
$\tta\ttb$ 
& $\tta\ttb$   
& $\Delta$ \\
$\ttb\tta$ 
& $\ttb\tta$
&
& $\Delta$ \\
$\Delta$  
& $\Delta$
\end{tabular}$$
\end{exam}

The good point is that, as in the above example, the germ induced by a Garside family always contains enough information to reconstruct the initial category:

\begin{lemm}\cite[Proposition 4.8]{Dif}
\label{L:Enough}
If $\SSS$ is a Garside family in a left-cancellative category~$\CCC$ and $\SSSg$ is the germ induced on~$\SSS$, the category~$\Cat(\SSSg)$ is isomorphic to~$\CCC$.
\end{lemm}

\subsection{Establishing left-cancellativity and recognizing Garside}
\label{S:RecGerm}

Lemma~\ref{L:Enough} precisely shows that using a germ to specify a category is relevant here: if we start with a good candidate for a Garside family, the germ will indeed define the category. 

\begin{defi}
A germ~$\SSSg$ is said to be a \emph{Garside germ} if $\SSS$ embeds in~$\Cat(\SSSg)$, the latter is left-cancellative, and (the image of)~$\SSS$ is a Garside family in that category.
\end{defi}

So, for instance, the germs in Example~\ref{X:Germ} are Garside germs. By contrast, below is an example of a germ that is not a Garside germ although it defines the ambient monoid.

\pagebreak

\begin{exam}
\rightskip45mm
Let $\MM = \PRESp{\tta, \ttb}{\tta\ttb = \ttb\tta, \tta^2 = \ttb^2}$, and $\SSS$ consist of $1, \tta, \ttb, \tta\ttb, \tta^2$. The germ~$\SSSg$ induced on~$\SSS$ is shown aside. It is left-associative and left-cancellative. The category (here the monoid) $\Cat(\SSSg)$ is (isomorphic to)~$\MM$, as the relations $\tta\pc\tta = \tta^2 = \ttb\pc\ttb$ and $\tta\pc\ttb = \tta\ttb = \ttb \pc \tta$ belong to the family~$\ROP$. However $\SSS$ is not a Garside family in~$\MM$, as $\tta^3$ admits no $\SSS$-normal decomposition: $\tta^2 \pc \tta$ is not $\SSS$-greedy as $\tta\ttb$ left-divides~$\tta^3$ but not~$\tta^2$, and $\tta\ttb \pc \ttb$ is not $\SSS$-greedy as $\tta^2$ left-divides~$\tta^3$ but not~$\tta\ttb$.
\hfill\begin{picture}(0,0)(-7,-16)
\begin{tabular}{c|c@{\hspace{1.5ex}}c@{\hspace{1.5ex}}c@{\hspace{1.5ex}}c@{\hspace{1.5ex}}c}
$\OP$   
&$1$     
&$\tta$     
&$\ttb$
&$\tta^2$
&$\tta \ttb$\\
\hline
$1$ 
&$1$     
&$\tta$     
&$\ttb$
&\VR(3.3,0)$\tta^2$
&$\tta \ttb$\\
$\tta$   
&$\tta$
&$\tta^2$
&$\tta \ttb$\\
$\ttb$   
&$\ttb$
&$\tta \ttb$
&$\tta^2$\\
$\tta^2$
&$\tta^2$\\
$\tta \ttb$
&$\tta \ttb$
\end{tabular}
\end{picture}
\end{exam}

What we do below is to develop algorithms to decide whether a (finite) germ is a Garside germ. Note that this includes establishing that the defined category is left-cancellative. Our method is based on a result of~\cite{Dif}.

\begin{defi}
Assume that $\SSSg$ is a germ. 

\ITEM1 We define the \emph{local left-divisibility} relation~$\diveS$ of~$\SSSg$ by saying that $\aa \diveS \bb$ holds if and only if there exists $\bb'$ in~$\SSS$ satisfying $\bb =\aa \bb'$. We write $\aa \divS \bb$ for the conjunction of $\aa \diveS \bb$ and $\bb \not\diveS \aa$, and we call a sequence $\aa_1 \wdots \aa_\nn$ in~$\SSS$ \emph{non-ascending} if $\aa_\ii\not\divS\aa_\jj$ holds for $1 \le \ii < \jj \le \nn$.

\ITEM2 For~$\seqq{\aa_1}{\aa_2}$ in~$\Seq\SSS2$, we put $\JJJ_{\SSS}(\aa_1, \aa_2) = \{\bb \in \SSS \mid \aa_1 \OP \bb \mbox{ is defined and } \bb \diveS \aa_2 \}$.
\end{defi}

\begin{prop}\cite[Proposition~5.9]{Dif}
\label{P:GERecGarGerm2}
A germ~$\SSSg$ is a Garside germ if and only if it is left-associative, left-cancellative, and if, for any $\aa_1, \aa_2$ in $\SSS$ there exists a $\diveS$-greatest element in~$\JJJ_{\SSS}(\aa_1, \aa_2)$ (that is, an element $\cc$ in~$\JJJ_{\SSS}(\aa_1, \aa_2)$ such that $\bb\diveS\cc$ holds for all $\bb$ in~$\JJJ_{\SSS}(\aa_1, \aa_2)$).
\end{prop}

\begin{coro}
\label{C:MaxJFunction}
If $\SSS$ is a Garside germ, if $\seqq{\aa_1}{\aa_2}$ is in~$\Seq\SSS2$,
if $\bb$ is a $\diveS$-greatest element in~$\JJJ_{\SSS}(\aa_1, \aa_2)$, and if $\cc$ is an element of $\SSS$ satisfying $\aa_2 = \bb\OP \cc$, then the path $\aa_1\bb \pc \cc$ is an $\SSS$-normal decomposition of $\aa_1 \aa_2$ in~$\Cat(\SSSg)$.
\end{coro}

\begin{algo}{Recognizing a Garside germ}
\label{A:RecGarGerm}
\begin{algorithmic}[1]
\INPUT{A finite germ $\SSSg$}
\OUTPUT{$\mathtt{true}$ if $\SSSg$ is a Garside germ, and $\mathtt{false}$ otherwise}

\STATE{$\mathtt{isLeftCancellative}$, $\diveS$ := \CALL{LeftDivisibility}{$\SSS$}}
\IF{not $\mathtt{isLeftCancellative}$ or not \CALL{IsLeftAssociative}{$\SSS$}}
 \RETURN{$\mathtt{false}$}
\ENDIF
\STATE{$\SSS'$ := \CALL{NonAscending}{$\SSS$, $\diveS$}}
\FOR{$\seqq{\aa_1}{\aa_2}$ in~$\Seq\SSS2$}
  \IF{not \CALL{JHasGreatestElement}{$\SSS'$, $\aa_1$, $\aa_2$, $\diveS$}}
   \RETURN{$\mathtt{false}$}
  \ENDIF
\ENDFOR
\RETURN{$\mathtt{true}$}
\medskip

\FUNCTION{IsLeftAssociative}{$\SSS$}
 \FOR{$\seqqq\aa\bb\cc$ in $\Seq\SSS3$}
    \IF{$(\aa\OP\bb)\OP\cc\isdef$ and not $\bb\OP\cc\isdef$}
     \RETURN{$\mathtt{false}$}
    \ENDIF
 \ENDFOR
\RETURN{$\mathtt{true}$}
\ENDFUNCTION
\medskip

\FUNCTION{LeftDivisibility}{$\SSS$}
\COMMENT{whether $\SSS$ is left-cancellative and, if it is, a table (denoted by $\diveS$) with the truth values $\TV{\aa\diveS \bb}$ of $\aa \diveS \bb$ for $\aa,\bb$ in $\SSS$}
 \FOR{$\aa\in\SSS$}
  \FOR{$\bb\in\SSS$}
   \STATE{$\TV{\aa\diveS \bb} := \mathtt{false}$}
  \ENDFOR
  \FOR{$\bb\in\SSS$}
   \IF{$\aa\OP\bb\isdef$}
    \IF{$\TV{\aa \diveS \aa\OP\bb}$}
     \RETURN{$\mathtt{false}$}
    \ELSE
     \STATE{$\TV{\aa \diveS \aa\OP\bb} := \mathtt{true}$}
    \ENDIF
   \ENDIF
  \ENDFOR
 \ENDFOR
 \RETURN{$\mathtt{true}$, $\diveS$}
\ENDFUNCTION
\medskip

\FUNCTION{NonAscending}{$\SSS$, $\diveS$}
\COMMENT{$\SSS$ as a non-ascending sequence}
 \STATE{$\SSS' := [\,]$}
 \FOR{$\aa \in\SSS$}
  \IF{$\exists \ii\in\{1 \wdots |\SSS'|\}$ with $\SSS'[\ii]\divS \aa$}
   \STATE{insert $\aa$ into $\SSS$ at position $\min(\{ \ii\in\{1 \wdots |\SSS'|\} \mid \SSS'[\ii]\divS\aa\})$}
  \ELSE
   \STATE{append $\aa$ to $\SSS'$}
  \ENDIF
 \ENDFOR
 \RETURN{$\SSS'$}
\ENDFUNCTION
\medskip

\FUNCTION{JHasGreatestElement}{$\SSS'$, $\aa_1$, $\aa_2$, $\diveS$}
\COMMENT{$\SSS'$ non-ascending}
 \STATE{$\cc := \bot$}
 \FOR{$\ii$ := 1 to $|\SSS'|$}
  \IF{$\aa_1\OP\SSS'[\ii]\isdef$ and $\SSS'[\ii]\diveS\aa_2$}
   \COMMENT{$\SSS'[\ii] \in \JJJ_\SSS(\aa_1, \aa_2)$}
   \IF{$\cc = \bot$}
    \STATE{$\cc := \SSS'[\ii]$}
   \ELSEIF{not $\SSS'[\ii] \diveS \cc$}
    \RETURN{$\mathtt{false}$}
   \ENDIF
  \ENDIF
 \ENDFOR
 \RETURN{$\mathtt{true}$}
\ENDFUNCTION
\end{algorithmic}
\end{algo}

\begin{prop}
Assume that $\SSSg$ is a finite germ with $|\SSS|=n$ and that the partial binary operation~$\OP$ can be computed in time $O(1)$. Then the following hold; cf.\ Algorithm~\ref{A:RecGarGerm}.

\ITEM1 The function $\mathtt{IsLeftAssociative}$ decides in time $O(n^3)$ whether $\SSSg$ is left-associative.

\ITEM2 The function $\mathtt{LeftDivisibility}$ decides whether $\SSSg$ is left-cancella\-tive and, if it is, computes the left-divisibility relation~$\diveS$ on~$\SSS$ with respect to~$\OP$ in time $O(n^2)$.

\ITEM3 Given the left-divisibility relation~$\diveS$, the function $\mathtt{NonAscending}$ computes a non-ascend\-ing sequence containing the elements of~$\SSS$ in time $O(n^2)$.

\ITEM4 Given $\seqq{\aa_1}{\aa_2}$ in~$\Seq\SSS2$, a non-ascending sequence~$\SSS'$ containing the elements of~$\SSS$, and the left-divisibility relation~$\diveS$, the function $\mathtt{JHasGreatestElement}$ decides in time $O(n)$ whether $\JJJ_{\SSS}(\aa_1, \aa_2)$ has a $\diveS$-greatest element.

\ITEM5 It is decidable in time $O(n^3)$ whether $\SSS$ is a Garside germ.
\end{prop}

\begin{proof}
Claims \ITEM1 and \ITEM2 are obvious from the pseudocode in Algorithm~\ref{A:RecGarGerm} and the definitions.

The sequence $\SSS'$ constructed in the function $\mathtt{NonAscending}$ is non-ascending at any time by induction: If there is no $\ii$ satisfying $\SSS'[k]\divS\aa$, the induction step is trivial. Otherwise, one has $\SSS'[\jj]\not\divS\aa$ for all $\jj<\ii$ by the choice of $\ii$, and $\aa\not\divS\SSS'[j]$ for $\jj\ge\ii$, as $\SSS'[k]\divS\aa\divS\SSS'[j]$ would contradict the induction hypothesis. It is clear from the pseudocode that the time complexity is $O(n^2)$, so claim~\ITEM3 holds.

The function $\mathtt{JHasGreatestElement}$ tests the elements of $\SSS$ for membership in $\JJJ_{\SSS}(\aa_1, \aa_2)$ in non-ascending order. Moreover, the set $\JJJ_{\SSS}(\aa_1, \aa_2)$ is non-empty, as it contains $\ew_\xx$, where~$\xx$ is the target of~$\aa_1$. Hence, the set $\JJJ_{\SSS}(\aa_1, \aa_2)$ has a $\diveS$-greatest element if and only if the \emph{first} encountered element is an upper bound; this is what the function $\mathtt{JHasGreatestElement}$ tests. It is clear from the pseudocode that the time complexity is $O(n)$, so claim~\ITEM4 holds.

Claim~\ITEM5 then follows with Proposition~\ref{P:GERecGarGerm2}.
\end{proof}

\begin{exam}
\label{X:BraidGerm}
We apply Algorithm~\ref{A:RecGarGerm} to the second germ of Example~\ref{X:Germ}, that is, we have $\SSS = \{1,\tta,\ttb,\tta\ttb,\ttb\tta,\Delta\}$ and $\OP$ is the partial binary operation on~$\SSS$ induced by the multiplication in~$B_3^+$.

One readily verifies that the germ is left-associative and left-cancellative, and computes the left-divisibility relation~$\diveS$ which is given in the left table below.

We obtain $\SSS' = (\Delta, \tta\ttb, \ttb\tta, \ttb, \tta, 1)$ and non-ascending sequences describing the sets $\JJJ_{\SSS}(\aa_1, \aa_2)$ given in the right table. For each set $\JJJ_{\SSS}(\aa_1, \aa_2)$, the first listed element is a maximum, showing that $\SSS$ is a Garside germ.

\begin{center}
\hfill
\begin{tabular}{c|c@{\hspace{1.5ex}}c@{\hspace{1.5ex}}c@{\hspace{1.5ex}}c@{\hspace{1.5ex}}c@{\hspace{1.5ex}}c}
$\diveS$   & $1$ & $\tta$ & $\ttb$ & $\tta\ttb$ & $\ttb\tta$ & $\Delta$ \\
\hline
$1$   & $\divesmall$ & $\divesmall$ & $\divesmall$ & $\divesmall$   & $\divesmall$   & $\divesmall$ \\
$\tta$   &    & $\divesmall$ &    & $\divesmall$   &      & $\divesmall$ \\
$\ttb$   &    &    & $\divesmall$ &      & $\divesmall$   & $\divesmall$ \\
$\tta\ttb$ &    &    &    & $\divesmall$   &      & $\divesmall$ \\
$\ttb\tta$ &    &    &    &      & $\divesmall$   & $\divesmall$ \\
$\Delta$  &    &    &    &      &      & $\divesmall$
\end{tabular}
\hfill
\begin{tabular}{c|c@{\hspace{1ex}}c@{\hspace{1ex}}c@{\hspace{1ex}}c@{\hspace{1ex}}c@{\hspace{1ex}}c}
$\JJJ_\SSS$ & $1$   & $\tta$    & $\ttb$    & $\tta\ttb$
     & $\ttb\tta$      & $\Delta$       \\
\hline\rule[9pt]{0pt}{0pt}
$1$    & $(1)$  & $(\tta,1)$ & $(\ttb,1)$ & $(\tta\ttb,\tta,1)$
     & $(\ttb\tta,\ttb,1)$ & $\SSS'$        \\
$\tta$   & $(1)$  & $(1)$   & $(\ttb,1)$ & $(1)$
     & $(\ttb\tta,\ttb,1)$ & $(\ttb\tta,\ttb,1)$ \\
$\ttb$   & $(1)$  & $(\tta,1)$ & $(1)$   & $(\tta\ttb,\tta,1)$
     & $(1)$        & $(\tta\ttb,\tta,1)$ \\
$\tta\ttb$ & $(1)$  & $(\tta,1)$ & $(1)$   & $(\tta,1)$
     & $(1)$        & $(\tta,1)$     \\
$\ttb\tta$ & $(1)$  & $(1)$   & $(\ttb,1)$ & $(1)$
     & $(\ttb,1)$     & $(\ttb,1)$     \\
$\Delta$  & $(1)$  & $(1)$   & $(1)$   & $(1)$
     & $(1)$        & $(1)$
\end{tabular}
\hfill\null
\end{center}
\end{exam}

\subsection{Further questions}
\label{SS:GermMore}

From that point, results are similar to those of Subsection~\ref{SS:Further}: once a germ has been shown to be a Garside germ, it is known to be a Garside family in its ambient category, and, for instance, Propositions~\ref{P:CommRMult} and~\ref{P:CommLMult} directly apply for the existence of common multiples. The only point worth mentioning here is right-cancellativity: let us observe that a germ is a symmetric structure, so applying the method of Subsection~\ref{S:RecGerm} to the opposite of the considered germ~$\SSSg$, that is, $(\SSS,\widetilde{\OP})$ with $\widetilde{\OP}$ defined by $\aa \mathbin{\widetilde{\OP}} \bb = \bb \OP \aa$, directly leads to a right-cancellativity criterion.

\section{Computing with a Garside family, positive case}
\label{S:Pos}

We now turn to a different series of questions: here we no longer aim at deciding whether the ambient category and a candidate family are relevant for the Garside approach, but we assume that they are and we aim at computing in the context so obtained. In this section, we consider computations that take place inside the considered category, and postpone to the next section the computations that take place in the groupoid of fractions. We shall successively consider computing normal decompositions (Subsections~\ref{SS:NormalLeft} and~\ref{SS:NormalRight}), solving the Word Problem (Subsection~\ref{SS:WordPb}) and computing least common multiples (Subsection~\ref{SS:Lcm}).

\subsection{Computing normal decompositions from the left}
\label{SS:NormalLeft}

By definition, a Garside family gives rise to distinguished decompositions for the elements of the ambient category, namely $\SSS$-normal decompositions. The first natural algorithmic question is to determine an $\SSS$-normal decomposition of an element~$\gg$ specified by an arbitrary decomposition in terms of~$\SSSs$. So the question is, starting from an $\SSSs$-path~$\ww$, to find an equivalent $\SSS$-normal path. Here we shall use an incremental approach, starting from a solution in the case of length two paths. 

\begin{defi}
\label{D:Square}
\rightskip30mm
A subfamily~$\AAA$ of a left-cancellative category~$\CCC$ \emph{satisfies Property~$\Square$} if, for every~$\seqq{\aa_1}{\aa_2}$ in~$\Seq\AAA2$, there exists an $\AAA$-greedy decomposition $\seqq{\bb_1}{\bb_2}$ of~$\aa_1\aa_2$ with $\bb_1$ and $\bb_2$ in~$\AAA$. In this case, a map $\SW$ that chooses, for every~$\seqq{\aa_1}{\aa_2}$ in~$\Seq\AAA2$, a pair~$(\bb_1, \bb _2)$ as above is called a \emph{$\Square$-witness} on~$\AAA$.
\hfill
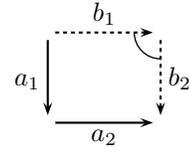
\begin{picture}(0,0)(-10,-2)
\pcline{->}(1,0)(14,0)
\tbput{$\aa_2$}
\pcline[style=exist]{->}(1,12)(14,12)
\taput{$\bb_1$}
\pcline{->}(0,11)(0,1)
\tlput{$\aa_1$}
\pcline[style=exist]{->}(15,11)(15,1)
\trput{$\bb_2$}
\psarc[style=thin](15,12){3.5}{180}{270}
\end{picture}
\end{defi}

If $\SSS$ is a Garside family, then one easily shows that there exists a $\Square$-witness on~$\SSSs$, that is, every element of~$(\SSSs)^2$ admits an $\SSS$-normal decomposition of length~$2$. Moreover, if the ambient category admits no nontrivial invertible element, then, by Proposition~\ref{P:NormalUnique}, the $\Square$-witness on~$\SSSs$ is unique. Note that, if the Garside family was initially specified as a germ, then Corollary~\ref{C:MaxJFunction} directly provides a $\Square$-witness; see Table~\ref{T:Witness} for an example. 

\begin{table}[htb]
\begin{tabular}{c|c@{\hspace{1.5ex}}c@{\hspace{1.5ex}}c@{\hspace{1.5ex}}c@{\hspace{1.5ex}}c@{\hspace{1.5ex}}c}
&$\ew$     
& $\tta$    
& $\ttb$    
& $\tta\ttb$   
& $\ttb\tta$ 
& $\Delta$\\
\hline
\VR(4,0)$\ew$   
&$(1, 1)$     
&$(\tta, 1)$     
&$(\ttb, 1)$     
&$(\tta\ttb, 1)$   
&$(\ttb\tta, 1)$ 
&$(\Delta, 1)$ \\
$\tta$   
&$(\tta, 1)$     
&$(\tta, \tta)$        
&$(\tta\ttb, 1)$   
&$(\tta, \tta\ttb)$   
&$(\Delta, 1)$
&$(\Delta, \ttb)$\\
$\ttb$   
&$(\ttb, 1)$     
&$(\ttb\tta, 1)$        
&$(\ttb, \ttb)$   
&$(\Delta, 1)$   
&$(\ttb, \ttb\tta)$
&$(\Delta, \tta)$\\
$\tta\ttb$   
&$(\tta\ttb, 1)$     
&$(\Delta)$        
&$(\tta\ttb, \ttb)$   
&$(\Delta, \ttb)$
&$(\tta\ttb, \ttb\tta)$
&$(\Delta, \ttb\tta)$\\
$\ttb\tta$   
&$(\ttb\tta, 1)$     
&$(\ttb\tta, \tta)$        
&$(\Delta, 1)$   
&$(\ttb\tta, \tta\ttb)$
&$(\Delta, \tta)$
&$(\Delta, \tta\ttb)$\\
$\Delta$   
&$(\Delta, 1)$     
&$(\Delta, \tta)$        
&$(\Delta, \ttb)$   
&$(\Delta, \tta\ttb)$
&$(\Delta, \ttb\tta)$
&$(\Delta, \Delta)$
\end{tabular}
\vspace{2mm}
\caption{\sf\smaller The $\Square$-witness on the Garside family $\{1, \tta, \ttb, \tta\ttb, \ttb\tta, \Delta\}$ in the braid monoid~$\BP3$: for every pair of elements~$(\aa_1, \aa_2)$, it specifies the unique normal decomposition of~$\aa_1\aa_2$ of length~$2$. The values directly follow from considering~$\SSS$ as a germ and using the maximal $\JJJ$-function of Example~\ref{X:BraidGerm}.}
\label{T:Witness}
\end{table}

\begin{algo}{Left-multiplication---see Figure~\ref{F:LeftMult}}
\label{A:LeftMult}
\begin{algorithmic}[1]
\CONTEXT{A Garside family~$\SSS$ in a left-cancellative category~$\CCC$, a $\Square$-witness~$\SW$ for~$\SSSs$}
\INPUT{An element~$\bb$ of~$\SSSs$ and an 
$\SSS$-normal decomposition~$\seqqq{\aa_1}\etc{\aa_\pp}$ of an element~$\gg$ of~$\CCC$ such that $\bb \gg$ exists}
\OUTPUT{An $\SSS$-normal decomposition of $\bb \gg$}
\STATE{$\bb_0 := \bb$}
\FOR{$\ii$ increasing from~$1$ to~$\pp$}
\STATE{$(\aa'_\ii, \bb_\ii) := \SW(\bb_{\ii-1}, \aa_\ii)$}
\ENDFOR
\RETURN{$\seqqqq{\aa'_1}\etc{\aa'_\pp}{\bb_\pp}$}
\end{algorithmic}
\end{algo}

\begin{figure}[htb]
\begin{picture}(45,18)(0,-2)
\pcline{->}(-14,0)(-1,0)
\tbput{$\bb$}
\pcline{->}(1,0)(14,0)
\tbput{$\aa_1$}
\pcline{->}(16,0)(29,0)
\tbput{$\aa_2$}
\psline[style=etc](33,0)(41,0)
\pcline{->}(46,0)(59,0)
\tbput{$\aa_\pp$}

\pcline{->}(1,12)(14,12)
\taput{$\aa'_1$}
\pcline{->}(16,12)(29,12)
\taput{$\aa'_2$}
\psline[style=etc](33,12)(41,12)
\pcline{->}(46,12)(59,12)
\taput{$\aa'_\pp$}

\pcline{->}(0,11)(0,1)
\trput{$\bb_0$}
\pcline{->}(15,11)(15,1)
\trput{$\bb_1$}
\pcline{->}(30,11)(30,1)
\trput{$\bb_2$}
\pcline{->}(45,11)(45,1)
\trput{$\bb_{\pp-1}$}
\pcline{->}(60,11)(60,1)
\trput{$\bb_\pp$}

\psarc[style=thin](15,12){3}{180}{270}
\psarc[style=thin](30,12){3}{180}{270}
\psarc[style=thin](60,12){3}{180}{270}
\psarc[style=thin](14.5,0){2.5}{180}{360}
\psarc[style=thin](29.5,0){2.5}{180}{360}
\psarc[style=thin](44.5,0){2.5}{180}{360}

\psline[style=double](-15,1)(-15,5)
\psarc[style=double](-8,5){7}{90}{180}
\psline[style=double](-8,12)(-1,12)
\end{picture}
\caption[]{\sf\smaller Algorithm~\ref{A:LeftMult}: starting from~$\bb$ in~$\SSSs$ and an $\SSS$-normal decomposition of~$\gg$, it returns an $\SSS$-normal decomposition of~$\bb \gg$.}
\label{F:LeftMult}
\end{figure}

\begin{prop}
\label{P:LeftMult}
Assume that $\SSS$ is a Garside family in a left-cancellative category~$\CCC$, and $\SW$ is a $\Square$-witness for~$\SSSs$. Then Algorithm~\ref{A:LeftMult} running on~$\bb$ and an $\SSS$-normal decomposition $\seqqq{\aa_1}\etc{\aa_\pp}$ of~$\gg$ returns an $\SSS$-normal decomposition of~$\bb\gg$. The function~$\SW$ is called $\pp$~times. 
\end{prop}

\begin{proof}
By construction, the diagram of Figure~\ref{F:LeftMult} is commutative. By assumption, $\seqq{\aa_\ii}{\aa_{\ii+1}}$ is $\SSS$-greedy for 
every~$\ii$, and, by the defining property of a $\Square$-witness, $\seqq{\aa'_\ii}{\bb_\ii}$ is $\SSS$-normal for every~$\ii$. Then the first domino rule (Lemma~\ref{L:Domino1}) implies that $\seqq{\aa'_\ii}{\aa'_{\ii+1}}$ is $\SSS$-normal. It is obvious from Algorithm~\ref{A:LeftMult} that the function~$\SW$ is invoked $\pp$~times.
\end{proof}

We can now compute $\SSS$-normal decompositions for an arbitrary element of $\CCC$ specified by an arbitrary $\SSSs$-decomposition by iterating Algorithm~\ref{A:LeftMult}.

\begin{algo}{Normal decomposition}
\label{A:Normal}
\begin{algorithmic}[1]
\CONTEXT{A Garside family~$\SSS$ in a left-cancellative category~$\CCC$, a $\Square$-witness~$\SW$ for~$\SSSs$}
\INPUT{An $\SSSs$-decomposition $\seqqq{\bb_1}\etc{\bb_\pp}$ of an element~$\gg$ of~$\CCC$}
\OUTPUT{An $\SSS$-normal decomposition of $\gg$}
\STATE{$\bb_{0,\ii} := \bb_\ii$ for $1 \le \ii \le \pp$}
\FOR{$\jj$ decreasing from~$\pp$ to~$1$}
\FOR{$\ii$ increasing from~$1$ to~$\pp - \jj$ (if any)}
\STATE{$(\aa_{\ii, \jj-1}, \bb_{\ii, \jj}) := \SW(\bb_{\ii-1, \jj}, \aa_{\ii, \jj})$}
\ENDFOR
\STATE{$\aa_{\pp-\jj+1, \jj-1} := \bb_{\pp-\jj, \jj}$}
\ENDFOR
\RETURN{$\seqqq{\aa_{1,0}}\etc{\aa_{\pp,0}}$}
\end{algorithmic}
\end{algo}

\begin{prop}
\label{P:Normal}
Assume that $\SSS$ is a Garside family in a left-cancellative category~$\CCC$ and $\SW$ is a $\Square$-witness for~$\SSSs$. Then Algorithm~\ref{A:Normal} running on an $\SSSs$-decomposition $\seqqq{\bb_1}\etc{\bb_\pp}$ of an element~$\gg$ of~$\CCC$ returns an $\SSS$-normal decomposition of~$\gg$. The map~$\SW$ is appealed to $\pp(\pp - 1)/2$~times.
\end{prop}

\begin{proof}
Consider the \textbf{for}-loop for $\jj$ and assume that before the execution of its body,
$\seqqq{\aa_{1,\jj}}\etc{\aa_{\pp-\jj,\jj}}$ is an $\SSS$-normal decomposition of $\bb_{0,\jj+1} \pdots \bb_{0,\pp}$, which is also $\bb_{\jj+1} \pdots \bb_{\pp}$; this condition is trivially satisfied for $\jj=\pp$.

According to Proposition~\ref{P:LeftMult}, the execution of lines 3-5 involves $\pp-\jj$ invocations of the map~$\SW$ and produces an $\SSS$-normal decomposition $\seqqq{\aa_{1,\jj-1}}\etc{\aa_{\pp-\jj+1,\jj-1}}$ of $\bb_{0,\jj} \pdots \bb_{0,\pp}$ that is, of $\bb_{\jj} \pdots \bb_{\pp}$. The claim then follows by induction.
\end{proof}

\begin{exam}
Consider the Garside family $\SSS = \{1, \tta, \ttb, \tta\ttb, \ttb\tta, \Delta\}$ and the element~$\gg = \ttb\tta\ttb^2$. We begin with the empty word, $\SSS$-normal decomposition of~$1$, and, starting from the right, we multiply by the successive letters on the left using Algorithm~\ref{A:LeftMult} and the $\Square$-witness of Table~\ref{T:Witness}, thus finding $\ttb$, $\seqq\ttb\ttb$, $\seqq{\tta\ttb}\ttb$ and, finally, $\seqq\Delta\ttb$ for the (unique) $\SSS$-normal decompositions of~$\ttb$, $\ttb^2$, $\tta\ttb^2$, and~$\gg$.
\end{exam}

\subsection{Computing normal decompositions using right-multiplication}
\label{SS:NormalRight}

Algorithm \ref{A:LeftMult} is not symmetric: the construction starts from the left and moves to the right. It is natural to wonder whether a symmetric right--to--left version exists. The answer depends on the considered Garside family~$\SSS$: if the latter satisfies the second domino rule (Lemma~\ref{L:Domino2}), the previous constructions have counterparts based on right-multiplication.

\begin{algo}{Right-multiplication---see Figure~\ref{F:RightMult}}
\label{A:RightMult}
\begin{algorithmic}[1]
\CONTEXT{A Garside family~$\SSS$ satisfying the second domino rule in a left-cancella\-tive category~$\CCC$, a $\Square$-witness~$\SW$ for~$\SSSs$}
\INPUT{An element~$\bb$ of~$\SSSs$, an $\SSS$-normal decomposition~$\seqqq{\aa_1}\etc{\aa_\pp}$ of an element~$\gg$ of~$\CCC$ such that ~$\gg\bb$ is defined}
\OUTPUT{An $\SSS$-normal decomposition of~$\gg\bb$}
\STATE{$\bb_\pp := \bb$}
\FOR{$\ii$ decreasing from~$\pp$ to~$1$}
\STATE{$(\bb_{\ii-1}, \aa'_\ii) := \SW(\aa_\ii, \bb_\ii)$}
\ENDFOR
\RETURN{$\seqqqq{\bb_0}{\aa'_1}\etc{\aa'_\pp}$}
\end{algorithmic}
\end{algo}

\begin{figure}[htb]
\begin{picture}(75,18)(0,-2)
\pcline{->}(1,0)(14,0)
\tbput{$\aa'_1$}
\pcline{->}(16,0)(29,0)
\tbput{$\aa'_2$}
\psline[style=etc](33,0)(41,0)
\pcline{->}(46,0)(59,0)
\tbput{$\aa'_\pp$}

\pcline{->}(1,12)(14,12)
\taput{$\aa_1$}
\pcline{->}(16,12)(29,12)
\taput{$\aa_2$}
\psline[style=etc](33,12)(41,12)
\pcline{->}(46,12)(59,12)
\taput{$\aa_\pp$}
\pcline{->}(61,12)(74,12)
\taput{$\bb$}

\pcline{->}(0,11)(0,1)
\trput{$\bb_0$}
\pcline{->}(15,11)(15,1)
\trput{$\bb_1$}
\pcline{->}(30,11)(30,1)
\trput{$\bb_2$}
\pcline{->}(45,11)(45,1)
\trput{$\bb_{\pp-1}$}
\pcline{->}(60,11)(60,1)
\trput{$\bb_\pp$}

\psarc[style=thin](0,0){3}{0}{90}
\psarc[style=thin](15,0){3}{0}{90}
\psarc[style=thin](45,0){3}{0}{90}
\psarc[style=thin](14.5,12){2.5}{0}{180}
\psarc[style=thin](29.5,12){2.5}{0}{180}
\psarc[style=thin](44.5,12){2.5}{0}{180}
\psline[style=double](75,7)(75,11)
\psarc[style=double](68,7){7}{270}{360}
\psline[style=double](61,0)(68,0)
\end{picture}
\caption[]{\sf\smaller Algorithm~\ref{A:RightMult}: starting from an $\SSS$-normal decomposition of~$\gg$ and~$\bb$, it returns an $\SSS$-normal decomposition of~$\gg \bb$.}
\label{F:RightMult}
\end{figure}

\begin{prop}[\bf right-multiplication]
\label{P:RightMult}
Assume that $\SSS$ is a Garside family satisfying the second domino rule in a left-cancellative category~$\CCC$ and $\SW$ is a $\Square$-witness for~$\SSSs$. Then Algorithm~\ref{A:RightMult} running on~$\bb$ in~$\SSSs$ and an $\SSS$-normal decomposition~$\seqqq{\aa_1}\etc{\aa_\pp}$ of~$\gg$ in~$\CCC$ returns an $\SSS$-normal decomposition of~$\gg\bb$, if the latter is defined. The function~$\SW$ is invoked $\pp$~times. 
\end{prop}

\begin{proof}
The commutativity of the diagram gives $\bb_0 \aa'_1 \pdots \aa'_\pp = \aa_1 \pdots \aa_\pp \bb$. Applying the second domino rule to each two-square subdiagram of the diagram of Figure~\ref{F:RightMult} starting from the right, we see that the sequence $\seqqqq{\bb_0}{\aa'_1}\etc{\aa'_\qq}$ is $\SSS$-greedy. As all entries lie in~$\SSSs$, the sequence is $\SSS$-normal.
\end{proof}

By iterating Algorithm~\ref{A:RightMult} (when applicable), we easily obtain a symmetric version of Algorithm~\ref{A:Normal} and Proposition~\ref{P:Normal}, that we shall not explicitly state.

\begin{rema}
The effect of the second domino rule is to shorten the computation of certain normal decompositions. Indeed, assume that $\seqqq{\bb_1}\etc{\bb_\qq}$ and  $\seqqq{\aa_1}\etc{\aa_\pp}$ are $\SSS$-normal paths and $\bb_\qq \aa_1$ is defined. By applying  Proposition~\ref{P:RightMult}, we can compute an $\SSS$-normal  decomposition of the product $\bb_1 \pdots \bb_\qq \aa_1 \pdots \aa_\pp$ by filling a diagram as in Figure~\ref{F:Domino12}. When valid, the second domino rule guarantees that the path consisting of the first $\pp$ top edges followed by $\qq$ vertical edges is $\SSS$-normal, that is, the triangular part of the diagram may be forgotten. 
\end{rema}

\begin{figure}[htb]
\begin{picture}(90,40)(0,-1)
\pcline{->}(1,36)(14,36)
\taput{$\aa'_1$}
\pcline[style=etc](16,36)(28,36)
\pcline{->}(31,36)(44,36)
\taput{$\aa'_\pp$}
\pcline{->}(46,36)(59,36)
\taput{$\aa'_{1+\pp}$}
\pcline[style=etc](61,36)(74,36)
\pcline{->}(76,36)(89,36)
\taput{$\aa'_{\qq+\pp}$}

\pcline{->}(1,24)(14,24)
\pcline[style=etc](16,24)(28,24)
\pcline{->}(31,24)(44,24)
\pcline{->}(46,24)(59,24)
\pcline[style=etc](61,24)(74,24)
\pcline[style=double](76,24)(83,24)
\psarc[style=double](83,31){7}{270}{360}
\pcline[style=double](90,31)(90,35)

\pcline{->}(1,12)(14,12)
\pcline[style=etc](16,12)(28,12)
\pcline{->}(31,12)(44,12)
\pcline{->}(46,12)(59,12)
\pcline[style=etc](61,12)(68,12)
\psarc[style=etc](68,19){7}{270}{360}
\pcline[style=etc](75,19)(75,23)

\pcline{->}(1,0)(14,0)
\tbput{$\aa_1$}
\pcline[style=etc](16,0)(28,0)
\pcline{->}(31,0)(44,0)
\tbput{$\aa_\pp$}
\pcline[style=double](46,0)(53,0)
\psarc[style=double](53,7){7}{270}{360}
\pcline[style=double](60,7)(60,11)

\pcline{->}(0,35)(0,25)
\tlput{$\bb_1$}
\pcline[style=etc](0,22)(0,14)
\pcline{->}(0,11)(0,1)
\tlput{$\bb_\qq$}

\pcline{->}(15,35)(15,25)
\pcline[style=etc](15,22)(15,14)
\pcline{->}(15,11)(15,1)

\pcline{->}(30,35)(30,25)
\pcline[style=etc](30,22)(30,14)
\pcline{->}(30,11)(30,1)

\pcline{->}(45,35)(45,25)
\trput{$\bb'_1$}
\pcline[style=etc](45,22)(45,14)
\pcline{->}(45,11)(45,1)
\trput{$\bb'_\qq$}

\pcline{->}(60,35)(60,25)
\pcline[style=etc](60,22)(60,14)

\pcline{->}(75,35)(75,25)

\psarc[style=thin](15,0){2.5}{180}{360}
\psarc[style=thin](30,0){2.5}{180}{360}

\psarc[style=thin](15,36){2.5}{0}{180}
\psarc[style=thin](30,36){2.5}{0}{180}
\psarc[style=thin](45,36){2.5}{0}{180}
\psarc[style=thin](60,36){2.5}{0}{180}
\psarc[style=thin](75,36){2.5}{0}{180}

\psarc[style=thin](0,24){2.5}{90}{270}
\psarc[style=thin](0,12){2.5}{90}{270}

\psarc[style=thinexist](45,24){2.5}{-90}{90}
\psarc[style=thinexist](45,12){2.5}{-90}{90}

\psarc[style=thin](15,36){3}{180}{270}
\psarc[style=thin](45,36){3}{180}{270}
\psarc[style=thin](60,36){3}{180}{270}
\psarc[style=thin](15,12){3}{180}{270}
\psarc[style=thin](45,12){3}{180}{270}
\end{picture}
\caption[]{\sf\smaller Finding an $\SSS$-normal decomposition of $\bb_1 \pdots \bb_\qq \aa_1 \pdots \aa_\pp$ when $\seqqq{\bb_1}\etc{\bb_\qq}$ and $\seqqq{\aa_1}\etc{\aa_\pp}$ are $\SSS$-normal: using Proposition~\ref{P:LeftMult}, hence the first domino rule only, one determines the $\SSS$-normal sequence $\seqqq{\aa'_1}\etc{\aa'_{\qq+ \pp}}$ in $\pp\qq + \qq(\qq-1)/2$ steps; if the second domino rule is valid, the sequence $\seqqq{\bb'_1}\etc{\bb'_\qq}$ is already $\SSS$-normal, and $\seqqqqqq{\aa'_1}\etc{\aa'_\pp}{\bb'_1}\etc{\bb'_\qq}$ is an $\SSS$-normal decomposition of $\bb_1 \pdots \bb_\qq \aa_1 \pdots \aa_\pp$.}
\label{F:Domino12}
\end{figure}

We conclude with an example of a Garside family that does not satisfy the second domino rule: its existence shows that no uniform result for multiplication on the right is possible.

\begin{exam}
\label{X:Domino2}
For $\nn \ge 2$, let $\MM_\nn = \PRESp{\tta, \ttb}{\tta \ttb^\nn = \ttb^{\nn+1}}$. The method of Subsection~\ref{SS:Cancel} shows that $\MM_\nn$ is left-cancellative, and that any two elements of~$\MM_\nn$ that admit a common right-multiple admit a right-lcm. On the other hand, $\MM_\nn$ is not right-cancellative since we have $\tta \ttb^{\nn-1} \not= \ttb^\nn$ and $\tta \ttb^\nn = \ttb^{\nn+1}$. Let $\SS_\nn = \{1, \tta, \ttb, \ttb^2 \wdots \ttb^{\nn+1}\}$, a subset of~$\MM_\nn$ with $\nn+3$ elements. Using Lemma~\ref{L:LengthTwo}, one can check that $\SS_\nn$ is a Garside family in~$\MM_\nn$. 

\rightskip48mm
Now the second domino rule is not valid for~$\SS_\nn$ in~$\MM_\nn$. Indeed, the diagram aside is commutative, the paths $\seqq{\tta}{\ttb}$ and $\seqq{\ttb^{\nn+1}}{\ttb}$ are $\SS_\nn$-greedy, and all edges corresponds to elements of~$\SS_\nn$. However $\seqq{\ttb}{\ttb}$ is not $\SS_\nn$-greedy since $\ttb^2$ lies in~$\SS_\nn$. 
\hfill
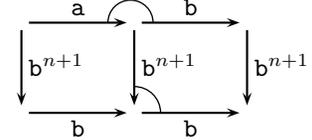
\begin{picture}(0,0)(-10,-3)
\psarc[style=thin](15,0){3.5}{0}{90}
\psarc[style=thin](14.5,12){3}{0}{180}
\pcline{->}(1,0)(14,0)
\tbput{$\ttb$}
\pcline{->}(16,0)(29,0)
\tbput{$\ttb$}
\pcline{->}(1,12)(14,12)
\taput{$\tta$}
\pcline{->}(16,12)(29,12)
\taput{$\ttb$}
\pcline{->}(0,11)(0,1)
\trput{$\ttb^{\nn+1}$}
\pcline{->}(15,11)(15,1)
\trput{$\ttb^{\nn+1}$}
\pcline{->}(30,11)(30,1)
\trput{$\ttb^{\nn+1}$}
\end{picture}

\end{exam}

\subsection{Solving the Word Problem}
\label{SS:WordPb}

We shall describe two solutions of the Word Problem for a left-cancellative category equipped with a Garside family, one based on normal decompositions, and one based on reversing.

\subsubsection*{Using $\SSS$-normal decompositions}

Whenever an effective method determining an $\SSS$-normal decomposition is available, a solution to the Word Problem is very close. However, due to the possible existence of nontrivial invertible elements, $\SSS$-normal decompositions need not be readily unique, so a final comparison step is needed.

\begin{defi}
\label{D:ETest}
Assume that $\CCC$ is a left-cancellative category and $\AAA$ is included in~$\CCC$. An \emph{$\eqir$-test} on~$\AAA$ is a map~$\EE$ of $\AAA^2$ to $\CCCi \cup \{\bot\}$ satisfying $\aa \EE(\aa, \bb) = \bb$ whenever $\aa, \bb$ are $\eqir$-equivalent, and $\EE(\aa, \bb) = \bot$ otherwise.
\end{defi}

\begin{algo}{Word Problem, positive case I}
\label{A:WordPb1}
\begin{algorithmic}[1]
\CONTEXT{A left-cancellative category~$\CCC$, a Garside subfamily~$\SSS$ of~$\CCC$, a $\Square$-witness~$\SW$ on~$\SSSs$, an $\eqir$-test~$\EE$ on~$\SSSs$}
\INPUT{Two $\SSSs$-paths $\uu, \vv$}
\OUTPUT{$\mathtt{true}$ if $\uu, \vv$ represent the same element of~$\CCC$, and $\mathtt{false}$ otherwise}
\IF{$\src\uu \not= \src\vv$ or $\trg\uu \not= \trg\vv$}
\RETURN{$\mathtt{false}$}
\ELSE
\STATE{use Algorithm~\ref{A:Normal} with~$\SW$ to find an $\SSS$-normal path $\seqqq{\aa_1}\etc{\aa_\pp}$ representing~$\cl\uu$}\label{A:WordPb1:path1}
\STATE{use Algorithm~\ref{A:Normal} with~$\SW$ to find an $\SSS$-normal path $\seqqq{\bb_1}\etc{\bb_\qq}$ representing~$\cl\vv$}\label{A:WordPb1:path2}
\RETURN{the value of \CALL{CompareNormalPaths}{$\seqqq{\aa_1}\etc{\aa_\pp}$, $\seqqq{\bb_1}\etc{\bb_\qq}$}}
\ENDIF
\medskip

\FUNCTION{CompareNormalPaths}{$\seqqq{\aa_1}\etc{\aa_\pp}$, $\seqqq{\bb_1}\etc{\bb_\qq}$}

\COMMENT{the paths should have the same source and the same target}
\STATE{$\xx := $ source of~$\bb_1$\,; $\yy := $ target of~$\bb_\qq$}
\STATE{$\ee_0 := \id\xx$\,; $\bb_\ii:= \id\yy$ for $\qq < \ii \le \max(\qq, \pp)$\,; $\aa_\jj:= \id\yy$ for $\pp < \jj \le \max(\qq, \pp)$}
\FOR{$k := 1 \text{ to}\max(\qq, \pp)$}\label{A:WordPb1:loop-start}
\IF{$\EE(\bb_\ii, \ee_{\ii-1}\aa_\ii) \not= \bot$}
\STATE{$\ee_\ii:= \EE(\bb_\ii, \ee_{\ii-1}\aa_\ii)$}
\ELSE
\RETURN{$\mathtt{false}$}
\ENDIF
\ENDFOR\label{A:WordPb1:loop-end}
\RETURN{$\TV{\ee_\ii = \id\yy}$}
\ENDFUNCTION
\end{algorithmic}
\end{algo}

\begin{prop}\label{P:WordPb1}
Assume that $\SSS$ is a Garside family in a left-cancellative category~$\CCC$, $\SW$ is a $\Square$-witness on~$\SSSs$, and $\EE$ is an $\eqir$-test on~$\SSSs$.

\ITEM1
Given two $\SSS$-normal paths $\uu$ and $\vv$ of length at most $\ell$, the function \textsc{CompareNormalPaths} in Algorithm~\ref{A:WordPb1} decides in time~$O(\ell)$ whether $\cl\uu=\cl\vv$ holds.

\ITEM2
Given two $\SSSs$-paths $\uu$ and $\vv$ of length at most $\ell$, Algorithm~\ref{A:WordPb1} decides in time~$O(\ell^2)$ whether $\cl\uu=\cl\vv$ holds.
\end{prop}

\begin{proof}
\ITEM1 By Proposition~\ref{P:NormalUnique}, two $\SSS$-normal paths~$\uu$ and~$\vv$ satisfy $\cl\uu=\cl\vv$ if and only if they are $\CCCi$-deformations of one another. The latter in turn is the case if and only if $\EE(\bb_\ii, \ee_{\ii-1}\aa_\ii) \neq\bot$ for all $\ii=1 \wdots \max(\pp, \qq)$ and $\ee_{\max(\pp, \qq)}=1_y$ hold in lines \ref{A:WordPb1:loop-start}-\ref{A:WordPb1:loop-end} of the function \textsc{CompareNormalPaths}. It is obvious that this test takes time~$O(\ell)$.

\ITEM2 If $\uu$ and $\vv$ have different sources or different targets, clearly we have $\cl\uu\ne \nobreak \cl\vv$. Otherwise, one has $\cl\uu=\cl\vv$ if and only if the paths computed in lines~\ref{A:WordPb1:path1} and~\ref{A:WordPb1:path2} of Algorithm~\ref{A:WordPb1} represent the same element. By Proposition~\ref{P:Normal}, lines~\ref{A:WordPb1:path1} and~\ref{A:WordPb1:path2} of Algorithm~\ref{A:WordPb1} have a cost of~$O(\ell^2)$ and the lengths of the produced paths is at most~$\ell$. The claim then follows with~\ITEM1.
\end{proof}

\begin{exam}
\label{X:BraidWordProblemViaNF}
Consider once again the 3-strand braid monoid~$\BP3$. Let $\SSS$ be $\{1, \tta, \ttb, \tta\ttb, \ttb\tta, \Delta\}$. We saw in Example~\ref{X:BraidMinGar} that $\SSS$ is a (minimal) Garside family in~$\BP3$ containing~$1$, and determined in Table~\ref{T:Witness} a $\Square$-witness on~$\SSS$. As $\BP3$ contains no nontrivial invertible element, testing for $\eqir$ reduces to testing equality.

As in Figure~\ref{F:RightRev}, consider $\uu = \tta \pc \ttb \pc \ttb$ and $\vv = \ttb \pc \tta \pc \ttb \pc \ttb$. Applying Algorithm~\ref{A:Normal}, we obtain the normal decomposition $\tta\ttb \pc \ttb$ of $\cl\uu$ and the normal decomposition $\Delta \pc \ttb$ of $\cl\vv$. As $\tta\ttb$ and $\Delta$ are not equal, hence not $\eqir$-equivalent, Algorithm~\ref{A:WordPb1} returns $\mathtt{false}$, that is, one has $\cl\uu\neq\cl\vv$.
\end{exam}

\subsubsection*{Using reversing}

An alternative solution of the Word Problem that does not require computing a normal form can be given using the reversing method of Subsection~\ref{SS:Rev}. By definition, if the investigated category~$\CCC$ is specified using a right-complemented presentation~$(\AAA ; \RRR)$ for which right-reversing is complete, then two $\AAA$-paths~$\uu, \vv$ represent the same element of~$\CCC$ if and only if the signed path~$\INV\uu \pc \vv$ reverses to an empty path, that is, if and only if Algorithm~\ref{A:RightRev} running on~$\INV\uu \pc \vv$ returns an empty path. This however provides a solution to the Word Problem only if right-reversing is known to terminate in finite time. By Lemma~\ref{L:Terminating}, this happens whenever any two elements with the same source admit a common right-multiple in the considered category, but nothing can be said in more general cases.

Now the termination problem does not arise when one considers a presentation that is associated with a short right-complement, in which case the Word Problem can be solved using reversing. The point here is that every Garside family gives rise to such a presentation:

\begin{lemm}
\label{L:Gar2Rev}
Under the assumptions of Proposition~\ref{P:Gar2Pres}, the obtained presentation $(\SSS ; \RRR)$ of~$\CCC$ is associated with a short right-complement, and right-reversing is complete for this presentation.
\end{lemm}

\begin{proof}
Let $\RC$ be the partial map defined on~$\SSS$ by $\RC(\aa, \bb) = \aa \under \bb$. By definition, $\RC$ is a right-complement on~$\SSS$ and, by Lemma~\ref{L:GarClosed}, it is short. Next, the associativity of the right-lcm operation implies that $\RC$ satisfies the cube condition for every triple of elements of~$\SSS$. Then Proposition~\ref{P:Complete}, implies that right-reversing is complete for~$(\SSS;\RRR)$.
\end{proof}

It follows that the presentation~$(\SSS ; \RRR)$ is eligible for right-reversing and that the latter solves the Word Problem.

\begin{algo}{Word Problem, positive case II}
\label{A:WordPb2}
\begin{algorithmic}[1]
\CONTEXT{A Garside family~$\SSS$ in a left-cancellative category~$\CCC$ that admits unique local right-lcms}
\INPUT{Two $\SSS$-paths $\uu, \vv$}
\OUTPUT{$\mathtt{true}$ if $\uu, \vv$ represent the same element of~$\CCC$, and $\mathtt{false}$ otherwise}
\STATE{$\RC$ := the right-lcm selector on~$\SSS \cup \Id\CCC$}
\STATE{$\mathtt{ret}$ := the return value of Algorithm~\ref{A:RightRevShort} for right-complement~$\RC$ and input $\INV\uu \pc \vv$}
\RETURN{$\TV{\mathtt{ret} = \ew_{\ud}}$}\label{A:WordPb2:return_success}
\COMMENT{$\mathtt{ret}$ is either a positive--negative path or equal to $\mathtt{fail}$}
\end{algorithmic}
\end{algo}

Lemmas~\ref{L:Gar2Rev} and~\ref{L:TerminatingShort} immediately imply:

\begin{prop}
\label{P:WordPb2}
Assume that $\SSS$ is a Garside family in a left-cancellative category~$\CCC$ that admits unique local right-lcms. 

\ITEM1 Algorithm~\ref{A:WordPb2} solves the Word Problem of~$\CCC$ with respect to~$\SSS$. 

\ITEM2 If $\SSS$ is finite, the complexity of Algorithm~\ref{A:WordPb2} is quadratic in the length of the input paths.
\end{prop}

\begin{exam}
\label{X:BraidGarToPres}
Consider once again the 3-strand braid monoid~$\BP3$. As seen in Example~\ref{X:BraidMinGar}, the five elements family $\{\tta, \ttb, \tta\ttb, \ttb\tta, \Delta\}$ is a (minimal) Garside family in~$\BP3$. The resulting presentation of~$\BP3$ is $(\tta, \ttb, \tta\ttb, \ttb\tta, \Delta;\RRR)$, where $\RRR$ consists of $\binom52$ relations (in which we leave the concatenation sign to avoid ambiguities)
$\tta \pc \ttb\tta = \ttb \pc \tta\ttb$, 
$\tta \pc \ttb = \tta\ttb$, 
$\tta \pc \ttb\tta = \ttb\tta \pc \ttb$,
$\tta \pc \ttb\tta = \Delta$,
$\ttb \pc \tta\ttb = \tta\ttb \pc \tta$,
$\ttb \pc \tta = \ttb\tta$,
$\ttb \pc \tta\ttb = \Delta$,
$\tta\ttb \pc \tta = \ttb\tta \pc \ttb$,
$\tta\ttb \pc \tta = \Delta$,
$\ttb\tta \pc \ttb = \Delta$.
Starting, as in Example~\ref{X:BraidWordProblemViaNF}, from the words $\uu = \tta \pc \ttb \pc \ttb$ and $\vv = \ttb \pc \tta \pc \ttb \pc \ttb$, we apply Algorithm~\ref{A:RightRevShort} for the input $\INV{\uu}\pc\vv$, obtaining $\ww = \tta\ttb \pc \INV{\tta}$ (see Figure~\ref{F:RightRevWord}) and thus, as $\ww$ is not empty, we conclude once again that $\uu$ and~$\vv$ are not equivalent.
\end{exam}

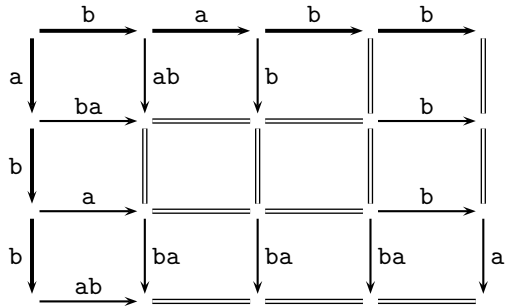
\begin{figure}[htb]
\begin{picture}(60,40)(0,0)
\pcline[style=thick]{->}(1,36)(14,36)
\taput{$\ttb$}
\pcline[style=thick]{->}(16,36)(29,36)
\taput{$\tta$}
\pcline[style=thick]{->}(31,36)(44,36)
\taput{$\ttb$}
\pcline[style=thick]{->}(46,36)(59,36)
\taput{$\ttb$}

\pcline[style=thick]{->}(0,35)(0,25)
\tlput{$\tta$}
\pcline{->}(15,35)(15,25)
\trput{$\tta\ttb$}
\pcline{->}(30,35)(30,25)
\trput{$\ttb$}
\pcline[style=double](45,35)(45,25)
\pcline[style=double](60,35)(60,25)

\pcline{->}(1,24)(14,24)
\taput{$\ttb\tta$}
\pcline[style=double](16,24)(29,24)
\pcline[style=double](31,24)(44,24)
\pcline{->}(46,24)(59,24)
\taput{$\ttb$}

\pcline[style=thick]{->}(0,23)(0,13)
\tlput{$\ttb$}
\pcline[style=double](15,23)(15,13)
\pcline[style=double](30,23)(30,13)
\pcline[style=double](45,23)(45,13)
\pcline[style=double](60,23)(60,13)

\pcline{->}(1,12)(14,12)
\taput{$\tta$}
\pcline[style=double](16,12)(29,12)
\pcline[style=double](31,12)(44,12)
\pcline{->}(46,12)(59,12)
\taput{$\ttb$}

\pcline[style=thick]{->}(0,11)(0,1)
\tlput{$\ttb$}
\pcline{->}(15,11)(15,1)
\trput{$\ttb\tta$}
\pcline{->}(30,11)(30,1)
\trput{$\ttb\tta$}
\pcline{->}(45,11)(45,1)
\trput{$\ttb\tta$}
\pcline{->}(60,11)(60,1)
\trput{$\tta$}

\pcline{->}(1,0)(14,0)
\taput{$\tta\ttb$}
\pcline[style=double](16,0)(29,0)
\pcline[style=double](31,0)(44,0)
\pcline[style=double](46,0)(59,0)
\end{picture}
\caption{\sf\small The grid associated with the right-reversing of $\INV{\uu}\pc\vv$ in Example~\ref{X:BraidGarToPres}; compare with the grid of Figure~\ref{F:RightRev}: the final paths are equivalent, but not equal, which is not contradictory as they are obtained in different ways.}
\label{F:RightRevWord}
\end{figure}

One of the benefits of the reversing approach is to provide at the same time a decision method for the left-divisibility relation, as Lemma~\ref{L:Gar2Rev} also implies:

\begin{prop}
If line~\ref{A:WordPb2:return_success} of Algorithm~\ref{A:WordPb2} is modified so that the algorithm returns $\mathtt{true}$ if $\mathtt{ret}$ is a positive path, and $\mathtt{false}$ otherwise, the modified algorithm returns $\mathtt{true}$ if and only if $\cl\uu$ left-divides~$\cl\vv$ in~$\CCC$.
\end{prop}

\begin{rema}
\label{R:NonDet2}
The above results extend to every left-cancellative category, possibly containing nontrivial invertible elements and not admitting local right-lcms, at the expense of considering a more general, non-deterministic version of reversing. But, as said in Remark~\ref{R:NonDet}, we do not consider such extensions here.
\end{rema}

\pagebreak

\subsection{Computing lcms}
\label{SS:Lcm}

\subsubsection*{Right-lcms} 

It follows from Lemma~\ref{L:Lcm} that, whenever a right-compl\-emented presentation~$(\AAA ; \RRR)$ is given for a left-cancellative category~$\CCC$ and right-reversing is complete for that presentation, then right-reversing provides a direct algorithm for computing right-lcms when they exist.

\begin{exam}
\label{X:Lcm}
Consider again the presentation $(\tta, \ttb ; \tta\ttb\tta = \ttb\tta\ttb)$ of the braid monoid~$\BP3$ and the elements~$\tta\ttb^2$ and~$\ttb\tta\ttb^2$. We saw in Example~\ref{X:Complete} that right-reversing is complete for this presentation. Now, we saw in Figure~\ref{F:RightRev} that the word $\INV\ttb \pc \INV\ttb \pc \INV\tta \pc \ttb \pc \tta \pc \ttb \pc \ttb$ is right-reversible to $\tta \pc \ttb \pc \INV\tta$ and, therefore, we conclude that the right-lcm of $\tta\ttb^2$ and~$\ttb\tta\ttb^2$ is $\tta\ttb^2 \cdot \tta\ttb$ (which is also $\ttb\tta\ttb^2 \cdot \tta$).
\end{exam}

Now, independently of the way the considered category~$\CCC$ was initially specified, if one happens to know a right-lcm selector~$\RC$ on a generating subfamily~$\AAA$, we can use the latter, which by definition is a right-complement, to right-reverse $\AAA$-paths. If the right-selector is not short, termination is not guaranteed but, if it is short, we immediately obtain:

\begin{prop}
\label{P:ComputLcm}
Assume that $\RC$ is a short right-lcm selector on some generating family~$\AAA$ in a left-cancellative category~$\CCC$ that admits local right-lcms. For all $\AAA$-paths~$\uu, \vv$, if Algorithm~\ref{A:RightRevShort} running on~$\INV\uu \pc \vv$ returns a positive--negative path~$\vv' \pc \INV{\uu'}$, then $\cl{\uu \pc \vv'}$ is a right-lcm of~$\cl\uu$ and~$\cl\vv$ in~$\CCC$; if it returns $\mathtt{fail}$, then $\cl\uu$ and~$\cl\vv$ admit no right-lcm in~$\CCC$. The complexity of the computation is in~$O((\LG\uu + \LG\vv)^2)$
\end{prop}

\begin{proof}
As observed in the proof of Lemma~\ref{L:Gar2Rev}, right-$\RC$-reversing must be complete in this case, and Lemma~\ref{L:Lcm} then gives the expected result. Alternatively, we can directly observe that, if $\gg_1\ff'$ is a right-lcm of~$\ff$ and~$\gg_1$ and $\gg_2\ff''$ is a right-lcm of~$\ff'$ and~$\gg_2$, then $\gg_1 \gg_2 \ff''$ is a right-lcm of~$\ff$ and~$\gg_1\gg_2$, which inductively implies that the diagonal of every rectangle in a right-reversing grid represents the right-lcm of the left and top edges.
\end{proof}

The point here is that, by Lemma~\ref{L:Gar2Rev}, every Garside family in a left-cancellative category that admits right-lcms gives rise to a short right-lcm selector, and therefore is eligible for Proposition~\ref{P:ComputLcm}. Moreover, by Proposition~\ref{P:GarLcm}, the existence of right-lcms in~$\CCC$ reduces to the existence of right-lcms inside~$\SSS$. It follows that, when $\SSSs$ is finite, one can effectively decide the existence of right-lcms for the elements of~$\SSSs$ and, in addition, obtain a right-lcm selector. 

\begin{exam}
\label{X:LcmBis}
Starting with the Garside family $\{\tta, \ttb, \tta\ttb, \ttb\tta, \Delta\}$ in~$\BP3$, the right-lcm selector~$\RC$ on~$\SSS$ leads to the presentation of Example~\ref{X:BraidGarToPres}. Considering as in Example~\ref{X:Lcm} the elements $\tta\ttb^2$ and~$\ttb\tta\ttb^2$, we find their right-lcm by right-$\RC$-reversing $\INV\ttb \pc \INV\ttb \pc \INV\tta \pc \ttb \pc \tta \pc \ttb \pc \ttb$. As seen in Figure~\ref{F:RightRevWord}, the latter word reverses to $\tta\ttb \pc \INV{\tta}$, and we conclude once again that the right-lcm of $\tta\ttb^2$ and~$\ttb\tta\ttb^2$ is $\tta\ttb^2 \cdot \tta\ttb$.
\end{exam}

\begin{rema}
The above approach consisting in using the right-reversing transformation associated with a (short or non-short) right-lcm selector~$\RC$ on some subfamily~$\AAA$ may work even when $\AAA$ is not a Garside family, provided $\AAA$ generates the ambient category~$\CCC$ and the family of all relations $\aa \RC(\aa, \bb) = \bb \RC(\bb, \aa)$ with $\aa, \bb$ in~$\AAA$ makes a presentation of~$\CCC$. Interestingly, this is always the case when $\CCC$ is right-Noetherian, but the example of the Dubrovina-Dubrovin $4$-strand braid monoid~\cite{Die} shows that the property may fail in the general case.
\end{rema}

\subsubsection*{Left-lcms}

The case of left-lcms is symmetric, and we shall not say much: by the counterpart of Proposition~\ref{P:ComputLcm}, if there exists a short left-lcm-selector~$\LC$ on a generating family~$\AAA$, then left-$\LC$-reversing a positive--negative path~$\vv \pc \INV\uu$ possibly leads to a negative--positive path~$\INV{\uu'} \pc \vv'$ such that $\uu' \pc \vv$ and~$\vv' \pc \uu$ represent the left-lcm of~$\cl\uu$ and~$\cl\vv$.

\pagebreak

\section{Computing with a strong Garside family, signed case}
\label{S:Sym}

In this final section, we show how to use (strong) Garside families to compute in groupoids of fractions. The successive problems we address are finding symmetric normal decompositions (Subsection~\ref{SS:Sym}), finding $\Delta$-normal decompositions (Subsection~\ref{SS:Delta}), solving the Word Problem (Subsection~\ref{SS:SignedWordPb}), computing inverses (Subsection~\ref{SS:Inv}) and, finally, computing lower and upper bounds (Subsection~\ref{SS:Bounds}).

\subsection{Computing symmetric normal decompositions}
\label{SS:Sym}

According to Proposition~\ref{P:SymExist}, if $\SSS$ is a Garside family in a left-Ore category~$\CCC$, then every element of~$\Env\CCC$, the groupoid of fractions of~$\CCC$, that can be represented as a right-fraction admits a symmetric $\SSS$-normal decomposition. Here we address the question of algorithmically computing such a decomposition. 

\subsubsection*{Starting from a right fraction}

Owing to the above recalled restriction on the considered elements (namely that of being expressible as a right-fraction), it is natural to first start with a positive--negative path. Then, Lemma~\ref{L:Disjoint} reduces the computation of a symmetric normal decomposition to the determination of a left-lcm: 

\enlargethispage{2mm}

\begin{algo}{Symmetric normal decomposition, positive--negative input}
\label{A:SymNormal}
\begin{algorithmic}[1]
\CONTEXT{A Garside family~$\SSS$ in a left-Ore category~$\CCC$ that admits left-lcms, a $\Square$-witness~$\SW$ on~$\SSSs$, a procedure computing left-lcms in~$\CCC$}
\INPUT{A positive--negative $\SSSs$-path $\vv \pc \INV\uu$}
\OUTPUT{A symmetric $\SSS$-normal decomposition of the element
$\cl\vv \cl\uu\inv$ of~$\Env\CCC$}
\STATE{find~$\uu', \vv'$ such that $\uu' \pc \vv$ and $\vv' \pc \uu$ represent a left-lcm of~$\cl\uu$ and~$\cl\vv$}
\STATE{find an $\SSS$-normal path $\seqqq{\bb_1}\etc{\bb_\qq}$ equivalent to~$\uu'$ using Algorithm~\ref{A:Normal} with~$\SW$}
\STATE{find an $\SSS$-normal path $\seqqq{\aa_1}\etc{\aa_\pp}$ equivalent to~$\vv'$ using Algorithm~\ref{A:Normal} with~$\SW$}
\RETURN{$\seqqqqqq{\INV{\bb_\qq}}\etc{\INV{\bb_1}}{\aa_1}\etc{\aa_\pp}$}
\end{algorithmic}
\end{algo}

\begin{prop}
\label{P:SymNormal}
Assume that $\CCC$ is a left-Ore category that admits left-lcms, $\SSS$ is a Garside family in~$\CCC$, $\SW$ is a $\Square$-witness on~$\SSSs$, and that left-lcms in~$\CCC$ can be computed effectively. Then Algorithm~\ref{A:SymNormal} returns a symmetric $\SSS$-normal decomposition of~$\cl\vv \cl\uu\inv$.
\end{prop}

\begin{proof}
By construction, the paths $\seqqq{\aa_1}\etc{\aa_\pp}$ and $\seqqq{\bb_1}\etc{\bb_\qq}$ are $\SSS$-normal, and $\bb_1 \pdots \bb_\qq \gg$, which is also $\aa_1 \pdots \aa_\pp \ff$, is a left-lcm of~$\ff$ and~$\gg$. By Lemma~\ref{L:Disjoint}, this implies that $\seqqqqqq{\INV{\bb_\qq}}\etc{\INV{\bb_1}}{\aa_1}\etc{\aa_\pp}$ is symmetric $\SSS$-normal. 
\end{proof}

\begin{exam}
\label{X:SymNormal}
Again in the case of the braid monoid~$\BP3$ as specified by the presentation $(\tta, \ttb ; \tta\ttb\tta = \ttb\tta\ttb)$ and the Garside family $\SSS = \{\tta, \ttb, \tta\ttb, \ttb\tta, \Delta\}$, let us consider the element $\gg = \tta^2\ttb\tta\inv\ttb\inv$. A decomposition of $\gg$ as a right-fraction is the signed $\AAA$-path $\tta\pc\tta\pc\ttb\pc\INV\tta\pc\INV\ttb$. By a symmetric argument to that used in Example~\ref{X:Complete}, the presentation $(\tta, \ttb ; \tta\ttb\tta = \ttb\tta\ttb)$ is left-complemented and left-reversing is complete for it. Hence left-reversing computes left-lcms. Here, left-reversing $\tta\pc\tta\pc\ttb\pc\INV\tta\pc\INV\ttb$ leads to $\INV\ttb \pc \INV\tta \pc \ttb \pc \tta \pc \tta$ (see Figure~\ref{F:LeftRev}, left), so $\tta\ttb \cdot \tta^2\ttb$, which is also $\ttb\tta^2 \cdot \ttb\tta$, is a left-lcm of~$\tta^2\ttb$ and~$\ttb\tta$. The $\SSS$-normal decompositions of~$\tta\ttb$ and $\ttb\tta^2$ respectively are $\tta\ttb$ and $\ttb\tta \pc \tta$. We conclude that $\INV{\tta\ttb} \pc \ttb\tta \pc \tta$ is a symmetric $\SSS$-normal decomposition of~$\gg$.
\end{exam}

Algorithm~\ref{A:SymNormal} is not fully satisfactory in that, as explained in Subsection~\ref{SS:Lcm}, there exists no general solution to the question of computing left-lcms. It may be that the ambient category has been specified by a left-complemented presentation~$(\AAA ; \RRR)$ for which left-reversing is complete, in which case left-reversing effectively computes left-lcms, but this need not be the case in general. Now, by (the counterpart of) Proposition~\ref{P:ComputLcm}, the natural context for computing left-lcms effectively is when a short left-lcm selector exists. 

\begin{lemm}
Assume that $\SSS$ is a Garside family in a left-Ore category~$\CCC$ that admits left-lcms. Then there exists a short left-lcm selector on~$\SSSs$ if and only if $\SSS$ is strong.
\end{lemm}

\begin{proof}
For~$\SSS$ to be strong means that, for all~$\aa, \bb$ with the same target in~$\SSSs$, there exist~$\aa', \bb'$ in~$\SSSs$ satisfying $\aa' \bb = \bb' \aa$ and such that $\aa'$ and $\bb'$ are left-disjoint. By Lemma~\ref{L:Disjoint}, $\aa' \bb$ is a left-lcm of~$\aa$ and~$\bb$. 
\end{proof}

We are thus led to restate Algorithm~\ref{A:SymNormal} in the case of a strong Garside family.

\begin{algo}{Symmetric normal decomposition, positive--negative input}
\label{A:SymNormal2}
\begin{algorithmic}[1]
\CONTEXT{A strong Garside family~$\SSS$ in a left-Ore category~$\CCC$ that admits left-lcms, a $\Square$-witness~$\SW$ and a short lcm-selector~$\LC$ on~$\SSSs$}
\INPUT{A positive--negative $\SSSs$-path $\vv \pc \INV\uu$}
\OUTPUT{A symmetric $\SSS$-normal decomposition of the element
$\cl\vv \cl\uu\inv$ of~$\Env\CCC$}
\STATE{left-$\LC$-reverse $\vv \pc \INV\uu$ into a negative--positive path~$\INV{\uu'}\pc \vv'$}
\STATE{find an $\SSS$-normal path $\seqqq{\bb_1}\etc{\bb_\qq}$ equivalent to~$\uu'$ using Algorithm~\ref{A:Normal} with~$\SW$}
\STATE{find an $\SSS$-normal path $\seqqq{\aa_1}\etc{\aa_\pp}$ equivalent to~$\vv'$ using Algorithm~\ref{A:Normal} with~$\SW$}
\RETURN{$\seqqqqqq{\INV{\bb_\qq}}\etc{\INV{\bb_1}}{\aa_1}\etc{\aa_\pp}$}
\end{algorithmic}
\end{algo}

\begin{prop}
\label{P:SymNormal2}
Assume that $\SSS$ is a strong Garside family in a left-Ore category~$\CCC$ that admits left-lcms, $\SW$ is a $\Square$-witness on~$\SSSs$, and $\LC$ is a short left-lcm selector on~$\SSSs$. Then Algorithm~\ref{A:SymNormal2} running on a positive--negative $\SSSs$-path~$\vv \pc \INV\uu$ returns a symmetric $\SSS$-normal decomposition of $\cl\vv \cl\uu\inv$ in time $O((\LG\uu + \LG\vv)^2)$.
\end{prop}

We skip the proof, which is similar to that of Proposition~\ref{P:SymNormal}.

\begin{exam}
\label{X:SymNormal2}
With the notation of Example~\ref{X:SymNormal}, the difference is that, in Algorithm~\ref{A:SymNormal2}, we use the left-reversing associated with the presentation~$(\SSS ; \widetilde\RRR)$ deduced from~$\LC$ rather than the initial presentation~$(\AAA ; \RRR)$. The result of left-reversing $\tta\pc\tta\pc\ttb\pc\INV\tta\pc\INV\ttb$ is now the length~$3$ word $\INV{\tta\ttb} \pc \ttb\tta \pc \tta$ (see Figure~\ref{F:LeftRev}, right), which is equivalent to $\INV\ttb \pc \INV\tta \pc \ttb \pc \tta \pc \tta$ as obtained previously (see Figure~\ref{F:LeftRev}, left). As $\tta\ttb$ and $\ttb\tta \pc \tta$ are $\SSS$-normal, the subsequent normalization steps change nothing, and we conclude again that $\INV{\tta\ttb} \pc \ttb\tta \pc \tta$ is a symmetric $\SSS$-normal decomposition of~$\gg$.
\end{exam}

\begin{figure}[htb]
\begin{picture}(105,28)(0,-2)
\pcline{->}(1,0)(14,0)
\tbput{$\tta$}
\pcline{->}(16,0)(29,0)
\tbput{$\tta$}
\pcline{->}(31,0)(44,0)
\tbput{$\ttb$}
\pcline{->}(1,6)(14,6)
\tbput{$\tta$}
\pcline[style=double](16,6)(29,6)
\pcline{->}(1,12)(7,12)
\taput{$\ttb$}
\pcline{->}(8,12)(14,12)
\taput{$\tta$}
\pcline[style=double](16,12)(29,12)
\pcline{->}(31,12)(37,12)
\taput{$\tta$}
\pcline{->}(38,12)(44,12)
\taput{$\ttb$}
\pcline{->}(1,24)(7,24)
\taput{$\ttb$}
\pcline{->}(8,24)(14,24)
\taput{$\tta$}
\pcline[style=double](16,24)(29,24)
\pcline{->}(31,24)(37,24)
\taput{$\tta$}
\pcline[style=double](38,24)(44,24)
\pcline[style=double](0,23)(0,13)
\pcline{->}(0,11.5)(0,9.5)
\tlput{$\tta$}
\pcline{->}(0,8.5)(0,6.5)
\tlput{$\ttb$}
\pcline[style=double](0,5)(0,1)
\pcline[style=double](0,23)(0,13)
\pcline[style=double](7.5,23)(7.5,13)
\pcline[style=double](15,23)(15,13)
\pcline[style=double](30,23)(30,13)
\pcline[style=double](37.5,23)(37.5,13)
\pcline{->}(45,23)(45,13)
\trput{$\ttb$}
\pcline{->}(15,11)(15,6.5)
\trput{$\ttb$}
\pcline[style=double](15,5)(15,1)
\pcline{->}(30,11)(30,6.5)
\trput{$\ttb$}
\pcline{->}(30,5.5)(30,1)
\trput{$\tta$}
\pcline{->}(45,11)(45,1)
\trput{$\tta$}

\pcline{->}(61,0)(74,0)
\tbput{$\tta$}
\pcline{->}(76,0)(89,0)
\tbput{$\tta$}
\pcline{->}(91,0)(104,0)
\tbput{$\ttb$}
\pcline{->}(61,12)(74,12)
\taput{$\ttb\tta$}
\pcline[style=double](76,12)(89,12)
\pcline{->}(91,12)(104,12)
\taput{$\tta\ttb$}
\pcline{->}(61,24)(74,24)
\taput{$\ttb\tta$}
\pcline[style=double](76,24)(89,24)
\pcline{->}(91,24)(104,24)
\taput{$\tta$}
\pcline[style=double](60,23)(60,13)
\pcline{->}(60,11)(60,1)
\tlput{$\tta\ttb$}
\pcline[style=double](75,23)(75,13)
\pcline{->}(75,11)(75,1)
\trput{$\ttb$}
\pcline[style=double](90,23)(90,13)
\pcline{->}(90,11)(90,1)
\trput{$\ttb\tta$}
\pcline{->}(105,23)(105,13)
\trput{$\ttb$}
\pcline{->}(105,11)(105,1)
\trput{$\tta$}
\end{picture}
\caption{\sf\smaller Left-reversing of the word $\tta\pc\tta\pc\ttb\pc\INV\tta\pc\INV\ttb$ using the left-complement on $\{\tta, \ttb\}$ (left) and the left-lcm selector on the Garside family $\{\tta, \ttb, \tta\ttb, \ttb\tta, \Delta\}$ (right). Here the right diagram corresponds to gathering pieces of the left diagram, but this would not necessarily be the case for more complicated examples.}
\label{F:LeftRev}
\end{figure}
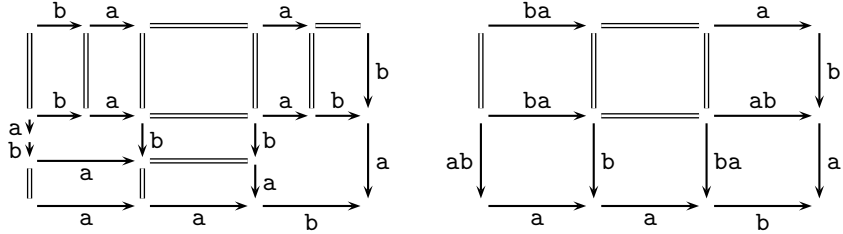

\begin{rema}
Algorithm~\ref{A:SymNormal2} (as well as Algorithm~\ref{A:SymNormal}) obeys the scheme ``left-reverse, then normalize''. We could as well use switch the operations into ``normalize, then left-reverse'', as suggested in Figure~\ref{F:Switch}. The property that guarantees the correctness of the switched version is the third domino rule (Lemma~\ref{L:Domino3}): left-reversing $\SSS$-normal paths always produces $\SSS$-normal paths. 

\rightskip45mm
For instance, in the context of Example~\ref{X:SymNormal2}, we could first find $\SSS$-normal words equivalent to~$\tta \pc \tta \pc \ttb$ and $\ttb \pc \tta$, namely $\tta \pc \tta\ttb$ and $\ttb\tta$, and then left-reverse the (length~$3$) signed word~$\tta \pc \tta\ttb \pc \INV{\ttb\tta}$. Of course, the result is $\INV{\tta\ttb} \pc \ttb\tta \pc \tta$ once more, but the involved grid is now the one given on the right.
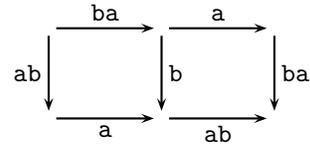
\hfill\begin{picture}(0,0)(-10,-4)
\pcline{->}(1,0)(14,0)
\tbput{$\tta$}
\pcline{->}(16,0)(29,0)
\tbput{$\tta\ttb$}
\pcline{->}(1,12)(14,12)
\taput{$\ttb\tta$}
\pcline{->}(16,12)(29,12)
\taput{$\tta$}
\pcline{->}(0,11)(0,1)
\tlput{$\tta\ttb$}
\pcline{->}(15,11)(15,1)
\trput{$\ttb$}
\pcline{->}(30,11)(30,1)
\trput{$\ttb\tta$}
\end{picture}
\end{rema}

\begin{figure}[htb]
\begin{picture}(105,37)(0,0)
\pcline[style=double](7,36)(14,36)
\pcline{->}(16,36)(44,36)
\taput{$\vv''$}
\pcline[style=double](1,24)(14,24)
\pcline{->}(16,24)(44,24)
\tbput{$\vv'$}
\pcline[style=double](1,0)(14,0)
\pcline{->}(16,0)(44,0)
\tbput{$\vv$}
\psarc[style=double](7,29){7}{90}{180}
\pcline[style=double](0,29)(0,25)
\pcline{->}(0,23)(0,1)
\tlput{$\uu''$}
\pcline[style=double](15,35)(15,25)
\pcline{->}(15,23)(15,1)
\trput{$\uu'$}
\pcline[style=double](45,35)(45,25)
\pcline{->}(45,23)(45,1)
\trput{$\uu$}
\psarc[style=thin](15,24){3}{270}{360}
\put(28,11){$\revLinv$}
\put(20,29){Algorithm~\ref{A:Normal}}
\put(6,1.5){\rotatebox{90}{\hbox{Algorithm~\ref{A:Normal}}}}

\pcline{->}(61,36)(89,36)
\taput{$\vv''$}
\pcline[style=double](91,36)(104,36)
\pcline{->}(61,12)(89,12)
\taput{$\vv'$}
\pcline[style=double](91,12)(104,12)
\pcline{->}(61,0)(89,0)
\tbput{$\vv$}
\pcline[style=double](91,0)(98,0)
\psarc[style=double](98,7){7}{270}{360}
\pcline[style=double](105,11)(105,7)
\pcline{->}(60,35)(60,13)
\tlput{$\uu''$}
\pcline{->}(90,35)(90,13)
\tlput{$\uu'$}
\pcline{->}(105,35)(105,13)
\trput{$\uu$}
\pcline[style=double](90,11)(90,1)
\pcline[style=double](60,11)(60,1)
\pcline{->}(45,23)(45,1)
\trput{$\uu$}
\psarc[style=thin](60,36){3}{270}{360}
\put(73,23){$\revLinv$}
\put(65,5){Algorithm~\ref{A:Normal}}
\put(96,13.5){\rotatebox{90}{\hbox{Algorithm~\ref{A:Normal}}}}
\end{picture}
\caption{\sf\smaller Switching normalization and left-reversing in Algorithm~\ref{A:SymNormal2}: on the left, we first left-reverse $\vv \pc \INV\uu$ into~$\INV{\uu'}\pc \vv'$ and then normalize~$\uu'$ into~$\uu''$ and $\vv'$ into~$\vv''$; on the right, we first normalize $\uu$ into~$\uu'$ and~$\vv$ into~$\vv'$ and then left-reverse $\vv' \pc \INV{\uu'}$ into $\INV{\uu''} \pc \vv''$.}
\label{F:Switch}
\end{figure}
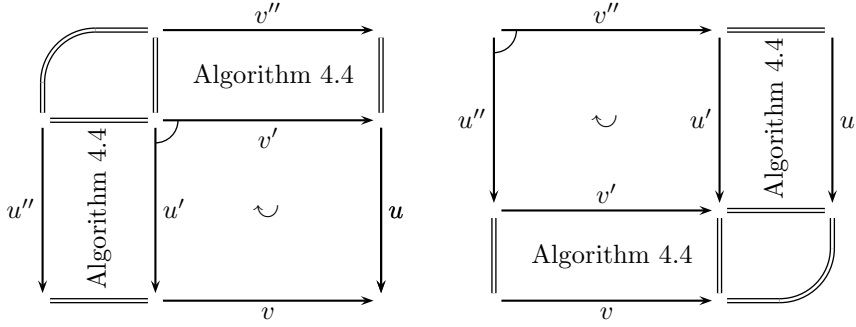

\subsubsection*{Starting from an arbitrary path}

Extending Algorithm~\ref{A:SymNormal2} to start with an arbitrary path rather than with a positive--negative path is easy. Indeed, we saw that every Garside family provides a short right-lcm selector and using the associated right-reversing transforms an arbitrary path into an equivalent positive--negative path, if some exists.

\enlargethispage{3mm}

\begin{algo}{Symmetric normal decomposition, general input}
\label{A:SymNormalGen}
\begin{algorithmic}[1]
\CONTEXT{A strong Garside family~$\SSS$ in a left-Ore category~$\CCC$ that admits left-lcms, a $\Square$-witness~$\SW\!$, a short right-lcm selector~$\RC$ on~$\SSSs$, a short left-lcm selector~$\LC$ on~$\SSSs$}
\INPUT{A signed $\SSSs$-path $\ww$}
\OUTPUT{A symmetric $\SSSs$-normal decomposition of~$\cl\ww$}
\STATE{$O := $ the return value of Algorithm~\ref{A:RightRev} with~$\RC$ running on~$\ww$}
\IF{$O$ is a positive--negative path $\vv \pc \INV\uu$}
\STATE{compute a symmetric $\SSS$-normal decomposition $\INV{\uu'} \pc \vv'$ of $\uu \pc \INV\vv$ (Algorithm~\ref{A:SymNormal2})}
\RETURN{$\INV{\uu'} \pc \vv'$}
\ELSE
\COMMENT{$O$ is $\mathtt{fail}$}
\RETURN{$\mathtt{fail}$}
\COMMENT{$\cl\ww$ has no symmetric $\SSS$-normal decomposition}
\ENDIF
\end{algorithmic}
\end{algo}

\begin{prop}
\label{P:SymNormalGen}
Assume that $\SSS$ is a strong Garside family in an Ore category~$\CCC$, $\SW$ is a $\Square$-witness, $\RC$ is a short right-lcm selector on~$\SSSs$, and $\LC$ is a short left-lcm selector on~$\SSSs$. Then Algorithm~\ref{A:SymNormalGen} running on a signed $\SSSs$-path~$\ww$ returns a symmetric $\SSS$-normal decomposition of~$\cl\ww$ in time $O(\LG\ww^2)$.
\end{prop}

\begin{proof}
By Proposition~\ref{P:SymNormal2}, the final sequence $\INV{\uu'} \pc \vv'$ is a symmetric $\SSS$-normal decomposition of the element $\cl{\vv \pc \INV\uu}$ of~$\Env\CCC$. Now, by construction, $\vv \pc \INV\uu$ represents the same element of~$\Env\CCC$ as the initial path~$\ww$.

The complexity follows from Lemma~\ref{L:Terminating} and Propositions~\ref{P:Normal} and~\ref{P:SymNormal2}.
\end{proof}

\begin{exam}
Again in the $\BP3$ context, let $\ww$ be the signed word $\ww = \INV\ttb \pc \INV\tta \pc \ttb \pc \tta \pc \tta$. Right-reversing~$\ww$ using the right-lcm selector on~$\{1, \tta, \ttb, \tta\ttb, \ttb\tta, \Delta\}$ yields the posi\-tive--negative word $\tta \pc \tta\ttb \pc \INV{\ttb\tta}$. Left-reversing the latter using the left-lcm selector yields $\INV{\tta\ttb} \pc \ttb\tta \pc \tta$. Normalizing $\tta\ttb$ respectively $\ttb\tta \pc \tta$ do not change these words, and we conclude (once again) that $\ww$ is symmetric normal.
\end{exam}

\begin{rema}
\label{R:NoRightFraction}
The Baumslag--Solitar presentation $(\tta, \ttb ; \tta = \ttb^2\tta\ttb)$ gives an example of a monoid that is left-Ore with left-lcms but not Ore. For instance, as follows from Example~\ref{X:BaumslagSolitar}, the elements~$\tta$ and $\ttb\tta$ admit no common right-multiple. As a consequence, the element~$\tta\inv \ttb\tta$ of the group $\PRES{\tta, \ttb}{\tta = \ttb^2\tta\ttb}$ admits no expression as a right-fraction, and no expression as a left-fraction~$\ff\inv \gg$ where $\ff, \gg$ are left-disjoint.
\end{rema}

\subsubsection*{Incremental method}
\label{SS:Quotient}

Besides Algorithms~\ref{A:SymNormal2} and~\ref{A:SymNormalGen}, which are global, we can also look for local methods in the spirit of Algorithms~\ref{A:LeftMult} and~\ref{A:Normal}, that is, methods for computing, say, a symmetric normal decomposition of~$\cc\gg$ from a symmetric normal decomposition of~$\gg$ when $\cc$ lies in the reference Garside family. This is easy.

\enlargethispage{4mm}

\begin{algo}{Left-multiplication---see Figure~\ref{F:SymLeftMult}}
\label{A:SymLeftMult}
\begin{algorithmic}[1]
\CONTEXT{A strong Garside family~$\SSS$ in a left-Ore category~$\CCC$ that admits left-lcms, a $\Square$-witness~$\SW\!$, a short left-lcm selector~$\LC$ on~$\SSSs$}
\INPUT{A symmetric $\SSS$-normal decomposition $\seqqqqqq{\INV{\bb_\qq}}\etc{\INV{\bb_1}}{\aa_1}\etc{\aa_\pp}$ of an element~$\gg$ of~$\Env\CCC$ and an element~$\cc$ of~$\SSSs$}
\OUTPUT{A symmetric $\SSS$-normal decomposition of $\bb \gg$, if the latter is defined}
\STATE{$\cc_{-\qq} := \cc$}
\FOR{$\ii$ decreasing from~$\qq$ to~$1$}
\STATE{$\bb'_\ii := \LC(\cc_{-\ii}, \bb_\ii)$\,; $\cc_{-\ii+1} := \LC(\bb_\ii, \cc_{-\ii})$}
\ENDFOR
\FOR{$\ii$ increasing from~$1$ to~$\pp$}
\STATE{$(\aa'_\ii, \cc_\ii) := \SW(\cc_{\ii-1}, \aa_\ii)$}
\ENDFOR
\RETURN{$\seqqqqqqq{\INV{\bb'_\qq}}\etc{\INV{\bb'_1}}{\aa'_1}\etc{\aa'_\pp}{\cc_\pp}$}
\end{algorithmic}
\end{algo}

\begin{figure}[htb]
\begin{picture}(105,18)(0,-3)
\pcline{->}(1,0)(14,0)
\tbput{$\cc$}
\pcline{<-}(16,0)(29,0)
\tbput{$\bb_\qq$}
\pcline[style=etc](32,0)(43,0)
\pcline{<-}(46,0)(59,0)
\tbput{$\bb_1$}
\pcline{->}(61,0)(74,0)
\tbput{$\aa_1$}
\pcline[style=etc](77,0)(88,0)
\pcline{->}(91,0)(104,0)
\tbput{$\aa_\pp$}

\psarc[style=double](7,5){7}{90}{180}
\pcline[style=double](7,12)(14,12)
\pcline{<-}(16,12)(29,12)
\taput{$\bb'_\qq$}
\pcline[style=etc](32,12)(43,12)
\pcline{<-}(46,12)(59,12)
\taput{$\bb'_1$}
\pcline{->}(61,12)(74,12)
\taput{$\aa'_1$}
\pcline[style=etc](77,12)(88,12)
\pcline{->}(91,12)(104,12)
\taput{$\aa'_\pp$}

\pcline[style=double](0,5)(0,1)
\pcline{->}(15,11)(15,1)
\tlput{$\cc_{-\qq}$}
\pcline{->}(30,11)(30,1)
\tlput{$\cc_{-\qq+1}$}
\pcline{->}(45,11)(45,1)
\tlput{$\cc_{-1}$}
\pcline{->}(60,11)(60,1)
\tlput{$\cc_{0}$}
\pcline{->}(75,11)(75,1)
\trput{$\cc_1$}
\pcline{->}(90,11)(90,1)
\trput{$\cc_{\pp-1}$}
\pcline{->}(105,11)(105,1)
\trput{$\cc_\pp$}

\psarc[style=thin](30,12){3}{180}{270}
\psarc[style=thin](60,12){3}{180}{270}
\psarc[style=thin](75,12){3}{180}{270}
\psarc[style=thin](105,12){3}{180}{270}
\end{picture}
\caption[]{\sf\smaller Algorithm~\ref{A:SymLeftMult}: starting from $\cc$ in~$\SSSs$ and a symmetric $\SSS$-normal decomposition $\seqqqqqq{\INV{\bb_\qq}}\etc{\INV{\bb_1}}{\aa_1}\etc{\aa_\pp}$ of~$\gg$, one obtains a symmetric $\SSS$-normal decomposition of~$\cc\gg$.}
\label{F:SymLeftMult}
\end{figure}

\begin{prop}
\label{P:SymLeftMult}
Assume that $\SSS$ is a strong Garside family in a left-Ore category~$\CCC$ that admits left-lcms, $\SW$ is a $\Square$-witness on~$\SSSs$, and $\LC$ is a short left-lcm selector on~$\SSSs$. Then Algorithm~\ref{A:SymLeftMult} running on a symmetric $\SSS$-normal decomposition of an element~$\gg$ of~$\Env\CCC$ and $\cc$ in~$\SSSs$ returns a symmetric $\SSS$-normal decomposition of~$\cc\gg$ if the latter is defined. The map $\LC$ is invoked $2\qq$ times and the map $\SW$ is invoked $\pp$ times.
\end{prop}

\begin{proof}
The third domino rule (Lemma~\ref{L:Domino3}) implies that $\seqqq{\bb'_1}\etc{\bb'_\qq}$ is $\SSS$-greedy. The fourth domino rule implies that $\bb'_1$ and~$\aa'_1$ are left-disjoint. Next, the first domino rule implies that $\seqqq{\aa'_1}\etc{\aa'_\pp}$ is $\SSS$-normal. Finally, $\seqq{\aa'_\pp}{\cc_\pp}$ is $\SSS$-normal by construction. Hence $\seqqqqqqq{\INV{\bb'_\qq}}\etc{\INV{\bb'_1}}{\aa'_1}\etc{\aa'_\pp}{\cc_\pp}$ is symmetric $\SSS$-normal. The commutativity of the diagram in Figure~\ref{F:SymLeftMult} then guarantees that the latter sequence is a decomposition of~$\cc\bb_\qq\inv \pdots \bb_1\inv \aa_1 \pdots \aa_\pp$, hence of~$\cc\gg$.
\end{proof}

The symmetry of the definition of symmetric normal paths implies that there exists a counterpart of Algorithm~\ref{A:SymLeftMult} computing a symmetric normal decomposition of~$\gg \cc\inv$ from one of~$\gg$ for $\cc$ in the reference Garside family~$\SSS$. By contrast, there exists no simple general method for left-division (or right-multiplication). However, when $\SSS$ happens to be bounded, such a method exists, because left-dividing by an element~$\Delta(\xx)$ is easy. Now, if $\SSS$ is bounded by~$\Delta$ and $\aa$ belongs to~$\SSS(\xx, \ud)$, then $\aa \der(\aa) = \Delta(\xx)$ holds, so that left-dividing by~$\aa$ can be decomposed into left-dividing by~$\Delta(\xx)$ and then left-multiplying by~$\der(\aa)$. Hereafter, in order to simplify notation, we shall often skip the source of the considered $\Delta$-element, thus writing~$\Delta(\ud)$ instead of~$\Delta(\xx)$ when mentioning~$\xx$ explicitly is not necessary.

\enlargethispage{3mm}

\begin{algo}{Left-division---see Figure~\ref{F:SymLeftDiv}}
\label{A:SymLeftDiv}
\begin{algorithmic}[1]
\CONTEXT{A Garside family~$\SSS$ bounded by a map~$\Delta$ in an Ore category~$\CCC$, a $\Square$-witness~$\SW$ for~$\SSSs$}
\INPUT{A symmetric $\SSS$-normal decomposition $\seqqqqqq{\INV{\bb_\qq}}\etc{\INV{\bb_1}}{\aa_1}\etc{\aa_\pp}$ of~$\gg$ in~$\Env\CCC$, an element~$\cc$ of~$\SSSs$}
\OUTPUT{A symmetric $\SSS$-normal decomposition of $\cc\inv \gg$, if the latter is defined}
\STATE{$\cc_{-\qq} := \cc$}
\FOR{$\ii$ decreasing from~$\qq$ to~$1$}
\STATE{$(\cc_{-\ii+1}, \bb'_\ii) := \SW(\bb_\ii, \cc_{-\ii})$}
\ENDFOR
\STATE{$\cc'_0 := \derL(\cc_0)$}
\FOR{$\ii$ increasing from~$1$ to~$\pp$}
\STATE{$(\aa'_\ii, \cc'_\ii) := \SW(\cc'_{\ii-1}, \aa_\ii)$}
\ENDFOR
\STATE{$\bb'_0:= \der(\aa'_1)$}
\RETURN{$\seqqqqqqqq{\INV{\bb'_\qq}}\etc{\INV{\bb'_1}}{\INV{\bb'_0}}{\aa'_2}\etc{\aa'_\pp}{\cc'_\pp}$}
\end{algorithmic}
\end{algo}

\begin{figure}[htb]
\begin{picture}(120,29)(0,-2)
\psarc[style=double](7,7){7}{180}{270}
\pcline[style=double](7,0)(14,0)

\pcline{<-}(16,0)(29,0)
\tbput{$\bb'_\qq$}
\pcline[style=etc](32,0)(43,0)
\pcline{<-}(46,0)(59,0)
\tbput{$\bb'_1$}

\pcline{<-}(1,12)(14,12)
\taput{$\cc$}
\pcline{<-}(16,12)(29,12)
\taput{$\bb_\qq$}
\pcline[style=etc](32,12)(43,12)
\pcline{<-}(46,12)(59,12)
\tbput{$\bb_1$}
\pcline{->}(61,12)(74,12)
\put(64,9.5){$\aa_1$}
\pcline{->}(76,12)(89,12)
\tbput{$\aa_2$}
\pcline[style=etc](92,12)(103,12)
\pcline{->}(106,12)(119,12)
\tbput{$\aa_\pp$}

\pcline{<-}(46,24)(59,24)
\taput{$\phi\inv(\bb'_1)$}
\pcline{->}(61,24)(74,24)
\taput{$\aa'_1$}
\pcline{->}(76,24)(89,24)
\taput{$\aa'_2$}
\pcline[style=etc](92,24)(103,24)
\pcline{->}(106,24)(119,24)
\taput{$\aa'_\pp$}

\pcline[style=double](0,11)(0,7)
\pcline{->}(15,11)(15,1)
\trput{$\cc_{-\qq}$}
\pcline{->}(30,11)(30,1)
\trput{$\cc_{-\qq+1}$}
\pcline{->}(45,11)(45,1)
\trput{$\cc_{-1}$}
\pcline{->}(60,11)(60,1)
\tlput{$\cc_0$}

\pcline{->}(45,23)(45,13)
\tlput{$\derL(\cc_{-1})$}
\pcline{->}(60,23)(60,13)
\tlput{$\derL(\cc_0)$}
\trput{$\cc'_0$}
\pcline{->}(75,23)(75,13)
\trput{$\cc'_1$}
\pcline{->}(90,23)(90,13)
\trput{$\cc'_2$}
\pcline{->}(105,23)(105,13)
\trput{$\cc'_{\pp-1}$}
\pcline{->}(120,23)(120,13)
\trput{$\cc'_\pp$}

\psbezier(74,23.5)(70,20)(70,20)(70,13)
\pcline(70,11)(70,7)
\psarc(63,7){7}{270}{360}
\pcline{<-}(61,0)(63,0)
\put(71,5){$\bb'_0$}

\psarc[style=thin](30,12){3}{0}{180}
\psarc[style=thin](45,12){3}{0}{180}
\psarc[style=thin](60,12){3}{0}{180}
\psarc[style=thin](75,12){3}{0}{180}
\psarc[style=thin](90,12){3}{0}{180}
\psarc[style=thin](105,12){3}{0}{180}

\psarc[style=thin](30,0){3}{90}{180}
\psarc[style=thin](60,0){3}{90}{180}

\psarc[style=thin](75,24){3}{180}{270}
\psarc[style=thin](90,24){3}{180}{270}
\psarc[style=thin](120,24){3}{180}{270}

\psarc[style=thinexist](30,0){2.5}{180}{360}
\psarc[style=thinexist](45,0){2.5}{180}{360}
\psarc[style=thinexist](60.5,0){2.5}{180}{360}

\psarc[style=thinexist](75,24){3.5}{0}{210}
\psarc[style=thinexist](90,24){2.5}{0}{180}
\psarc[style=thinexist](105,24){2.5}{0}{180}
\end{picture}
\caption[]{\sf\smaller Left-division by an element of~$\SSSs$ starting from a symmetric $\SSS$-normal path $\seqqqqqq{\INV{\bb_\qq}}\etc{\INV{\bb_1}}{\aa_1}\etc{\aa_\pp}$.}
\label{F:SymLeftDiv}
\end{figure}

\begin{prop}
\label{P:SymLeftDiv}
Assume that $\SSS$ is a Garside family that is bounded by a map~$\Delta$ in an Ore category~$\CCC$. Then Algorithm~\ref{A:SymLeftDiv} running on a symmetric $\SSS$-normal decomposition of an element~$\gg$ of~$\Env\CCC$ and $\cc$ in~$\SSSs$ returns a symmetric $\SSS$-normal decomposition of~$\cc\inv \gg$ if the latter is defined. 
\end{prop}

\begin{proof}
By construction, the diagram of Figure~\ref{F:SymLeftDiv} is commutative, so the returned path is equivalent to $\seqqqqqqqq{\INV{\cc}}{\INV{\bb_\qq}}\etc{\INV{\bb_1}}{\aa_1}\etc{\aa_\pp}$, hence it is a decomposition of~$\cc\inv \gg$, and its entries are elements of~$\SSSs$. So the point is to check that the path is symmetric $\SSS$-greedy.

First, as $\seqq{\bb_\ii}{\bb_{\ii+1}}$ is $\SSS$-normal for $\ii = 1 \wdots \qq-1$, the first domino rule (Lemma~\ref{L:Domino1}) implies that $\seqq{\bb'_\ii}{\bb'_{\ii+1}}$ is $\SSS$-normal as well. Similarly, as $\seqq{\aa_\ii}{\aa_{\ii+1}}$ is $\SSS$-normal for $\ii = 1 \wdots \pp-1$, the first domino rule (Lemma~\ref{L:Domino1}) implies that $\seqq{\aa'_\ii}{\aa'_{\ii+1}}$ is $\SSS$-normal as well. 

So it remains to check that $\seqq{\bb'_0}{\bb'_1}$ is $\SSS$-normal and that $\bb'_0$ and $\aa'_2$ are left-disjoint. Consider $\seqq{\bb'_0}{\bb'_1}$. Assume that $\aa$ is an element of~$\SSSs$ satisfying $\aa \dive \phi\inv(\bb'_1)$ and~$\aa \dive \aa'_1$. A fortiori, we have $\aa \dive \phi\inv(\bb'_1) \derL(\cc_{-1})$ and~$\aa \dive \aa'_1 \cc'_1$, that is, $\aa\dive \cc'_0 \bb_1$ and~$\aa \dive \cc'_0 \aa_1$. By assumption, $\bb_1$ and $\aa_1$ are left-disjoint, hence we deduce $\aa \dive \cc'_0$. Now, by assumption, $\seqq{\cc_0}{\bb'_1}$ is $\SSS$-normal, hence, by Lemma~\ref{L:DeltaDisjoint}, $\der(\cc_0)$ and~$\bb'_1$ have no nontrivial common left-divisor. As $\phi$ and, therefore, $\phi\inv$ is an automorphism (Lemma~\ref{L:Auto}), it follows that $\phi\inv(\der(\cc_0))$ and~$\phi\inv(\bb'_1)$ have no nontrivial common left-divisor. Now, by definition, $\phi\inv(\der(\cc_0))$ is~$\cc'_0$. As we have $\aa \dive \phi\inv(\bb'_1)$ and $\aa \dive \cc'_0$, we deduce that $\aa$ is invertible and, therefore, $\phi\inv(\bb'_1)$ and~$\aa'_1$ have no nontrivial common left-divisor. Applying the automorphism~$\phi$, we deduce that $\bb'_1$ and~$\phi(\aa'_1)$, that is, $\bb'_1$ and $\der(\bb'_0)$, have no nontrivial common left-divisor. By Lemma~\ref{L:DeltaDisjoint} again, this implies that $\seqq{\bb'_0}{\bb'_1}$ is $\SSS$-normal.

Finally, again by Lemma~\ref{L:DeltaDisjoint}, the assumption that $\seqq{\aa'_1}{\aa'_2}$ is $\SSS$-normal implies that $\der(\aa'_1)$ and $\aa'_2$ are left-disjoint: by definition, $\der(\aa'_1)$ is~$\bb'_0$, so $\bb'_0$ and~$\aa'_2$ are left-disjoint. So $\seqqqqqqqq{\INV{\bb'_\qq}}\etc{\INV{\bb'_1}}{\INV{\bb'_0}}{\aa'_2}\etc{\aa'_\pp}{\cc'_\pp}$ is indeed a symmetric $\SSS$-normal path.
\end{proof}

\begin{exam}
\label{X:SymLeftDiv}
\rightskip54mm
Consider $(\tta\ttb)\inv(\ttb\tta^2)$ in~$\BP3$ once more. Suppose that we found that $\ttb\tta \pc \tta$ is an $\SSS$-normal decomposition of~$\ttb\tta^2$, and we wish to find a (the) $\SSS$-normal decomposition of $(\tta\ttb)\inv(\ttb\tta^2)$. Then applying Algorithm~\ref{A:SymLeftDiv} amounts to completing the diagram on the right, in which of course we read the expected symmetric normal decomposition $\INV{\tta\ttb} \pc \ttb\tta \pc \tta$.
\hfill
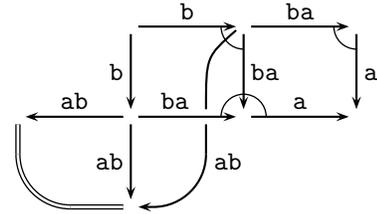
\begin{picture}(0,0)(-6,-0)
\psarc[style=double](7,7){7}{180}{270}
\psline[style=double](7,0)(14,0)
\pcline{<-}(1,12)(14,12)
\taput{$\tta\ttb$}
\pcline{->}(16,12)(29,12)
\put(19,13){$\ttb\tta$}
\pcline{->}(31,12)(44,12)
\taput{$\tta$}
\pcline{->}(16,24)(29,24)
\taput{$\ttb$}
\pcline{->}(31,24)(44,24)
\taput{$\ttb\tta$}
\psline[style=double](0,11)(0,7)
\pcline{->}(15,11)(15,1)
\tlput{$\tta\ttb$}
\pcline{->}(15,23)(15,13)
\tlput{$\ttb$}
\pcline{->}(30,23)(30,13)
\trput{$\ttb\tta$}
\pcline{->}(45,23)(45,13)
\trput{$\tta$}
\psbezier(29,23.5)(25,20)(25,20)(25,13)
\pcline(25,11)(25,7)
\psarc(18,7){7}{270}{360}
\pcline{<-}(16,0)(18,0)
\put(26,5){$\tta\ttb$}
\psarc[style=thin](30,12){3}{0}{180}
\psarc[style=thin](30,24){3}{180}{270}
\psarc[style=thin](45,24){3}{180}{270}
\end{picture}
\end{exam}

If, in Algorithm~\ref{A:SymLeftDiv}, $\cc$ is~$\Delta(\xx)$, then each element~$\cc_{-\ii}$ is of the form~$\Delta(\ud)$, so that $\bb'_\ii$ is simply~$\phi(\bb_\ii)$, so that $\cc'_0$ is trivial, and so are all elements~$\cc_\ii$ with $\ii \ge 0$. We immediately deduce

\begin{coro}
Assume that $\SSS$ is a Garside family that is bounded by a map~$\Delta$ in an Ore category~$\CCC$. If $\seqqqqqq{\INV{\bb_\qq}}\etc{\INV{\bb_1}}{\aa_1}\etc{\aa_\pp}$ is a symmetric $\SSS$-normal decomposition of~$\gg$, then a symmetric $\SSS$-normal decomposition of~$\Delta(\xx)\inv \gg$ (where $\xx$ is the source of~$\gg$) is $\seqqqqqq{\INV{\phi(\bb_\qq)}}\etc{\INV{\phi(\bb_1)}}{\INV{\der(\aa_1)}}{\aa_2}\etc{\aa_\pp}$. 
\end{coro}

\subsection{Computing $\Delta$-normal decompositions}
\label{SS:Delta}

When a Garside family is bound\-ed, alternative distinguished decompositions for the elements of the ambient category and its groupoid of fractions arise, namely the $\Delta$-normal decompositions. We now describe incremental methods for computing such decompositions, namely algorithms that, when starting from a $\Delta$-normal decomposition of an element~$\gg$ and an element~$\cc$ of the considered Garside family, return $\Delta$-normal decompositions of~$\cc \gg$ and~$\cc\inv \gg$, respectively. 

\begin{algo}{Left-multiplication for $\Delta$-normal---see Figure~\ref{F:DeltaLeftMult}}
\label{A:DeltaLeftMult}
\begin{algorithmic}[1]
\CONTEXT{A Garside family~$\SSS$ that is bounded by a map~$\Delta$ in an Ore category~$\CCC$, a $\Square$-witness~$\SW$ for~$\SSSs$, an $\eqir$-test~$\EE$ in~$\CCC$}
\INPUT{A $\Delta$-normal decomposition~$\seqqqq{\DELTA\nn(\xx)}{\aa_1}\etc{\aa_\pp}$ of an element~$\gg$ of~$\Env\CCC$, an element~$\cc$ of~$\SSSs$}
\OUTPUT{A $\Delta$-normal decomposition of $\cc\gg$, if the latter is defined}
\STATE{$\cc_0 := \phi^\nn(\cc)$}
\FOR{$\ii$ increasing from~$1$ to~$\pp$}
\STATE{$(\aa'_\ii, \cc_\ii) := \SW(\cc_{\ii-1}, \aa_\ii)$}
\ENDFOR
\IF{$\EE(\Delta(\ud), \aa'_1) \not= \bot$}
\COMMENT{$\aa'_1$ is $\Delta$-like}
\STATE{$\aa'_2:= \EE(\Delta(\ud), \aa'_1) \aa'_2$}
\RETURN{$\seqqqqq{\DELTA{\nn+1}(\ud)}{\aa'_2}\etc{\aa'_{\pp}}{\cc_\pp}$}
\ELSE
\COMMENT{$\aa'_1$ is not $\Delta$-like}
\RETURN{$\seqqqqq{\DELTA{\nn}(\ud)}{\aa'_1}\etc{\aa'_{\pp}}{\cc_\pp}$}
\ENDIF
\end{algorithmic}
\end{algo}

\begin{figure}[htb]
\begin{picture}(45,25)(0,-2)
\pcline{->}(-14,0)(-1,0)
\tbput{$\cc$}
\pcline{<->}(1,0)(14,0)
\tbput{$\DELTA\nn(\ud)$}
\pcline{->}(16,0)(29,0)
\tbput{$\aa_1$}
\psline[style=etc](33,0)(41,0)
\pcline{->}(46,0)(59,0)
\tbput{$\aa_\pp$}

\pcline{<->}(1,12)(14,12)
\taput{$\DELTA\nn(\ud)$}
\pcline{->}(16,12)(29,12)
\taput{$\aa'_1$}
\psline[style=etc](33,12)(41,12)
\pcline{->}(46,12)(59,12)
\taput{$\aa'_\pp$}

\pcline{->}(0,11)(0,1)
\tlput{$\cc$}
\pcline{->}(15,11)(15,1)
\tlput{$\phi^\nn(\cc)$}
\trput{$\cc_0$}
\pcline{->}(30,11)(30,1)
\trput{$\cc_1$}
\pcline{->}(45,11)(45,1)
\trput{$\cc_{\pp-1}$}
\pcline{->}(60,11)(60,1)
\trput{$\cc_\pp$}

\psarc[style=thin](30,12){3}{180}{270}
\psarc[style=thin](60,12){3}{180}{270}

\psline[style=double](-15,1)(-15,5)
\psarc[style=double](-8,5){7}{90}{180}
\psline[style=double](-8,12)(-1,12)

\pcline[style=thinexist](0,18)(30,18)
\taput{$\DELTA{\nn+1}(\ud) \EE(\Delta(\ud), \aa'_1)$ if $\aa'_1$ is $\Delta$-like}
\psline[style=thinexist](0,18)(0,16)
\psline[style=thinexist](30,18)(30,16)
\end{picture}
\caption[]{\sf\smaller Computing a $\Delta$-normal decomposition of~$\cc \gg$ starting from a $\Delta$-normal decomposition of~$\gg$: the $\vert\nn\vert$~entries of the type~$\Delta(\ud)$ are treated directly using~$\phi^\nn$, and then it just remains to left-multiply a (positive) $\Delta$-normal sequence by an element 
of~$\SSS$ using the standard method. Note that, depending on the sign of~$\nn$, the diagrams for~$\DELTAA\nn(\ud)$ consists either of $\nn$~right-oriented arrows (case $\nn \ge 0$) or of $\vert\nn\vert$~left-oriented arrows (case $\nn \le 0$).}
\label{F:DeltaLeftMult}
\end{figure}

\begin{algo}{Left-division for $\Delta$-normal---see Figure~\ref{F:DeltaLeftDiv}}
\label{A:DeltaLeftDiv}
\begin{algorithmic}[1]
\CONTEXT{A Garside family~$\SSS$ that is bounded by a map~$\Delta$ in an Ore category~$\CCC$, a $\Square$-witness~$\SW$ for~$\SSSs$, an $\eqir$-test~$\EE$ in~$\CCC$}
\INPUT{A $\Delta$-normal decomposition~$\seqqqq{\DELTA\nn(\xx)}{\aa_1}\etc{\aa_\pp}$ for an element~$\cc$ of~$\Env\CCC$ and an element~$\cc$ of~$\SSSs$ with source~$\xx$}
\OUTPUT{An $\Delta$-normal decomposition of $\cc\inv \gg$}
\STATE{$\cc_0 := \phi^\nn(\derL(\cc))$}
\FOR{$\ii$ increasing from~$1$ to~$\pp$}
\STATE{$(\aa'_\ii, \cc_\ii) := \SW(\cc_{\ii-1}, \aa_\ii)$}
\ENDFOR
\IF{$\EE(\Delta(\ud), \aa'_1) \not= \bot$}
\COMMENT{$\aa'_1$ is $\Delta$-like}
\STATE{$\aa'_2:= \EE(\Delta(\ud), \aa'_1) \aa'_2$}
\RETURN{$\seqqqqq{\DELTA{\nn}(\ud)}{\aa'_2}\etc{\aa'_{\pp}}{\cc_\pp}$}
\ELSE
\COMMENT{$\aa'_1$ is not $\Delta$-like}
\RETURN{$\seqqqqq{\DELTA{\nn-1}(\ud)}{\aa'_1}\etc{\aa'_{\pp}}{\cc_\pp}$}
\ENDIF
\end{algorithmic}
\end{algo}

\begin{figure}[htb]
\begin{picture}(45,28)(0,-2)
\pcline{<-}(-14,0)(-1,0)
\tbput{$\cc$}
\pcline{<->}(1,0)(14,0)
\tbput{$\DELTA\nn(\xx)$}
\pcline{->}(16,0)(29,0)
\tbput{$\aa_1$}
\psline[style=etc](33,0)(41,0)
\pcline{->}(46,0)(59,0)
\tbput{$\aa_\pp$}

\pcline{<-}(-14,12)(-1,12)
\taput{$\Delta(\ud)$}
\pcline{<->}(1,12)(14,12)
\taput{$\DELTA\nn(\ud)$}
\pcline{->}(16,12)(29,12)
\taput{$\aa'_1$}
\psline[style=etc](33,12)(41,12)
\pcline{->}(46,12)(59,12)
\taput{$\aa'_\pp$}

\pcline{->}(0,11)(0,1)
\tlput{$\derL(\cc)$}
\pcline{->}(15,11)(15,1)
\trput{$\cc_0$}
\pcline{->}(30,11)(30,1)
\trput{$\cc_1$}
\pcline{->}(45,11)(45,1)
\trput{$\cc_{\pp-1}$}
\pcline{->}(60,11)(60,1)
\trput{$\cc_\pp$}

\psarc[style=thin](30,12){3}{180}{270}
\psarc[style=thin](60,12){3}{180}{270}
\psline[style=double](-15,11)(-15,1)

\pcline[style=thin](-15,18)(15,18)
\taput{$\DELTA{\nn-1}(\ud)$}
\psline[style=thin](-15,18)(-15,16)
\psline[style=thin](15,18)(15,16)

\pcline[style=thinexist](-15,24)(30,24)
\taput{$\DELTA\nn(\ud) \EE(\Delta(\ud), \aa'_1)$ if $\aa'_1$ is $\Delta$-like}
\psline[style=thinexist](-15,24)(-15,22)
\psline[style=thinexist](30,24)(30,22)
\end{picture}
\caption[]{\sf\smaller Algorithm~\ref{A:DeltaLeftDiv} : Computing a $\Delta$-normal decomposition for $\cc\inv \gg$ from a $\Delta$-normal decomposition of~$\gg$.}
\label{F:DeltaLeftDiv}
\end{figure}
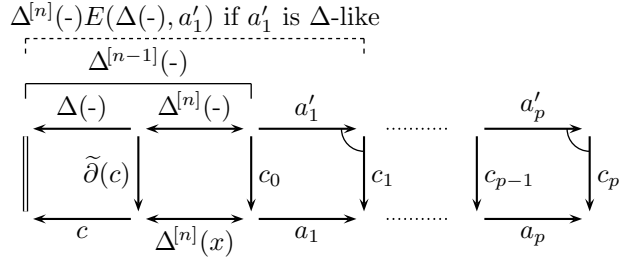

\begin{prop}
\label{P:DeltaLeftMult}
Assume that $\SSS$ is a Garside family that is bounded by a map~$\Delta$ in an Ore category~$\CCC$, $\SW$ is a $\Square$-witness on~$\SSS$, and $\EE$ is an $\eqir$-test in~$\CCC$. Then Algorithm~\ref{A:DeltaLeftMult} (\resp Algorithm~\ref{A:DeltaLeftDiv}) running on a $\Delta$-normal decomposition of an element~$\gg$ of~$\Env\CCC$ and $\cc$ in~$\SSSs$ returns a $\Delta$-normal decomposition of~$\cc \gg$ (\resp $\cc\inv \gg$) (when the latter is defined). 
\end{prop}

\begin{proof}
Consider first Algorithm~\ref{A:DeltaLeftMult}. Then the diagram of Figure~\ref{F:DeltaLeftMult} is commutative as, in particular, $\cc \DELTAA\nn(\ud) = \DELTAA\nn(\uu) \phi^\nn(\cc)$ holds whenever defined. As usual, the first domino rule implies that $\seqqq{\aa'_1}\etc{\aa'_\pp}$ is $\SSS$-normal and, by construction, $\seqq{\aa'_\pp}{\cc_\pp}$ is $\SSS$-normal. So it just remains to consider the relation between $\DELTA\nn(\ud)$ and~$\aa'_1$. Now, if $\aa'_1$ is not $\Delta$-like, $\seqq{\DELTA\nn(\ud)}{\aa'_1}$ is $\Delta$-normal. On the other hand, if $\aa'_1$ is $\Delta$-like, then, by definition, we have $\aa'_1 = \Delta(\ud) \EE(\Delta(\ud), \aa'_1)$, so that $\Delta(\ud)$ can be incorporated to~$\DELTA\nn(\ud)$ to form~$\DELTA{\nn+1}(\ud)$ and $\EE(\Delta(\ud), \aa'_1)$ can be incorporated to~$\aa'_2$ to form a new element of~$\SSSs$ as $\CCCi \SSSs \subseteq \SSSs$ holds. Then $\seqq{\DELTA{\nn+1}(\ud)}{\aa'_2}$ must be $\Delta$-normal because the new element~$\aa'_2$ may not be $\Delta$-like. Indeed, $\Delta(\ud) \dive \aa'_2$ would imply $\DELTAA2(\ud) \dive \aa'_1 \aa'_2$, whence $\DELTAA2(\ud) \dive \aa'_1 \aa'_2 \cc_2 = \cc_0 \aa_1 \aa_2$. Now, as $\cc_0$ lies in~$\SSSs$, we can write $\DELTAA2(\ud) = \cc_0 \der(\cc_0) \Delta(\ud) = \cc_0 \Delta(\ud) \phi(\der(\cc_0))$, so $\DELTAA2(\ud) \dive \cc_0 \aa_1 \aa_2$ implies $\cc_0 \Delta(\ud) \phi(\der(\cc_0)) \dive \cc_0 \aa_1 \aa_2$, whence $\Delta(\ud) \dive \aa_1 \aa_2$, contradicting the assumption that $\seqq{\DELTA\nn(\ud)}{\aa_1}$ is $\Delta$-normal. So no cascade may occur here.

The argument for Algorithm~\ref{A:DeltaLeftDiv} is entirely similar. The only new point is that, in the diagram of Figure~\ref{F:DeltaLeftDiv}, the left square is commutative because $\derL(\cc) \cc = \Delta(\ud)$ holds for every~$\cc$ in~$\SSSs$. \end{proof}

\begin{exam}
\label{X:DeltaLeftDiv}
\rightskip55mm
As in Example~\ref{X:SymLeftDiv}, assume that we know that $\seqq{\ttb\tta}{\tta}$ is a $\Delta$-normal decomposition of~$\ttb^2\tta$ and we wish to find a (the) $\Delta$-normal decomposition of $(\tta\ttb)\inv (\ttb^2\tta)$. Applying Algorithm~\ref{A:DeltaLeftDiv} amounts to completing the diagram on the right, and we read that the expected decomposition is $\DELTA{-1} \pc \ttb \pc \ttb\tta \pc \tta$.
\hfill
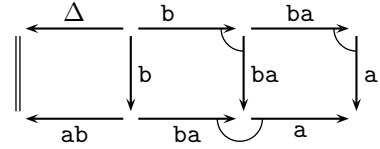
\begin{picture}(0,0)(-7,-6)
\pcline{<-}(1,0)(14,0)
\tbput{$\tta\ttb$}
\pcline{->}(16,0)(29,0)
\tbput{$\ttb\tta$}
\pcline{->}(31,0)(44,0)
\tbput{$\tta$}
\pcline{<-}(1,12)(14,12)
\taput{$\Delta$}
\pcline{->}(16,12)(29,12)
\put(19,13){$\ttb$}
\pcline{->}(31,12)(44,12)
\taput{$\ttb\tta$}
\psline[style=double](0,11)(0,1)
\pcline{->}(15,11)(15,1)
\trput{$\ttb$}
\pcline{->}(30,11)(30,1)
\trput{$\ttb\tta$}
\pcline{->}(45,11)(45,1)
\trput{$\tta$}
\psarc[style=thin](29.5,0){3}{180}{360}
\psarc[style=thin](30,12){3}{180}{270}
\psarc[style=thin](45,12){3}{180}{270}
\end{picture}
\end{exam}

\subsection{Solving the Word Problem}
\label{SS:SignedWordPb}

As in the positive case, we shall describe several solutions for the Word Problem in the involved groupoid of fractions

\subsubsection*{Using symmetric normal decompositions} 

First, symmetric normal decompositions give an direct solution extending Algorithm~\ref{A:WordPb1} to the signed case.

\begin{algo}{Word Problem, general case I}
\label{A:SignedWordPb1}
\begin{algorithmic}[1]
\CONTEXT{A strong Garside family~$\SSS$ in an Ore category~$\CCC$ that admits left-lcms, a $\Square$-witness~$\SW$, a short left-lcm selector~$\LC$, and an $\eqir$-test~$\EE$ on~$\SSSs$}
\INPUT{A signed $\SSSs$-path $\ww$}
\OUTPUT{$\mathtt{true}$ if $\ww$ represents an element~$\id\xx$ in~$\Env\CCC$, and $\mathtt{false}$ otherwise}
\STATE{$\xx := \src\ww$\,; $\yy := \trg\ww$}
\IF{$\xx \not= \yy$}
\STATE{return $\mathtt{false}$}
\ELSE
\STATE{use Algorithm~\ref{A:SymNormalGen} to find a symmetric $\SSS$-normal path $\INV\uu \pc \vv$ that represents~$\cl\ww$}\label{A:SignedWordPb1:SNF}\RETURN{the value of \CALL{CompareNormalPaths}{$\uu$, $\vv$} \quad(Algorithm~\ref{A:WordPb1})}
\ENDIF
\end{algorithmic}
\end{algo}

\begin{prop}
\label{P:SignedWordPb1}
Assume that $\SSS$ is a strong Garside family~$\SSS$ in an Ore category~$\CCC$ that admits left-lcms, $\SW$ is a $\Square$-witness on~$\SSS$, $\LC$ is a short left-lcm selector on~$\SSSs$, and $\EE$ is an $\eqir$-test on~$\SSSs$. Then Algorithm~\ref{A:SignedWordPb1} running on a signed $\SSSs$-path~$\ww$of length at most~$\ell$ decides in time~$O(\ell^2)$ whether $\cl\ww$ is an identity-element in~$\Env\CCC$.
\end{prop}

\begin{proof}
By Proposition~\ref{P:SymNormalGen}, line~\ref{A:SignedWordPb1:SNF} of Algorithm~\ref{A:SignedWordPb1} computes a symmetric $\SSS$-normal path~$\INV\uu \pc \vv$ representing~$\cl\ww$ in time~$O(\ell^2)$.
The signed path~$\ww$ represents an identity-element in~$\Env\CCC$ if and only if the (positive) paths~$\uu$ and~$\vv$ represent the same element in~$\CCC$. The claim then follows with Proposition~\ref{P:WordPb1}.
\end{proof}

\subsubsection*{Using $\Delta$-normal decompositions}

According to Proposition~\ref{P:DeltaUnique}, when they exist, $\Delta$-normal decompositions enjoy the same uniqueness property as symmetric normal decompositions. Therefore, in the case when the reference Garside family is not only strong but even bounded, Algorithm~\ref{A:SignedWordPb1} and Proposition~\ref{P:SignedWordPb1} admit exact counterparts where symmetric normal is replaced with $\Delta$-normal. We skip the details.

\subsubsection*{Using reversing}

As in the positive case, at least when the ambient category contains no nontrivial invertible element and admits lcms, we can also solve the Word Problem by using reversing and thus avoiding to compute distinguished decompositions. To stick to the context described in Subsection~\ref{SS:Rev} we assume here that the ambient category admits left- and right-lcms and contains no nontrivial invertible element, which amounts to requiring that the lcms are unique.

\begin{algo}{Word Problem, general case II}
\label{A:SignedWordPb2}
\begin{algorithmic}[1]
\CONTEXT{A strong Garside family~$\SSS$ in an Ore category~$\CCC$ that admits unique right- and left-lcms}
\INPUT{A signed $\SSS$-path $\ww$}
\OUTPUT{$\mathtt{true}$ if $\ww$ represents an element~$\id\xx$ in~$\Env\CCC$, and $\mathtt{false}$ otherwise}
\STATE{$\RC$:= the right-lcm selector on~$\SSS \cup \Id\CCC$}
\STATE{$\LC$:= the left-lcm selector on~$\SSS \cup \Id\CCC$}
\STATE{right-$\RC$-reverse~$\ww$ into a positive--negative path~$\uu \pc \INV\vv$ using Algorithm~\ref{A:RightRev}}
\STATE{left-$\LC$-reverse~$\uu \pc \INV\vv$ into a negative--positive path~$\INV{\uu'} \pc \vv'$ using (the left counterpart of) Algorithm~\ref{A:RightRevShort}}
\RETURN{$\TV{\uu' = \vv' = \ew_{\ud}}$}
\end{algorithmic}
\end{algo}

Lemma~\ref{L:Terminating}, Proposition~\ref{P:Gar2Pres}, and their left counterparts imply:

\begin{prop}
\label{P:SignedWordPb2}
Assume that $\SSS$ is a strong Garside family~$\SSS$ in an Ore category~$\CCC$ that admits unique right- and left-lcms.

\ITEM1 Algorithm~\ref{A:SignedWordPb2} solves the Word Problem of~$\Env\CCC$ with respect to~$\SSS$. 

\ITEM2 If $\SSS$ is finite, the complexity of Algorithm~\ref{A:SignedWordPb2} is quadratic in the length of the input path.
\end{prop}

Appealing to a right- and a left-reversing, Algorithm~\ref{A:SignedWordPb2} is nicely symmetric. However, it requires the existence of right-lcms, which is not guaranteed in every Ore category that admits a strong Garside family. Actually, this assumption is superfluous, since a double left-reversing can be used instead.

\begin{algo}{Word Problem, general case III}
\label{A:SignedWordPb3}
\begin{algorithmic}[1]
\CONTEXT{A strong Garside family~$\SSS$ in an Ore category~$\CCC$ that admits unique left-lcms}
\INPUT{A signed $\SSS$-path $\ww$}
\OUTPUT{$\mathtt{true}$ if $\ww$ represent an element~$\id\xx$ in~$\Env\CCC$, and $\mathtt{false}$ otherwise}
\STATE{$\LC$:= the left-lcm selector on~$\SSS \cup \Id\CCC$}
\STATE{left-$\LC$-reverse~$\ww$ into a negative--positive~$\INV\uu \pc \vv$ using (the left counterpart of) Algorithm~\ref{A:RightRev}}
\STATE{left-$\LC$-reverse~$\vv \pc \INV\uu$ into a negative--positive path~$\INV{\uu'} \pc \vv'$ using (the left counterpart of) Algorithm~\ref{A:RightRevShort}}
\RETURN{$\TV{\uu' = \vv' = \ew_{\ud}}$}
\end{algorithmic}
\end{algo}

\begin{prop}
\label{P:SignedWordPb3}
Assume that $\SSS$ is a strong Garside family in an Ore category~$\CCC$ that admits unique left-lcms.

\ITEM1 Algorithm~\ref{A:SignedWordPb3} solves the Word Problem of~$\Env\CCC$ with respect to~$\SSS$. 

\ITEM2 If $\SSS$ is finite, the complexity of Algorithm~\ref{A:SignedWordPb3} is quadratic in the length of the input path.
\end{prop}

\begin{proof}
The signed path~$\ww$ represents an identity-element in~$\Env\CCC$ if and only if the paths~$\uu$ and~$\vv$ represent the same element in~$\CCC$. By the counterpart of Proposition~\ref{P:Gar2Pres}, left-$\LC$-reversing is complete and terminating, so $\uu$ and~$\vv$ represent the same element of~$\CCC$ if and only if $\vv \pc \INV\uu$ is left-$\LC$-reversible to an empty path. By Lemma~\ref{L:Terminating}, the complexity of Algorithm~\ref{A:SignedWordPb3} is in~$O(\LG\ww^2)$.
\end{proof}

\begin{exam}
Running on the signed word $\ww = \INV{\tta\ttb} \pc \ttb\tta \pc \tta$, Algorithm~\ref{A:SignedWordPb2} right-reverses~$\ww$ into $\tta \pc \tta\ttb \pc \INV{\ttb\tta}$, and then left-reverses the latter words back to $\INV{\tta\ttb} \pc \ttb\tta \pc \tta$. The final word is not empty, hence $\ww$ does not represent~$1$ in the group~$B_3$.

Running on~$\ww$, Algorithm~\ref{A:SignedWordPb3} first left-reverses~$\ww$ into itself (since $\ww$ is a negative--positive word), then switches the numerator and the denominator into $\ttb\tta \pc \tta \pc \INV{\tta\ttb}$, and finally left-reverses the latter word. One finds now $\INV\tta \pc \ttb \pc \ttb$, a nonempty word, and one concludes again that $\ww$ does not represent~$1$.
\end{exam}

\begin{rema}
As in Section~\ref{S:Pos}, we could also state a result referring to an arbitrary generating family~$\AAA$
that satisfies convenient properties but that is not necessarily a strong Garside family.
\end{rema}

\subsection{Computing decompositions for an inverse}
\label{SS:Inv}

We complete the analysis of our distinguished decompositions with explicit methods for finding a decomposition of~$\gg\inv$ from one of~$\gg$.

In the case of symmetric normal decompositions, the result is trivial:

\begin{prop}
Assume that $\CCC$ is a left-Ore category and $\SSS$ is a Garside family in~$\CCC$.
If $\seqqqqqq{\INV{\bb_\qq}}\etc{\INV{\bb_1}}{\aa_1}\etc{\aa_\pp}$ is a symmetric $\SSS$-normal decomposition for an element~$\gg$ of $\Env\CCC$, then $\seqqqqqq{\INV{\aa_\pp}}\etc{\INV{\aa_1}}{\bb_1}\etc{\bb_\qq}$ is a symmetric $\SSS$-normal decomposition for~$\gg\inv$. 
\end{prop}

\begin{proof}
The definition makes it obvious that $\seqqqqqq{\INV{\aa_\pp}}\etc{\INV{\aa_1}}{\bb_1}\etc{\bb_\qq}$ is symmetric $\SSS$-normal whenever $\seqqqqqq{\INV{\bb_\qq}}\etc{\INV{\bb_1}}{\aa_1}\etc{\aa_\pp}$ is. On the other hand, $\gg = \bb_\qq\inv \pdots \bb_1\inv \aa_1 \pdots \aa_\pp$ implies $\gg\inv = \aa_\pp\inv \pdots \aa_1\inv \bb_1 \pdots \aa_\pp$ in~$\Env\CCC$.
\end{proof}

The case of $\Delta$-normal decompositions is more complicated, as the definition is not symmetric. However, finding an explicit formula is not difficult.

\begin{prop}
\label{P:Inverse}
Assume that $\SSS$ is a Garside family that is bounded by~$\Delta$ in a cancellative category~$\CCC$. If $\seqqqq{\DELTAA\nn}{\aa_1}\etc{\aa_\pp}$ is a $\Delta$-normal decomposition for an element~$\gg$ of~$\Env\CCC$, then $\seqqqq{\DELTAA{-\nn - \pp}}{\der(\phi^{-\nn-\pp}(\aa_\pp))}\etc{\der(\phi^{-\nn-1}(\aa_1))}$ is a 
$\Delta$-normal decomposition for~$\gg\inv$.
\end{prop}

\begin{proof}
We first check that the sequence mentioned in the statement is $\Delta$-normal. So let~$\ii < \pp$. By hypothesis, $\seqq{\aa_\ii}{\aa_{\ii+1}}$ is 
$\SSS$-normal. As $\phi$ is an automorphism of~$\SSS$, the path $\seqq{\phi^{-\nn-\ii}(\aa_\ii)}{\phi^{-\nn-\ii}(\aa_{\ii+1})}$ is also $\SSS$-normal. By Lemma~\ref{L:DeltaDisjoint}, this implies that $\der(\phi^{-\nn-\ii}(\aa_\ii))$ and~$\phi^{-\nn-\ii}(\aa_{\ii+1})$ are left-disjoint. The latter element is $\phi(\phi^{-\nn-\ii-1}(\aa_{\ii+1}))$, hence it is $\der(\der(\phi^{-\nn-\ii-1}(\aa_{\ii+1})))$. Reading the above disjointness result in a symmetric way and applying Lemma~\ref{L:DeltaDisjoint} again, we deduce that $\seqq{\der(\phi^{-\nn-\ii-1}(\aa_{\ii+1}))}{\der(\phi^{-\nn-\ii}(\aa_\ii))}$ is $\SSS$-normal. Hence $\seqqq{\der(\phi^{-\nn-\pp}(\aa_\pp))}\etc{\der(\phi^{-\nn-1}(\aa_1))}$ is $\SSS$-normal. 

Moreover, the hypothesis that $\aa_1$ is not $\Delta$-like implies that $\phi^{-\nn-1}(\aa_1)$ is not $\Delta$-like either, and, therefore, $\der(\phi^{-\nn-1}(\aa_1))$ is not invertible. On the other hand, the hypothesis that $\aa_\pp$ is not invertible implies that $\phi^{-\nn-\pp}(\aa_\pp)$ is not invertible either, and, therefore, $\der(\phi^{-\nn-\pp}(\aa_\pp))$ is not $\Delta$-like. Hence, we conclude that the path $\seqqqq{\DELTAA{-\nn - \pp}}{\der(\phi^{-\nn-\pp}(\aa_\pp))}\etc{\der(\phi^{-\nn-1}(\aa_1))}$ is $\Delta$-normal.

It remains to check that the latter $\Delta$-normal path is a decomposition of~$\gg\inv$. Now, by construction, the maps $\phi$ and~$\der$ commute and, pushing the factor~$\DELTAA\nn$ to the left, we obtain the equality
\begin{multline}
\label{E:DEInverse}
\quad\DELTAA{-\nn - \pp}(\ud) \, \der(\phi^{-\nn-\pp}(\aa_\pp)) 
\pdots\der(\phi^{-\nn-1}(\aa_1))\cdot \DELTAA\nn(\ud) \, \aa_1 \pdots\aa_\pp\\
= \DELTAA{-\nn - \pp}(\ud) \, \DELTAA\nn(\ud) \, \der(\phi^{-\pp}(\aa_\pp)) 
\pdots\der(\phi^{-1}(\aa_1)) \, \aa_1 \pdots\aa_\pp,\quad
\end{multline}
and our goal is to prove that this element is~$\id\yy$. Let us call the right term of~\eqref{E:DEInverse} $E(\nn, \aa_1\wdots \aa_\pp)$. We shall prove using induction on~$\pp \ge 0$ that an expression of the form $E(\nn, \aa_1\wdots \aa_\pp)$ equals~$\id\yy$, where $\yy$ is the target of~$\aa_\pp$.

Assume first $\pp = 0$. What remains for $E(\nn)$ is then $\DELTAA{-\nn}(\ud) \, \DELTAA\nn(\ud)$, which, by the remark following Notation~\ref{N:PowerOfDelta}, is~$\id\yy$. Assume $\pp \ge 1$. The above result enables us to gather the first two entries of~$\EE$ into $\DELTAA{-\pp}(\ud)$. Then $\der(\phi\inv(\aa_1)) \, \aa_1$, which is also $\derL(\aa_1) \, \aa_1$, equals~$\Delta(\zz)$, where $\zz$ is the source of~$\derL(\aa_1)$. But, then, we can push this $\Delta(\ud)$-factor to the left through the $\der(\phi^{-\ii}(\aa_\ii))$ factors with the effect of diminishing the exponents of~$\phi$ by one. In this way, $E(\nn, \aa_1\wdots \aa_\pp)$ becomes
$$\DELTAA{-\pp}(\ud) \, \Delta(\ud) \, \der(\phi^{-\pp+1}(\aa_\pp)) \pdots\der(\phi^{-1}(\aa_2)) \, \aa_2 \pdots\aa_\pp,$$
which is $E(1, \aa_2\wdots \aa_\pp)$. By the induction hypothesis, its value is~$\id\yy$.
\end{proof}

\subsection{Computing upper and lower bounds}
\label{SS:Bounds}

If $\CCC$ is an Ore category, the left-divisibility relation of~$\CCC$ naturally extends into a relation on~$\Env\CCC$, namely the relation, still denoted by~$\dive$, such that $\ff \dive \gg$ holds if there exists~$\hh$ in~$\CCC$ satisfying $\ff \hh = \gg$. This relation is a partial preordering on~$\Env\CCC$, and it is a partial ordering if $\CCC$ admits no nontrivial invertible element. It directly follows from the definition that, if $\ff, \gg$ are elements of~$\CCC$, then a least common upper bound (\resp a greatest common lower bound) of~$\ff$ and~$\gg$ with respect to~$\dive$ is (when it exists) a right-lcm (\resp a left-gcd) of~$\ff$ and~$\gg$.

\begin{lemm}
Assume that $\CCC$ is an Ore category. Then any two elements of~$\Env\CCC$ with the same source admit a least common upper bound (\resp a greatest common lower bound) with respect to~$\dive$ if and only if any two elements of~$\CCC$ with the same source admit a right-lcm (\resp a left-gcd). 
\end{lemm}

We skip the easy proof. Then a natural computation problem arises whenever the above bounds exist. The solution is easy.

\begin{algo}{Least common upper bound with respect to~$\dive$}
\label{A:Least}
\begin{algorithmic}[1]
\CONTEXT{A strong Garside family~$\SSS$ in an Ore category~$\CCC$ that admits unique left- and right-lcms}
\INPUT{Two signed $\SSS$-paths $\ww_1, \ww_2$ with the same source}
\OUTPUT{A signed $\SSS$-path representing the least common upper bound of~$\cl{\ww_1}$ and~$\cl{\ww_2}$}
\STATE{$\RC$:= the right-lcm selector on~$\SSS \cup \Id\CCC$}
\STATE{$\LC$:= the left-lcm witness on~$\SSS \cup \Id\CCC$}
\STATE{left-$\LC$-reverse~$\INV{\ww_1} \pc \ww_2$ into a negative--positive path~$\INV\uu \pc \vv$ using (the left counterpart of) Algorithm~\ref{A:RightRev}}\label{A:Greatest:rewrite}
\STATE{right-$\RC$-reverse~$\INV\uu \pc \vv$ into a positive--negative path~$\vv'\pc \INV{\uu'}$ using Algorithm~\ref{A:RightRevShort}}\label{A:Greatest:lcm}
\RETURN{$\ww_1 \pc \vv'$}
\end{algorithmic}
\end{algo}

The solution for greatest lower bound is entirely symmetric.

\begin{algo}{Greatest lower bound with respect to~$\dive$}
\label{A:Greatest}
\begin{algorithmic}[1]
\CONTEXT{A strong Garside family~$\SSS$ in an Ore category~$\CCC$ that admits unique left- and right-lcms}
\INPUT{Two signed $\SSS$-paths $\ww_1, \ww_2$ with the same source}
\OUTPUT{A signed $\SSS$-path representing the greatest lower bound of~$\cl{\ww_1}$ and~$\cl{\ww_2}$}
\STATE{$\RC$:= the right-lcm selector on~$\SSS \cup \Id\CCC$}
\STATE{$\LC$:= the left-lcm selector on~$\SSS \cup \Id\CCC$}
\STATE{right-$\RC$-reverse~$\INV{\ww_1} \pc \ww_2$ into a positive--negative path~$\uu \pc \INV{\vv}$ using Algorithm~\ref{A:RightRev}}
\STATE{left-$\LC$-reverse~$\uu \pc \INV{\vv}$ into a negative--positive path~$\INV{\vv'} \pc \uu'$ using (the left counterpart of) Algorithm~\ref{A:RightRevShort}}
\RETURN{$\ww_1 \pc \INV{\vv'}$}
\end{algorithmic}
\end{algo}

\begin{prop}\label{P:Bounds}
Assume that $\SSS$ is a strong Garside family in an Ore category~$\CCC$ that admits unique left- and right-lcms. Then Algorithm~\ref{A:Least} (\resp Algorithm~\ref{A:Greatest}) running on two signed $\SSS$-paths $\ww_1$ and $\ww_2$ of length at most $\ell$ returns the least upper bound (\resp the greatest lower bound) of $\cl{\ww_1}$ and $\cl{\ww_2}$ with respect to~$\dive$ in time $O(\ell^2)$.
\end{prop}

\begin{proof}
First consider Algorithm~\ref{A:Least}. By the left counterpart of Lemma~\ref{L:Terminating}, the left-$\LC$-reversing in line~\ref{A:Greatest:rewrite} takes time $O(\ell^2)$ and produces two $\SSSs$-paths~$\uu, \vv$ of length $O(\ell)$ such that, in~$\Env\CCC$, we have $\cl{\INV\uu \pc \vv} = \cl{\INV{\ww_1} \pc \ww_2}$, hence $\cl{\ww_1} = \gg \cl\uu$ and $\cl{\ww_2} = \gg \cl\vv$ for some $\gg$, namely the common class of the signed paths~$\ww_1 \pc \INV\uu$ and~$\ww_2 \pc \INV\vv$. By Proposition~\ref{P:ComputLcm}, the right-$\RC$-reversing in line~\ref{A:Greatest:lcm} takes time $O(\ell^2)$ and produces two $\SSS$-paths~$\uu', \vv'$ such that $\uu \pc \vv'$ and $\vv \pc \uu'$ both represent the right-lcm of $\cl\uu$ and $\cl\vv$, hence their least upper bound with respect to~$\dive$. As the partial ordering $\dive$ is invariant under left-multiplication, the claim follows.

The claim for Algorithm~\ref{A:Greatest} follows in an analogous way from Lemma~\ref{L:Terminating} and the left counterpart of Proposition~\ref{P:ComputLcm}, noting that, for all $\ff, \gg$ in $\Env\CCC$ and $\hh$ in~$\CCC$, the conditions $\ff \hh=\gg$ and $\ff\inv = \hh \gg\inv$ are equivalent.
\end{proof}

When we apply Algorithms~\ref{A:Least} and~\ref{A:Greatest} to a pair of positive paths, we obtain algorithms that determine right-lcms and left-gcds. In this case, Algorithm~\ref{A:Least} is simply Algorithm~\ref{A:RightRevShort}, since the left-reversing step is trivial: the initial path~$\INV{\ww_1} \pc \ww_2$ is directly negative--positive. Algorithm~\ref{A:Greatest}, however, has no equivalent in Section~\ref{S:Pos}. Its output is a path representing the left-gcd but, in general, it need not be a positive path, although it must be equivalent to a positive path. In order to obtain a positive output, a third reversing step can be added.

\begin{algo}{Left-gcd}
\label{A:Gcd}
\begin{algorithmic}[1]
\CONTEXT{A strong Garside family~$\SSS$ in an Ore category~$\CCC$ that admits unique left- and right-lcms}
\INPUT{Two $\SSS$-paths $\uu, \vv$ with the same source}
\OUTPUT{An $\SSS$-path that represents the left-gcd of~$\cl\uu$ and~$\cl\vv$ in~$\CCC$}
\STATE{$\RC$:= the right-lcm selector on~$\SSS \cup \Id\CCC$}
\STATE{$\LC$:= the left-lcm selector on~$\SSS \cup \Id\CCC$}
\STATE{right-$\RC$-reverse~$\INV\uu \pc \vv$ into a positive--negative path~$\vv' \pc \INV{\uu'}$ using Algorithm~\ref{A:RightRevShort}}
\STATE{left-$\LC$-reverse~$\vv' \pc \INV{\uu'}$ into a negative--positive path~$\INV{\uu''} \pc \vv''$ using (the left counterpart of) Algorithm~\ref{A:RightRevShort}}\label{A:Gcd:gcd}
\STATE{left-$\LC$-reverse~$\uu \pc \INV{\uu''}$ into a positive path~$\ww$ using (the left counterpart of) Algorithm~\ref{A:RightRevShort}}\label{A:Gcd:rewrite}
\RETURN{$\ww$}
\end{algorithmic}
\end{algo}

\begin{prop}
\label{P:Gc}
Assume that $\SSS$ is a strong Garside family in an Ore-category that admits unique left- and right-lcms. Then Algorithm~\ref{A:Gcd} running on two $\SSS$-paths~$\uu, \vv$ of length at most~$\ell$ returns the left-gcd of $\cl{\uu}$ and $\cl{\vv}$ in time $O(\ell^2)$.
\end{prop}

\begin{proof}
It follows from Proposition~\ref{P:Bounds} that, for~$\uu''$ as computed in line~\ref{A:Gcd:gcd}, the path $\uu\pc\INV{\uu''}$ represents the left-gcd of $\cl{\uu}$ and $\cl{\vv}$. As $\uu$ and $\vv$ are both positive, this path $\uu\pc\INV{\uu''}$ must be equivalent to some positive path, say~$\uu_0$. Then the signed path $\uu\pc\INV{\uu''}\pc\INV{\uu_0}$ represents an identity-element in~$\Env\CCC$, so the positive paths~$\uu$ and $\uu_0 \uu''$ are equivalent. Since, by the counterpart of Proposition~\ref{P:Gar2Pres}, left-reversing is complete for the considered presentation, the path $\uu\pc\INV{\uu''}\pc\INV{\uu_0}$ is left-reversible to an empty path, that is, the left-reversing grid constructed from $\uu\pc\INV{\uu''}\pc\INV{\uu_0}$ has empty arrows everywhere on the left and on the top. This implies in particular that the subgrid corresponding to left-reversing~$\uu\pc\INV{\uu''}$ has empty arrows on the left, that is, with our current notation, that the word~$\ww$ computed at line~5 is positive (and equivalent to~$\uu_0$).

By Proposition~\ref{P:Bounds}, the computation time of~$\ww$ lies in~$O(\ell^2)$ since, by the left counterpart of Lemma~\ref{L:TerminatingShort}, the left-$\LC$-reversing in line \ref{A:Gcd:rewrite} takes time $O(\ell^2)$.
\end{proof}

\begin{exam}
Let us consider a last time the monoid~$\BP3$ and the elements represented by $\uu = \tta\pc\ttb\pc \ttb$ and $\vv = \ttb \pc \tta\pc\ttb\pc \ttb$. As seen in Figure~\ref{F:RightRevWord}, right-reversing $\INV\uu \pc \vv$ using the right-lcm selector leads to the positive--negative word $\tta\ttb \pc \INV\tta$. Then left-reversing the latter word using the left-lcm selector leads to the negative--positive word~$\INV\ttb \pc \tta\ttb$. Finally, left-reversing $\tta\pc\ttb\pc\ttb\pc\INV\ttb$ gives the positive word~$\tta\pc \ttb$, and we conclude that the left-gcd of~$\cl\uu$ and~$\cl\vv$ in~$\BP3$ is~$\tta\ttb$. Note that, instead of using $\SSS$-words and the associated lcm selector, we could instead use $\AAA$-words and the initial right-complement associated with the Artin presentation: the uniqueness of the final result guarantees that the successive involves words must be pairwise equivalent.
\end{exam}

\pagebreak


\end{document}